\documentclass[11pt,a4paper]{article}
 \usepackage[latin1]{inputenc}
\usepackage{amsthm,mathrsfs,amsmath,amssymb}
\usepackage{dsfont}
\usepackage{comment}
\usepackage{indentfirst}
\usepackage{cite}
 \usepackage[colorlinks=true]{hyperref}
\hypersetup{urlcolor=blue, citecolor=red}
\numberwithin{equation}{section}
\allowdisplaybreaks
\makeatother
\date{}

\newtheorem{theorem}{Theorem}[section]
\newtheorem{lemma}[theorem]{Lemma}
\newtheorem{corollary}[theorem]{Corollary}

\theoremstyle{definition}
\newtheorem{definition}[theorem]{Definition}
\newtheorem{remark}[theorem]{Remark}

\begin{document}
\baselineskip 15pt \setcounter{page}{1}
\title{  Prescribed mass standing waves for Schr\"{o}dinger-Maxwell equations with combined nonlinearities\footnote{Supported by
National Natural Science Foundation of China (No.~11971393).}}
\author{{Jin-Cai Kang$^a$,\ \  Yong-Yong Li$^b$,\ \  Chun-Lei Tang$^a$\footnote{Corresponding author. E-mail address: tangcl@swu.edu.cn (C.-L. Tang)}}\\
{\small\emph{$^a$School of Mathematics and Statistics, Southwest University, Chongqing, {\rm400715}, China,}}\\
{\small\emph{$^b$College of Mathematics and Statistics, Northwest Normal University, Lanzhou, {\rm730070}, China}}}
\date{}
\maketitle

\baselineskip 15pt
\begin{quote}
  {\bf Abstract:}~In the present paper, we study the following Schr\"{o}dinger-Maxwell equation with
  combined nonlinearities
  \begin{align*}
  \begin{cases}
  \displaystyle - \Delta u+\lambda u+  \left(|x|^{-1}\ast |u|^2\right)u
  =|u|^{p-2}u +\mu|u|^{q-2}u\quad \text{in} \
  \mathbb{ R}^3,\\
  \displaystyle \int_{\mathbb{R}^3}|u|^2dx=a^2,
  \end{cases}
 \end{align*}
  where   $a>0$, $\mu\in \mathbb{R}$, $2<q\leq \frac{10}{3}\leq p<6$ with $q\neq p$, $\ast$   denotes the convolution and $\lambda\in \mathbb{R}$ appears as a Lagrange multiplier.
  Under some mild assumptions on $a$ and $\mu$,  we prove some existence, nonexistence and  multiplicity  of normalized solution to the above equation.
  Moreover, the asymptotic behavior of normalized solutions is verified as $\mu\rightarrow 0$ and $q\rightarrow \frac{10}{3}$,
  and the stability/instability of the corresponding standing waves to the related time-dependent problem is also discussed.

\end{quote}

\section{Introduction and main results}
In this paper, we investigate standing waves with prescribed
mass for a class of  Schr\"{o}dinger-Maxwell equations involving combined power nonlinearities
  \begin{align}\label{kk1}
   i \varphi_t+\Delta \varphi-\left(|x|^{-1}\ast| \varphi|^2\right)\varphi+|\varphi|^{p-2}\varphi+\mu|\varphi|^{q-2}\varphi=0,
  \end{align}
where    $2<q\leq\frac{10}{3}\leq p<6$ with $q\neq p$, $\mu \in \mathbb{R}$, $\varphi(x, t):\mathbb{R}^3\times [0, T]\rightarrow \mathbb{C} $ is the wave function   and  $(|x|^{-1}\ast  |\varphi| ^2)$ is a repulsive nonlocal Coulombic potential.   The case $\mu>0$ is focusing case, while the case $\mu < 0$ is referred to defocusing case. This class of Schr\"{o}dinger type equations with a repulsive nonlocal Coulomb potential is obtained by approximation of the Hartree-Fock equation describing a quantum mechanical system of many
particles, see \cite{2002RMP-BF, 1981CMP-BBL, 1981RMP-L, 1984CMP-L}. As a pioneering work, Tao et al.~in \cite{2007CPDETao} studied nonlinear Schr\"{o}dinger equations with combined nonlinearities, after which such problems have attracted widespread attention, see for e.g.~\cite{2007CPDETao, 2012DIE, 2016JDECheng, 2018JEE, 2017ARMA, 2016RMI,  2013CMP, 2017CV}. The standing wave solution $\varphi(t, x) =e^{i \lambda t}u(x)$ of Eq.~\eqref{kk1} corresponds to a solution of the equation
 \begin{align}\label{kkkk1}
 - \Delta u+\lambda u+ \left(|x|^{-1}\ast |u|^2\right)u
 =|u|^{p-2}u +\mu|u|^{q-2}u\quad  \text{in} \ \mathbb{ R}^3.
 \end{align}
As well known, fixed frequency problem (i.e.~$\lambda\in \mathbb{R}$ is fixed) and prescribed mass problem (i.e.~the L$^2$-mass of solution is prescribed) for Eq.~\eqref{kkkk1} are two focusing variational problems.
The former~was intensively studied in recent years, see for e.g. \cite{2008MJM-A, 2008CCM-AR, 2010AIHP-ADP, 2008JMAA-AP, 2016N-CM, 2006JFA, 2009-NA-CT, 2016-JDE-SM, 2018-NARWA-ZT, 2023-JGA-KLT, 2016-AMPA-LWZ}. However, the prescribed mass problems of Eq.~\eqref{kkkk1} are seldom  studied, they play a significant role in studying the Bose-Einstein condensation, which motivates us to search for solution with prescribed $L^2$-norm (called as normalized solution).
For this case, $\lambda$ is no longer a given constant but appears as an unknown Lagrange multiplier.

If $\mu=0$,  Eq.~\eqref{kkkk1} reduces to the following Schr\"{o}dinger-Maxwell equation with zero perturbation
   \begin{align}\label{k93}
- \Delta u+\lambda u+  \left(|x|^{-1}\ast |u|^2\right)u=|u|^{p-2}u \quad \text{in} \ \mathbb{ R}^3.
 \end{align}
To search for normalized solutions of Eq.~\eqref{k93}, we consider the critical points of the energy
functional
\begin{align*}
J (u)=\frac{1}{2}\int_{\mathbb{R}^3}|\nabla u|^2dx+\frac{1}{4}\int_{\mathbb{R}^3} \int_{\mathbb{R}^3} \frac{|u(x)|^{2}|u(y)|^2} {|x-y|}dxdy-\frac{1}{p}\int_{\mathbb{R}^3}|u|^{p}dx
\end{align*}
on the constraint
$$
S_a:=\left\{u\in H^1(\mathbb{R}^3, \mathbb{C}): \int_{\mathbb{R}^3}|u|^2dx=a^2 \right\}.
$$
It is easy to show that $J$ is a well defined $C^1$ functional on $S_a$ for any $p\in(2,6]$ (see \cite{2006JFA} for example).
When $p\in(2,\frac{10}{3})$, the functional $J$ is bounded from below and coercive on $S_a$, and the existence of minimizer for $J$ constrained on $S_a$ has been studied in \cite{2011ZAMP-BS, 1992CPDE-CL, 2011JFA-BS, 2004JSP}.  If $p\in (2, 3)$, the authors in  \cite{2011JFA-BS} proved that minimizers exist provided that $a > 0$ is small enough.  Specifically, by using the techniques introduced in \cite{1992CPDE-CL}, it has been proved in \cite{2004JSP} that minimizers exist for $p=\frac{8}{3}$ provided $a \in (0, a_0 )$ for some suitable $a_0 > 0$.
The case $p\in(3,\frac{10}{3})$ is considered in \cite{2011ZAMP-BS},  in which a minimizer is obtained for $a > 0$ large enough. Based on whether $J_{|_{S_a}}$ is bounded from below, the threshold $\frac{10}3$ is called as the L$^2$-critical exponent.
If $p\in[3, \frac{10}{3}]$,   Jeanjean and  Luo in \cite{2013ZAMP-JL} gave a threshold value of $a > 0$ separating existence and nonexistence of minimizers.
When $p\in( \frac{10}{3}, 6)$,  the functional $J$ is unbounded from below on $S_a$, and
Bellazzini, Jeanjean and Luo in  \cite{2013PLMS} found critical points of $J$ on $S_a$ for $a> 0$ sufficiently small.
After then, Chen et al.~\cite{2020JMAA-CT} considered normalized solutions for the Schr\"{o}dinger-Maxwell equations with general nonlinearity
\begin{align*}
- \Delta u+\lambda u+  \left(|x|^{-1}\ast |u|^2\right)u=f(u) \quad \text{in} \ \mathbb{ R}^3,
 \end{align*}
where $f\in C(\mathbb{R}, \mathbb{R})$ covers the case $f(u) = |u|^{s-2}u$ with $q\in (3, \frac{10}{3})\cup ( \frac{10}{3}, 6)$.
Afterwards, the most recent works on Eq.~\eqref{k93} was made in \cite{2023IJM},
where the authors not only demonstrate the old results in a unified way but also exhibit new ones.
For more investigations on normalized solutions of Schr\"{o}dinger-Maxwell equations, we refer readers to   \cite{2013ZAMP-JL, 2017JMAA-ZZ, 2020MMAS, 2023.3, 2014JMAA-L, 2018ZAMP-YL, 2017CMA-Y, 2023AMP-WQ} and the references therein.

Recently, a new phenomenon for the concave-convex variational problem reappears in the study of normalized solutions. Consider the following Schr\"{o}dinger equation with combining  nonlinearities
\begin{align}\label{k110}
  \begin{cases}
  \displaystyle - \Delta u+ \lambda u=|u|^{p-2}u+\mu |u|^{q-2}u \quad \text{in} \ \mathbb{ R}^N, \\
  \displaystyle \int_{\mathbb{R}^N}u^2dx=a^2,
  \end{cases}
  \end{align}
where $N\geq 1$, $a>0$, $\mu\in\mathbb{R}$, $\lambda\in\mathbb{R}$ is an unknown parameter and $2<q \leq p< 2^*$. If $q<2+\frac4N<p$,~the fiber mapping under the scaling of type $t^{\frac N2}u(t\cdot)$ admits
two critical points (a maximum point and a minimum point), which indicates that two normalized solutions can be catched by splitting the corresponding Nehari-Poho$\check{\mbox{z}}$aev manifold to Eq.~\eqref{k110} into two submanifolds. As mentioned above, Soave in \cite{2020Soave} firstly studied the existence and nonexistence of normalized solution for Eq.~\eqref{k110} by using Nehari-Poho\v{z}aev constraint.
After then, as an extension of works in \cite{2020Soave}, Soave in \cite{2021-JFA-Soave} further considered the existence and nonexistence of normalized solutions for Eq.~\eqref{k110} with $p=2^*:=\frac{2N}{N-2}$ and $ q\in(2,2^*)$.
For more results of  Eq.~\eqref{k110},  we refer the readers to  \cite{2022-MA-jean, 2022JFAWei,2022.9, 2021CVLixinfu, 2022CValves}.

When it comes to Schr\"{o}dinger-Maxwell equations with combining nonlinearities, there seem to be no results to our knowledge. Motivated by the above survey, a natural question is whether Eq.~\eqref{kkkk1} possesses normalized solution for $\mu\in \mathbb{R}$ and $2<q\leq\frac{10}{3}\leq p<6$ with $q\neq p$. Naturally, we study the following Schr\"{o}dinger-Maxwell equation in this work
  \begin{align}\label{k1}\tag{${ \mbox{SP}}$}
  \begin{cases}
  \displaystyle - \Delta u+\lambda u+  \left(|x|^{-1}\ast |u|^2\right)u
   =|u|^{p-2}u +\mu|u|^{q-2}u\quad \text{in} \ \mathbb{ R}^3,\\
  \displaystyle \int_{\mathbb{R}^3}|u|^2dx=a^2,
  \end{cases}
 \end{align}
where $a>0$, $\mu\in \mathbb{R}$,   $2<q\leq\frac{10}{3}\leq p<6$ with $q\neq p$,  $\ast$   denotes the convolution and $\lambda\in \mathbb{R}$ is an unknown Lagrange multiplier.~Normalized solutions of Eq.~\eqref{k1} correspond to critical points of the following functional on the constraint $S_a$,
  \begin{align*}
\mathcal{J} (u)=\frac{1}{2}\int_{\mathbb{R}^3}|\nabla u|^2dx+\frac{1}{4}\int_{\mathbb{R}^3}\int_{\mathbb{R}^3} \frac{|u(x)|^{2}|u(y)|^2} {|x-y|}dxdy
-\frac{1}{p}\int_{\mathbb{R}^3}|u|^{p}dx-\frac{\mu}{q}\int_{\mathbb{R}^3}|u|^{q}dx.
\end{align*}

For convenience, we let $\mathcal{H}:=H^1(\mathbb{R}^3, \mathbb{C})  $ denotes the usual Sobolev space with the norm
$$
\|u\|= \Big(\int_{\mathbb{R}^3}  |\nabla u|^2+  |u|^2 dx\Big)^{\frac{1}{2}}.
$$
$\mathcal{H }_r :=\{u\in \mathcal{H}: u(x)=u(|x|)\}$ has the same scalar product and norm as in $\mathcal{H}$.\\

Before stating our main results, we recall the following

\begin{definition}\label{de1.1}
We say that $\overline{u}$ is a ground state normalized solution of Eq.~\eqref{k1} on $S_a$ if $\overline{u}$ is a solution of Eq.~\eqref{k1} possessing minimal energy among all of the normalized solutions, namely,
\begin{align*}
  d\mathcal{J}_{|_{S_a }}(\overline{u})=0 \quad
  \text{and}
  \quad \mathcal{J}(\overline{u})=\inf\left\{\mathcal{J}(u):   d\mathcal{J}_{|_{S_a }}(u)=0 \ \text{and}\ u\in S_a\right \}.
\end{align*}
The set of all ground state normalized solutions for Eq.~\eqref{k1} will be denoted as $\mathcal{M}_a$.
\end{definition}

The case of $\mu>0$ and $\frac{10}{3}<q <p< 6$  has already been considered in \cite{2020JMAA-CT}, and the  existence result of Eq.~\eqref{k1} in this case is presented here for our convenience.

\begin{theorem}\label{K-T1}\cite[Theorem 1.1]{2020JMAA-CT}
Let $\mu>0$ and $\frac{10}{3}<q <p< 6$. Assume that  $a\in (0, \widetilde{\kappa})$ small enough, then there exists a  mountain pass type  solution   $(\lambda_{q}, u_{q})\in  \mathbb{R}^+\times \mathcal{H}_r$
 for Eq.~\eqref{k1} with positive energy  and $u_{q}>0$, where $J_q(u_q)= c_{q,r}>0$.
\end{theorem}

Our main results in this paper are stated as follows:

\begin{theorem}\label{K-TH1}
Assume $\mu>0$, $q\in(2, \frac{8}{3} )$ and $p\in (\frac{10}{3}, 6)$. Let $a\in(0, \overline{a}_0)$ with
\begin{align*}
 \overline{a}_0:=\min\left\{a_0, \ \bigg[ \frac{p(2-q\gamma_q)}{2C_{p}^p (p\gamma_p-q\gamma_q)}\bigg]^\frac{2-q\gamma_q}{\mathcal{B} }
 \bigg[ \frac{q(p\gamma_p-2)}{2\mu C_{q}^q (p\gamma_p-q\gamma_q)}\bigg]^\frac{p\gamma_p-2}{ \mathcal{B}}\right\},
 \end{align*}
where $a_0$ will be defined in \eqref{k50} hereafter, $C_{p} $ is defined in \eqref{k102}, $\gamma_p=\frac{3(p-2)}{2p}$, $\gamma_q=\frac{3(q-2)}{2q}$ and
$$
\mathcal{B}=(q-q\gamma_q)(p\gamma_p-2)+(p-p\gamma_p)(2-q\gamma_q).
$$
Then Eq.~\eqref{k1}  possesses  a real-valued ground state normalized solution $(\lambda, u)\in  \mathbb{R}\times  \mathcal{H}$ with $u>0$.
Moreover, this ground state is a local minimizer of $\mathcal{J}$ in the set $D_{\rho_0}$ and any ground state normalized solution is a local minimizer of $\mathcal{J}$ on $D_{\rho_0}$, where $D_{\rho_0}$ is defined by \eqref{k91}.
\end{theorem}

\begin{theorem}\label{K-TH2}
Assume $\mu>0$, $q\in(2, \frac{12}{5} ]$ and $p\in (\frac{10}{3}, 6)$. Let $a\in(0, \overline{a}_0)$ and
 \begin{align}\label{k3}
 a< \left(\frac{3}{4C_p^p}\right)^{\frac{1}{2p\gamma_p+p-6}}
 \left(\frac{4}{4\gamma_p-1 }\right)^{\frac{p\gamma_p-1}{2p\gamma_p+p-6}}\Bigg(\frac{1- \gamma_p }{C_{\frac{12}{5}}^{\frac{12}{5}}}\Bigg)^{\frac{p\gamma_p-2 }{2p\gamma_p+p-6}}.
 \end{align}
  Then there exists a second solution  of mountain pass type $( \widehat{\lambda}, \widehat{u})\in \mathbb{R}^+ \times \mathcal{H}_r $   for Eq.~\eqref{k1} with $ \widehat{u}>0$.
\end{theorem}

Here, we want to highlight the work \cite{2020Soave}, in which Soave made a pioneering work for Eq.~\eqref{k110} with $\mu\in \mathbb{R}$ and $2<q\leq 2+\frac{4}{N}\leq p<2^*$ by using Nehari-Poho\v{z}aev constraint.
In particular, when $2<q< 2+\frac{4}{N}< p<2^*$ and $\mu>0$, since the energy functional $\mathcal{E}$ for
Eq.~\eqref{k110} is no longer bounded from below on $S_a$, he considered the minimization of $\mathcal{E}$ on a subsets of $S_a$ and obtained two normalized solutions marked as $\overline{v}$ and $\overline{u}$, where $\overline{v}$ is an interior local minimizer of $\mathcal{E}$ on the set $\left\{u\in S_a:|\nabla u|_2<k \right\}$ for $k > 0$ small enough and $\overline{u}$ is of mountain pass type.
To search for a local minimizer, with the help of Nehari-Poho\v{z}aev manifold,
a bounded Palais-Smale sequence $\{u_n\}\subset S_a\cap H^1(\mathbb{R}^N)$ is received by using standard argument.
Using the Schwarz rearrangement, he obtained a bounded Palais-Smale sequence $\{u_n\}\subset S_a\cap H_r^1(\mathbb{R}^N)$.
For the Mountain-Pass type solution, he firstly established the mountain-pass geometry of $\mathcal{E}$ on $S_a\cap H_r^1(\mathbb{R}^N)$,
and then constructed a special bounded Palais-Smale sequence $\{u_n\}\subset S_a\cap H_r^1(\mathbb{R}^N)$ at the mountain pass level.
Since $H_r^1(\mathbb{R}^N)\hookrightarrow L^s(\mathbb{R}^N )$ with $s\in(2, 2^*)$ is compact,
Soave proved that the associated Lagrange multiplier $\lambda >0$.
In view of this fact, the compactness of bounded Palais-Smale sequence can be proved.
It is worth mentioning that the approaches in \cite{2020Soave} rely heavily on the compactness of $H_r^1(\mathbb{R}^N)\hookrightarrow L^s(\mathbb{R}^N )$ with $s\in(2, 2^*)$. We mention that Soave's methods in \cite{2020Soave} cannot be directly applied to prove our results due to two reasons.
Firstly, the Schwarz rearrangements are invalid in our present paper due to the presence of nonlocal term.
Secondly,  although our workspace is $\mathcal{H}_r$ in Theorem \ref{K-TH1}, we can only obtain a local minimizer to Eq.~\eqref{k1} for $q\in(2, \frac{12}{5} ]$ and $p\in (\frac{10}{3}, 6)$ by using Soave's methods in \cite{2020Soave}  directly.
Therefore, we will follow some ideas of \cite{2022-JMPA-jean, 2011JFA-BS} to fill the gap between $q\in(\frac{12}{5}, \frac{8}{3} )$ and $p\in (\frac{10}{3}, 6)$ to prove Theorem~\ref{K-TH1}.

\begin{theorem}\label{K-TH5}
Assume $\mu\leq0$, $q\in(2, \frac{8}{3} )$ and $p\in (\frac{10}{3}, 6)$. Then there exists some $\widetilde{a}_0>0$ such that Eq.~\eqref{k1} has a normalized solution $(\widetilde{\lambda}, \widetilde{u})\in\mathbb{R}^+ \times \mathcal{H}_r$ for any $a\in(0, \widetilde{a}_0)$.
\end{theorem}

\begin{theorem}\label{K-TH7}
Assume $q\in(2, \frac{8}{3} )$ and $p=\overline{p}=\frac{10}{3}$, then
\begin{enumerate}
  \setlength{\itemsep}{-0.50pt}
  \item [(1)] for any $\mu>0$, if $0<a\leq a^*:=\big(\frac{\overline{p}}{2 C_{\overline{p}}^{\overline{p}}}\big)^{\frac{3}{4}}$, it results that
       $e(a)=\inf_{S_a} \mathcal{J}<0$ and when $0<a< a^* $ the infimum is achieved by a real-valued  function $u\in S_a $;
       moreover, when $a=a^*$, $q\in(2, \frac{12}{5} ]$ and $p=\frac{10}{3}$, the infimum $e(a^*)$ is achieved by a real-valued  function $u\in S_a $;
  \item [(2)] for any $\mu>0$, if $a> a^*$, it holds that $\inf_{S_a} \mathcal{J}=-\infty$;
  \item [(3)] for any $\mu<0$, if $0<a\leq a^*:=\big(\frac{\overline{p}}{2 C_{\overline{p}}^{\overline{p}}}\big)^{\frac{3}{4}}$, it results that
      $\inf_{S_a} \mathcal{J}=0$ and Eq.~\eqref{k1} has no solution;
  \item [(4)] for any $\mu<0$, if $a> a^*$, it holds that $\inf_{S_a} \mathcal{J}=-\infty$.
\end{enumerate}
\end{theorem}

\begin{theorem}\label{K-TH15}
 Let $\mu>0$, $q=\overline{q}=\frac{10}{3}$ and   $p\in (\frac{10}{3}, 6)$. Assume that
\begin{align}\label{k201}
\bigg(\frac{1}{C_p^p \gamma_p a^{p(1-\gamma_p)}} \bigg)^{\frac{1}{p \gamma_p-2}} \Big(1-\frac{3}{5}  C_{\overline{q}}^{\overline{q}}\mu a^{\frac{4}{3}}\Big)^{\frac{1}{p \gamma_p-2}}\frac{1}{a^3}\geq \frac{4\gamma_p-1}{4(1-\gamma_p)}C_{\frac{12}{5}}^{\frac{12}{5}}.
\end{align}
Then there exists a  mountain pass type  solution   $(\lambda_{\overline{q}}, u_{\overline{q}})\in  \mathbb{R}^+\times \mathcal{H}_r$
 for Eq.~\eqref{k1}, where  $u_{\overline{q}} > 0$.
\end{theorem}
\begin{remark}  The assumption
\begin{align*}
\bigg(\frac{1}{C_p^p \gamma_p a^{p(1-\gamma_p)}} \bigg)^{\frac{1}{p \gamma_p-2}} \Big(1-\frac{3}{5}  C_{\overline{q}}^{\overline{q}}\mu a^{\frac{4}{3}}\Big)^{\frac{1}{p \gamma_p-2}}\frac{1}{a^3}\geq \frac{4\gamma_p-1}{4(1-\gamma_p)}C_{\frac{12}{5}}^{\frac{12}{5}}
\end{align*}
implies the condition $\mu a^{\frac{4}{3}}< \frac{\overline{q}}{2C_{\overline{q}}^{\overline{q}}}$.
\end{remark}

\begin{remark}
To our knowledge, the latest progress in the study of normalized solution to Schr\"{o}dinger-Maxwell equations owes to \cite{2023IJM},
which presents the existence of normalized solution for Schr\"{o}dinger-Maxwell equation with single power nonlinearity.
There seem to be no existence results of normalized solutions for Schr\"{o}dinger-Maxwell equations with combined power nonlinearities.
Our results can be seen as  an extension and improvement   of existing results.
Moreover, compared with \cite{2020Soave}, our problems are more complex because of the presence of nonlocal term,
and the method of Schwarz rearrangement is invalid here.
If we verify Theorem \ref{K-TH1} in a radial space, we can only obtain normalized solution of Eq.~\eqref{k1} for $\mu>0$, $q\in(2, \frac{12}{5} ]$ and $p\in (\frac{10}{3}, 6)$ by using  the same   technique  as \cite{2020Soave}, but we cannot obtain normalized solution in the case of $q\in(  \frac{12}{5},  \frac{8}{3} )$ and $p\in (\frac{10}{3}, 6)$. The difficulty lies in that the compactness of Palais-Smale sequence cannot be easily verified when $q\in(  \frac{12}{5},  \frac{8}{3} )$ and $p\in (\frac{10}{3}, 6)$, even in radial subspaces.
To avoid dichotomy in Theorem \ref{K-TH1}, we follow the ideas of \cite{2011JFA-BS} to prove the strong subadditivity inequality. Most importantly, the authors in \cite{2011JFA-BS} studied Eq. \eqref{k93} with $ p\in(2, 3)$,
in which the corresponding energy functional is
bounded from below on the constraint $S_a$.
However,   the energy functional $\mathcal{J}$ under the assumptions of Theorem \ref{K-TH1}    is unbounded from below on $S_a$.
In order to overcome this obstacle, we consider the minimization of $\mathcal{J}$ on a subsets $S_a\cap D_{\rho_0}$, where $D_{\rho_0}$ is defined in  \eqref{k91}. Noting that $ \rho_0$   does not depend on $a$, which is the key to applying the   methods of   \cite{2011JFA-BS}.
Last but not least, for the case of $q\in(2, \frac{8}{3} )$, $p=\overline{p}=\frac{10}{3}$, $\mu>0$ and $a=a^*$,  we deduce $-\infty<e(a)=\inf_{S_a} \mathcal{J}<0$, and then  $e(a^*)$ is achieved when $\mu>0$, $q\in(2, \frac{12}{5} ]$ and $p=\frac{10}{3}$, which is different from \cite[Theorem 1.1]{2020Soave}.
\end{remark}

\begin{remark}
Since we cannot avoid dichotomy  for the range $a=a^*$, $\mu>0$, $q\in(\frac{12}{5}, \frac{8}{3} )$ and $p=\frac{10}{3}$, we say nothing about the existence of normalized solution to Eq.~\eqref{k1}, so does for the case of $\mu>0$, $q\in(  \frac{12}{5},   \frac{8}{3}) $ and $p\in (\frac{10}{3}, 6)$ in Theorem \ref{K-TH2}.
 Furthermore, the existence/nonexistence of  normalized solution to Eq.~\eqref{k1} with   $\mu\in \mathbb{R}\backslash\{0\}$, $q\in(  \frac{8}{3}, \frac{10}{3} )$ and $p\in[\frac{10}{3}, 6)$ remains an unsolved open problem.
\end{remark}

Next, we verify the asymptotic behavior of normalized solutions as $\mu\rightarrow 0$ and $q\rightarrow \frac{10}{3}$.

\begin{theorem}\label{K-TH6}
Assume that $\mu>0$, $q\in(2, \frac{12}{5} ]$ and $p\in (\frac{10}{3}, 6)$,
further $a\in(0, \min \{\overline{a}_0, \widetilde{a}_0\})$ and \eqref{k3} hold.
Let $ (\widehat{\lambda}_{\mu}, \widehat{u}_{\mu})\in \mathbb{R}^+ \times \mathcal{H}_r $ be the mountain pass type solution for Eq.~\eqref{k1} obtained in Theorem \ref{K-TH2},
then $\widehat{u}_{\mu}\rightarrow \widehat{u}_{0}$ in  $\mathcal{H}_{r}$ and $\widehat{\lambda}_{\mu}\rightarrow \widehat{\lambda}_{0}$ in $\mathbb{R} $  as $\mu\rightarrow0$ in the sense of subsequence, where $(\widehat{\lambda}_{0}, \widehat{u}_{0})\in \mathbb{R}^+ \times \mathcal{H}_r  $  is a solution of
\begin{align*}
\begin{cases}
\displaystyle - \Delta u+\lambda u+  \big(|x|^{-1}\ast |u|^2\big)u=|u|^{p-2}u  \ \  \mbox{in} \ \mathbb{ R}^3,\\
\displaystyle\int_{\mathbb{R}^3}|u|^2dx=a^2.\\
\end{cases}
 \end{align*}
\end{theorem}

\begin{theorem}\label{K-TH14}
 Assume that $\mu>0$,  $\frac{10}{3}=\overline{q}<q<p<6$ and $a\in (0, \widetilde{\kappa})$ satisfying \eqref{k201}.   Let $(\lambda_q, u_q) \in \mathbb{R}^+ \times \mathcal{H}_r $  be the radial mountain pass type normalized solution of    Eq.~\eqref{k1} with $q>\overline{q}$ tending to $\overline{q}$
at the level  $c_{q,r}$ obtained in Theorem \ref{K-T1}.
Then, there is $(\lambda, u)\in \mathbb{R}^+ \times \mathcal{H}_r $ such that up to a subsequence,  $u_q\rightarrow u$ in $\mathcal{H}$ and $ \lambda_q\rightarrow \lambda$ in $\mathbb{R} $ as $q\rightarrow \overline{q}$,
where   $\mathcal{J}(u)= c_{\overline{q}, r}$ and  $(\lambda, u)\in \mathbb{R}^+ \times \mathcal{H}_r$ solves Eq.~\eqref{k1} with $q=\overline{q}$.
\end{theorem}

In what follows, we recall the notions of orbital stability and instability.

\begin{definition}\label{de1.3}
$Z\subset \mathcal{H} $ is stable if $Z\neq \emptyset$ and, for any $v \in Z$ and $\varepsilon>0$, there exists a $ \delta > 0$ such that
if $\varphi\in \mathcal{H}$  satisfies $\|\varphi - v\|<\delta$,  then $u_{\varphi}(t)$ is globally defined and $\inf_{z\in Z}\|u_{\varphi}(t)-z\| <\varepsilon $ for all $t \in\mathbb{ R}$, where $u_{\varphi}(t)$  is the solution to Eq~\eqref{kk1} with the initial condition $\varphi$.
\end{definition}

\begin{definition}\label{de1.2}
A standing wave $e^{i\lambda t}u$ is strongly unstable if, for every $\varepsilon> 0$, there exists $\varphi_0 \in \mathcal{H}$ such
that $\|u- \varphi_0\|<\varepsilon$ and $\varphi(t,\cdot)$ blows-up in finite time, where $\varphi(t,\cdot)$ denotes the solution to Eq~\eqref{kk1} with initial datum  $\varphi_0$.
\end{definition}

From Theorem \ref{K-TH1},   the set of all ground state normalized solutions for Eq.~\eqref{k1} is given by
\begin{align*}
\mathcal{M}_a=\Big\{u\in  \mathcal{H}: u\in D_{\rho_0}\  \mbox{and}\ \mathcal{J}(u)=m_a:= \inf_{D_{\rho_0} }\mathcal{J}(u)\Big\},
\end{align*}
where $D_{\rho_0} $ is defined in \eqref{k91}.
Next, we shall focus on the   stability of the ground state set $\mathcal{M}_a$.

\begin{theorem}\label{K-TH3}
Under the assumptions of Theorem \ref{K-TH1}, $\mathcal{M}_a$ has the following characterization
\begin{align*}
\mathcal{M}_a=\big\{ e^{i \theta}|u|:  \ \theta\in \mathbb{R}, \ |u|\in D_{\rho_0},\ \mathcal{J}(|u|)=m_a\ \text{and}\   |u| >0 \ \text{in}\ \mathbb{R}^3\big\}.
\end{align*}
Moreover, the set  $\mathcal{M}_a$ is orbitally stable.
\end{theorem}

We are also interested in the instability of standing waves obtained in Theorems \ref{K-TH2},  \ref{K-TH5} and  \ref{K-TH15}.

\begin{theorem}\label{K-TH4}
Under the assumptions of Theorem \ref{K-TH2} (or Theorem \ref{K-TH5} or  Theorem \ref{K-TH15}),  the standing
wave $\psi(t, x) =e^{i  \widehat{\lambda}t} \widehat{u}(x)$ (or $\psi(t, x) =e^{i   \widetilde{\lambda} t} \widetilde{u}(x)$ or  $\psi(t, x) =e^{i  \lambda_{\overline{q}} t} u_{\overline{q}}(x)$) is strongly unstable.
\end{theorem}

\begin{remark}
The proof of Theorem \ref{K-TH3} relies  on the classical Cazenave-Lions' stability argument introduced in \cite{1982-CMP-CL} and further developed in \cite{2004-ANS-Ha}. To prove Theorem \ref{K-TH4}, we shall follow the original approach of Berestycki and Cazenave \cite{1981CRASSM-BC}.
\end{remark}

The organization of this paper is as follows.
In Sec.~\ref{sec2} and Sec.~\ref{sec1}, we give some preliminaries and prove Theorem \ref{K-TH1}, respectively.
Then we prove Theorem \ref{K-TH2} in Sec.~\ref{sec3}.
The aim of Sec.~\ref{sec6} is to show~Theorem \ref{K-TH5}, and Sec.~\ref{sec5} is devoted to show Theorem \ref{K-TH7}. In Sec. \ref{sec7}, we give the proof of Theorem \ref{K-TH15}.
Finally, we show Theorems \ref{K-TH6},  \ref{K-TH14}, \ref{K-TH3} and \ref{K-TH4} in Sec.~\ref{sec4}.

Now, we conclude this section with the following notations applied throughout this paper:
\begin{itemize}
  \setlength{\itemsep}{-1.0pt}
  \item $\mathcal{D}^{1,2}(\mathbb{R}^3):= \mathcal{D}^{1,2}(\mathbb{R}^3,\mathbb{C})$ is the usual Sobolev space with the norm
  $\|u\|_{\mathcal{D}^{1,2}  }=\left(\int_{\mathbb{R}^3} |\nabla u|^2 dx\right)^{\frac12}$.
  \item $L^q(\mathbb{R}^3)= L^q(\mathbb{R}^3, \mathbb{C})$ is the Lebesgue space with the norm $|u|_q=\left(\int_{\mathbb{R}^3}|u|^qdx\right)^{\frac{1}{q}}$ for $q\in[1,+\infty)$.
  \item $\mathcal{H}^{-1}$ (or $\mathcal{H}_r^{-1}$)  denotes the dual space of $\mathcal{H} $ (or $\mathcal{H}_r$) and $\mathbb{R}^+=(0, +\infty)$.
  \item $C$, $C_i,  i=1,2, \ldots $,   denote positive constant which may depend on  $p$ and $q$ (but never on $a$ or $\mu$), and may possibly various in different places.
  \item $B_y(r):=\left\{x\in \mathbb{ R}^3:|x-y|<r\right\}$ for $y\in\mathbb{R}^3$ and $r>0$.
\end{itemize}

We also mention that, within a section, after having fixed the
parameters $a$ and $\mu$ we may choose to omit the dependence of $\mathcal{J}_{a,\mu}$,  $P_{a,\mu}$, $\mathcal{P}_{a,\mu}$,  $\ldots$  on these
quantities, writing simply $\mathcal{J}$,   $P$, $\mathcal{P}$, $\ldots$.

\section{Preliminaries}\label{sec2}
Since $ \mathcal{J}$ is unbounded from below on $S_a$, thus one cannot apply the minimizing argument on $S_a$ any more.
In this paper, we consider the minimization of $\mathcal{J}$ on a subset of $S_a$. For any $u\in S_a$, let
$$
(s\star u)(x):=e^{\frac{3s}{2}}u(e^sx).
$$
By direct computation, one has $(s\star u)\in S_a$ and
 \begin{align}\label{k8}
 \psi_u(s):=\mathcal{J}(s\star u)
  &=\frac{e^{2s}}{2}\int_{\mathbb{R}^3}|\nabla u|^2dx+\frac{e^{s}}{4} \int_{\mathbb{R}^3} \int_{\mathbb{R}^3} \frac{|u(x)|^{2}|u(y)|^2} {|x-y|}dxdy \nonumber\\
  &\quad-\frac{e^{p\gamma_ps}}{p}\int_{\mathbb{R}^3}|u|^{p}dx
  -\frac{\mu e^{q\gamma_qs}}{q}
  \int_{\mathbb{R}^3}|u|^{q}dx,
\end{align}
where $\gamma_p=\frac{3(p-2)}{2p}$. Clearly,
 \begin{align*}
 p \gamma_p
 \begin{cases}
 >2  \ &\text{if}\ p>\frac{10}{3},\\
 =2  \ &\text{if}\ p=\frac{10}{3},\\
 <2  \ &\text{if}\ p<\frac{10}{3}.
 \end{cases}
\end{align*}
It is easy to see that, if $u\in \mathcal{H}$ is a solution of Eq.~\eqref{k1}, then  the following Nehari  identity holds
 \begin{align}\label{kk4}
 \int_{\mathbb{R}^3}|\nabla u|^2+\lambda\int_{\mathbb{R}^3}  |u|^2dx+ \int_{\mathbb{R}^3} \int_{\mathbb{R}^3} \frac{|u(x)|^{2}|u(y)|^2} {|x-y|}dxdy =\int_{\mathbb{R}^3}|u|^{p}dx+\mu  \int_{\mathbb{R}^3}|u|^{q}dx.
\end{align}
Moreover, by  \cite[Proposition 2.1]{2014-JFA-le},  any solution $u$ of Eq.~\eqref{k1}  satisfies the following Poho\v{z}aev identity
 \begin{align}\label{kkk4}
\frac{1}{2}\int_{\mathbb{R}^3}|\nabla u|^2dx +\frac{3\lambda }{2} \int_{\mathbb{R}^3}  |u|^2dx
+\frac{5}{4}\int_{\mathbb{R}^3} \int_{\mathbb{R}^3} \frac{|u(x)|^{2}|u(y)|^2} {|x-y|}dxdy=
\frac{3}{p}\int_{\mathbb{R}^3}|u|^{p}dx+\frac{ 3\mu}{q} \int_{\mathbb{R}^3}|u|^{q}dx.
\end{align}
Hence, by \eqref{kk4} and \eqref{kkk4},  $u$ satisfies
\begin{align}\label{k4}
P(u):=\int_{\mathbb{R}^3}|\nabla u|^2dx+\frac{1}{4}\int_{\mathbb{R}^3} \int_{\mathbb{R}^3} \frac{|u(x)|^{2}|u(y)|^2} {|x-y|}dxdy
-\gamma_p\int_{\mathbb{R}^3}|u|^{p}dx
-\mu\gamma_q\int_{\mathbb{R}^3}|u|^{q}dx=0.
 \end{align}
Usually, \eqref{k4} is called as Nehari-Poho\v{z}aev identity. In particular, \eqref{k4} is widely used in the literature to study the prescribed mass problem.
In order to obtain the  normalized    solutions for Eq.~\eqref{k1}, we introduce the following Nehari-Poho\v{z}aev constrained set:
\begin{align*}
 \mathcal{P}=\left\{u\in S_a: P(u)=0\right\}.
 \end{align*}
 By simple calculation, we have
$$
P(s\star u)=e^{2s}\int_{\mathbb{R}^3}|\nabla u|^2dx+\frac{e^{s}}{4} \int_{\mathbb{R}^3} \int_{\mathbb{R}^3} \frac{|u(x)|^{2}|u(y)|^2} {|x-y|}dxdy
-\gamma_p e^{p\gamma_ps} \int_{\mathbb{R}^3}|u|^{p}dx
-\mu\gamma_q e^{q\gamma_qs}
\int_{\mathbb{R}^3}|u|^{q}dx.
$$
Obviously, $\psi_u'(s)=P(s\star u)$. Hence, for any $u\in S_a$, $s\in \mathbb{R}$ is a critical point of $\psi_u(s)$ iff $s\star u\in \mathcal{P}$.

The following classical Gagliardo-Nirenberg inequality is very important in this paper, which can be found in  \cite{1983W}.
Concretely, let  $2< t <6$, then
\begin{align}\label{k10}
|u|_t\leq C_{t}|u|_2^{1-\gamma_t}|\nabla u|_2^{\gamma_t},\ \ \ \ \forall \ u\in \mathcal{H},
\end{align}
where the sharp constant $C_{t}$ is defined by
\begin{align}\label{k102}
C_{t}^t=\frac{2t}{6-t}\left(\frac{6-t}{3(t-2)}  \right)^{\frac{3(t-2)}{4}}\frac{1}{|Q_t|_2^{t-2}},
 \end{align}
and $Q_t$ is the unique positive radial solution of equation
$$
-\Delta Q+Q=|Q|^{t-2}Q.
$$

Now, we recall some properties of the operator $\int_{\mathbb{R}^3} \frac{|u(y)|^2} {|x-y|}dy$ in the following lemma.

\begin{lemma}\label{gle4}(see \cite{2006JFA})
 The following results hold
\begin{itemize}
  \item [$(1)$]   $\int_{\mathbb{R}^3} \frac{|u(y)|^2} {|x-y|}dy\ge0$ for any $x\in{\mathbb{R}^{3}}$;
  \item [$(2)$] there exist some constants $C_1,C_2>0$ such that $\int_{\mathbb{R}^3} \int_{\mathbb{R}^3} \frac{|u(x)|^{2}|u(y)|^2} {|x-y|}dxdy
      \le{C_1|u|^4_\frac{12}{5}}\le{C_2\|u\|^4};$

  \item [$(3)$] if $u_n\rightarrow u$ in $L^{\frac{12}{5}}(\mathbb{R}^3)$, then $\int_{\mathbb{R}^3} \frac{|u_n(y)|^2} {|x-y|}dy\rightarrow \int_{\mathbb{R}^3} \frac{|u(y)|^2} {|x-y|}dy$ in $\mathcal{D}^{1,2}(\mathbb{R}^{3});$
  \item [$(4)$]    if $u_n\rightharpoonup u$ in $\mathcal{H}$, then $\int_{\mathbb{R}^3} \frac{|u_n(y)|^2} {|x-y|}dy\rightharpoonup \int_{\mathbb{R}^3} \frac{|u(y)|^2} {|x-y|}dy$ in $\mathcal{D}^{1,2}(\mathbb{R}^{3})$.
\end{itemize}
\end{lemma}

Define for $(a, t)\in \mathbb{R}^+\times \mathbb{R}^+$ the function
\begin{align*}
f(a,t)=\frac{1}{2}-\mu \frac{C_q^q}{q}a^{(1-\gamma_q)q}t^{\gamma_qq-2}
 -\frac{C_p^p}{p}a^{(1-\gamma_p)p}t^{p\gamma_p-2}.
\end{align*}

\begin{lemma}\label{KK-Lem2.1}
Assume that $\mu>0$, $q\in(2, \frac{8}{3})$ and $p\in (\frac{10}{3}, 6)$.
 Then for any fixed $a>0$,  the function $f_a(t):=f(a, t)$ has a unique global maximum, and the maximum value satisfies
 \begin{align}\label{0k50}
 \max\limits_{t>0}f_a(t)
 \begin{cases}
 >0 &\text{if}\ a<a_0,\\
 =0 &\text{if}\ a=a_0,\\
 <0 &\text{if}\ a>a_0,
 \end{cases}
 \end{align}
where
\begin{align}\label{k50}
a_0=\left(\frac{1}{2K}\right)^{\frac{p\gamma_p-q\gamma_q}{(p-p\gamma_p)(2-q\gamma_q)+(q-q\gamma_q)(p\gamma_p-2)}}
\end{align}
with
\begin{align*}
K=\left(\frac{\mu C_q^q}{q}\right)^{\frac{p\gamma_p-2 }{p\gamma_p-q\gamma_q }}\left[\frac{p(2-q\gamma_q)}{(p\gamma_p-2)C_p^p} \right]^{\frac{q\gamma_q-2 }{p\gamma_p-q\gamma_q}}+ \left(\frac{p}{C_p^p} \right)^{\frac{q\gamma_q-2 }{p\gamma_p-q\gamma_q}} \left[\frac{\mu(2-q\gamma_q)C_q^q}{q(p\gamma_p-2) } \right]^{\frac{p\gamma_p-2 }{p\gamma_p-q\gamma_q}}.
\end{align*}
\end{lemma}
\begin{proof}
 By definition of $f_a(t)$, we have
  $$
f_a'(t)= \mu (2-q\gamma_q)\frac{C_q^q}{q}a^{(1-\gamma_q)q}t^{\gamma_qq-3}
 -(p\gamma_p-2)\frac{C_p^p}{p}a^{(1-\gamma_p)p}t^{p\gamma_p-3}.
 $$
 We know that  the equation $f_a'(t)=0$  has a unique solution
   \begin{align}\label{k90}
 \rho_a=\left[\frac{\mu p(2-q\gamma_q)C_q^q}{q(p\gamma_p-2)C_p^p}a^{q(1-\gamma_q)-p(1-\gamma_p)} \right]^{\frac{1}{p\gamma_p-q\gamma_q}}.
   \end{align}
By simple analysis, we obtain that $f_a(t)$ is increasing on $(0, \rho_a )$ and decreasing on $( \rho_a, +\infty )$,
and $f_a(t)\rightarrow-\infty $ as  $t\rightarrow 0$ and $f_a(t)\rightarrow-\infty $ as  $t\rightarrow +\infty$, which implies that $\rho_a$ is the
unique global maximum point of $f_a(t)$. Hence, the maximum value of $f_a(t)$ is
\begin{align*}
\max_{t>0}f_a(t)=  f_a(\rho_a) =&\frac{1}{2}-\mu \frac{C_q^q}{q}a^{(1-\gamma_q)q}\rho_a^{\gamma_qq-2}
-\frac{C_p^p}{p}a^{(1-\gamma_p)p}\rho_a^{p\gamma_p-2}\\
=&\frac{1}{2}-\mu \frac{C_q^q}{q}a^{(1-\gamma_q)q}\left[\frac{\mu p(2-q\gamma_q)C_q^q}{q(p\gamma_p-2)C_p^p}a^{q(1-\gamma_q)-p(1-\gamma_p)} \right]^{\frac{\gamma_qq-2}{p\gamma_p-q\gamma_q}}\\
&-\frac{C_p^p}{p}a^{(1-\gamma_p)p}\left[\frac{\mu p(2-q\gamma_q)C_q^q}{q(p\gamma_p-2)C_p^p}a^{q(1-\gamma_q)-p(1-\gamma_p)} \right]^{\frac{p\gamma_p-2}{p\gamma_p-q\gamma_q}}\\
=&\frac{1}{2}-K a^{ \frac{(p-p\gamma_p)(2-q\gamma_q)+(q-q\gamma_q)(p\gamma_p-2)}{p\gamma_p-q\gamma_q}}.
\end{align*}
By the definition of $a_0$, we conclude \eqref{0k50}. Thus we complete the proof.
\end{proof}

\begin{lemma}\label{KKK-Lem2.1}
Assume that $\mu>0$, $q\in(2, \frac{8}{3})$ and $p\in (\frac{10}{3}, 6)$.   Let $(\widehat{a}_1, \widehat{t}_{1})\in \mathbb{R}^+\times \mathbb{R}^+$ be such that $f(\widehat{a}_1, \widehat{t}_{1}) \geq0$. Then for any $\widehat{a}_2\in(0, \widehat{a}_1]$,  it holds that
$f(\widehat{a}_2,  \widehat{t}_{2})\geq 0$ if $\widehat{t}_{2}\in\big[\frac{\widehat{a}_2}{\widehat{a}_1}  \widehat{t}_{1},   \widehat{t}_{1}\big]$.
\end{lemma}

\begin{proof}
Since the function $a \rightarrow f(\cdot, t)$ is non-increasing, we obtain that, for any $\widehat{a}_2\in(0, \widehat{a}_1]$,
\begin{align}\label{k76}
f(\widehat{a}_2,  \widehat{t}_{1})\geq f(\widehat{a}_1,  \widehat{t}_{1})\geq0.
\end{align}
By direct calculation, we get
\begin{align}\label{k77}
f\big(\widehat{a}_2,  \frac{\widehat{a}_2}{\widehat{a}_1}   \widehat{t}_{1}\big)- f(\widehat{a}_1,  \widehat{t}_{1})\geq0.
\end{align}
Note that if $f(\widehat{a}_2, t')\geq 0$ and $ f(\widehat{a}_2, t'' )\geq 0$,  then $f(\widehat{a}_2, k)\geq 0$ for any $k\in [t',  t'']$.
Indeed, if   $f(\widehat{a}_2, \widehat{t})<0$ for some $\widehat{t} \in [t',  t'']$, then there exists a local minimum point on  $(t',  t'')$,
which contradicts the fact that the function $f(\widehat{a}_2, t)$ has a unique   global maximum by Lemma \ref{KK-Lem2.1}. Hence, by \eqref{k76} and \eqref{k77}, taking $t'= \frac{\widehat{a}_2}{\widehat{a}_1}  \widehat{t}_{1}$ and $t''=  \widehat{t}_{1}$, we get the conclusion. This lemma is verified.
\end{proof}

\begin{lemma}\label{K-Lem2.1}
Assume $\mu>0$, $q\in(2, \frac{8}{3})$ and $p\in (\frac{10}{3}, 6)$.
Let $a\in(0, \overline{a}_0)$, $\rho_0:=\rho_{a_0}$ with $\rho_{a_0} $ defined by \eqref{k90} and
\begin{align*}
h(t):=\frac{1}{2}t^2-\mu \frac{C_q^q}{q}a^{(1-\gamma_q)q}t^{\gamma_qq}
 -\frac{C_p^p}{p}a^{(1-\gamma_p)p}t^{p\gamma_p},~~~~  \forall\ t\in \mathbb{R}^+.
\end{align*}
Then there exist $0 < R_0<\rho_0< R_1$, both $R_0$ and $R_1$ depend on $a$ and $\mu$,
such that $h(R_0) = 0 = h(R_1)$ and $h(t) > 0$ iff $t \in(R_0, R_1)$. Moreover, the function $h$ has a local strict minimum at negative level
and a global strict maximum at positive level.
  \end{lemma}
 \begin{proof}
 From the analysis of Lemma \ref{KK-Lem2.1}, we know that    there exist $0 < R_0< R_1$, both $R_0$ and $R_1$ depending
on $a$ and $\mu$, such that $f_a(R_0) = 0 = f_a(R_1)$ and $f_a(t) > 0$ iff $t \in(R_0, R_1)$.
Due to $h(t)=t^2 f_a(t)$, then the function $h$  satisfies $h(R_0) = 0 = h(R_1)$ and $h(t) > 0$ iff $t \in(R_0, R_1)$.
Moreover,  noting  that   $a \rightarrow f(\cdot, t)$ is decreasing and $f(a_0, \rho_0)=0$,
then $f(a, \rho_0)> f(\overline{a}_0, \rho_0)\geq f(a_0, \rho_0)=0 $ when $a\in(0, \overline{a}_0)$.
Hence, $\rho_0\in (R_0, R_1)$. The remains are similar to \cite[Lemma 5.1]{2020Soave}, so we omit it here.
 \end{proof}

In the following, we try to split $\mathcal{P}$ into three disjoint unions $  \mathcal{P}_{+}\cup  \mathcal{P}_{-}\cup  \mathcal{P}_{0}$, where
\begin{align*}
&\mathcal{P}_+:=\left\{u\in \mathcal{P}: \psi_u''(0)>0\right \}, \\
&\mathcal{P}_-:=\left\{u\in \mathcal{P}: \psi_u''(0)<0 \right\},\\
&\mathcal{P}_0:=\left\{u\in \mathcal{P}: \psi_u''(0)=0 \right\}.
 \end{align*}

\begin{lemma}\label{K-Lem2.2}
Let $\mu>0$, $q\in(2, \frac{8}{3})$, $p\in (\frac{10}{3}, 6)$ and  $a\in(0,\overline{a}_0)$.
Then $ \mathcal{P}_0=\emptyset$ and $\mathcal{P}$ is a smooth manifold of codimension 2 in $\mathcal{H}$.
\end{lemma}
\begin{proof}
Suppose on the contrary that there exists $u\in  \mathcal{P}_0$. By \eqref{k8} and \eqref{k4} one has
\begin{align}
&P(u)=\int_{\mathbb{R}^3}|\nabla u|^2dx+\frac{1}{4}\int_{\mathbb{R}^3} \int_{\mathbb{R}^3} \frac{|u(x)|^{2}|u(y)|^2} {|x-y|}dxdy
-\gamma_p\int_{\mathbb{R}^3}|u|^{p}dx
-\mu\gamma_q\int_{\mathbb{R}^3}|u|^{q}dx=0,\label{k5}\\
&\psi_u''(0)=2\int_{\mathbb{R}^3}|\nabla u|^2dx+\frac{1}{4}\int_{\mathbb{R}^3} \int_{\mathbb{R}^3} \frac{|u(x)|^{2}|u(y)|^2} {|x-y|}dxdy
-p\gamma_p^2\int_{\mathbb{R}^3}|u|^{p}dx
-\mu q \gamma_q^2\int_{\mathbb{R}^3}|u|^{q}dx=0.\label{k6}
 \end{align}
  By eliminating  $|\nabla u|_2^2$ from \eqref{k5}  and   \eqref{k6}, we get
  \begin{align}\label{k7}
\mu  (2-q\gamma_q)\gamma_q\int_{\mathbb{R}^3}|u|^{q}dx
= (p \gamma_p-2)\gamma_p\int_{\mathbb{R}^3}|u|^{p}dx
+\frac{1}{4}
\int_{\mathbb{R}^3}\int_{\mathbb{R}^3}\frac{|u(x)|^{2}|u(y)|^2}{|x-y|}dxdy.
 \end{align}
On the one hand, it follows from   \eqref{k5} and   \eqref{k7} that
\begin{align}\label{k9}
\int_{\mathbb{R}^3}|\nabla u|^2dx+\frac{1}{4}\left(1-\frac{1}{2-q\gamma_q}\right)
\int_{\mathbb{R}^3}\int_{\mathbb{R}^3}\frac{|u(x)|^{2}|u(y)|^2} {|x-y|}dxdy
=\gamma_p\left(\frac{p\gamma_p-q\gamma_q}{2-q\gamma_q}\right)\int_{\mathbb{R}^3}|u|^{p}dx.
\end{align}
Since $q\gamma_q<1$ and $p\gamma_p>q\gamma_q$, then by \eqref{k9},
\begin{align}\label{k11}
 |\nabla u|_2^2
 \leq\gamma_p\left(\frac{p\gamma_p-q\gamma_q}{2-q\gamma_q}\right)\int_{\mathbb{R}^3}|u|^{p}dx
 \leq C_p^p\gamma_p\left(\frac{p\gamma_p-q\gamma_q}{2-q\gamma_q}\right)a^{p(1-\gamma_p)}|\nabla u|_2^{p\gamma_p}.
 \end{align}
 On the other  hand, by eliminating  $\int_{\mathbb{R}^3}|u|^{p}dx$ from \eqref{k5}  and   \eqref{k7}, one obtains
  \begin{align}\label{k12}
 \int_{\mathbb{R}^3}|\nabla u|^2dx+\frac{1}{4}\left(1+\frac{1}{p\gamma_p-2}\right)\int_{\mathbb{R}^3}\int_{\mathbb{R}^3} \frac{|u(x)|^{2}|u(y)|^2} {|x-y|}dxdy
 =\mu\gamma_q\left(\frac{p\gamma_p-q\gamma_q}{p\gamma_p-2}\right)\int_{\mathbb{R}^3}|u|^{q}dx.
 \end{align}
 Since $p\gamma_p>2$ and $p\gamma_p> q\gamma_q$, we deduce from \eqref{k12} that
  \begin{align}\label{k13}
 |\nabla u|_2^2
\leq\mu\gamma_q\left(\frac{p\gamma_p-q\gamma_q}{p\gamma_p-2}\right)\int_{\mathbb{R}^3}|u|^{q}dx
 \leq\mu\gamma_q C_q^q\left(\frac{p\gamma_p-q\gamma_q}{p\gamma_p-2}\right) a^{q(1-\gamma_q)}|\nabla u|_2^{q\gamma_q}.
 \end{align}
Hence, from \eqref{k11}, \eqref{k13} and the definition of $\overline{a}_0$, we can get a contradiction based on \cite[Lemma 5.2]{2020Soave}. Thus, $  \mathcal{P}_0=\emptyset$. Moreover,    similar to \cite[Lemma 5.2]{2020Soave}, we obtain that   $\mathcal{P}$ is a smooth manifold of codimension 2 in $\mathcal{H}$.
Thus we complete the proof.
\end{proof}

\begin{lemma}\label{K-Lem2.3}
Assume $\mu>0$ and $a\in(0, \overline{a}_0)$.
Let $q\in(2, \frac{8}{3})$, $p\in (\frac{10}{3}, 6)$ and $\rho_0:=\rho_{a_0}$, for any $u\in S_a$,
the function $\psi_u$  has exactly two critical points $s_u < t_u $ and
two zeros $c_u < d_u$, with $s_u < c_u < t_u < d_u$. Moreover,
\begin{itemize}
  \item [(i)] $s_u\star u\in \mathcal{P}_+$ and $t_u \star u\in \mathcal{P}_-$. If $s\star   u\in \mathcal{P}$, then $s=s_u$ or $s=t_u$.
  \item [(ii)] $|\nabla (s \star u)|_2\leq R_0\leq \rho_0$ for every $s\leq c_u$, $\ln\big(\frac{\rho_0}{|\nabla u|_2}\big)< \ln\big(\frac{R_1}{|\nabla u|_2}\big)\leq d_u$ and
  \begin{align*}
  \mathcal{J}(s_u \star u)&=\min\left\{  \mathcal{J}(s\star u):s\in \mathbb{R} \ \text{and}\ |\nabla (s\star u)|_2\leq R_0 \right\}\\
  &=\min\left\{  \mathcal{J}(s\star u):s\in \mathbb{R} \ \text{and}\ |\nabla (s\star u)|_2\leq \rho_0 \right\}\\
  &<0.
  \end{align*}
  \item [(iii)] $\mathcal{J}(t_u\star u)=\max\left\{  \mathcal{J}(s\star u):s\in \mathbb{R} \right\}>0$,
  and $\psi_u$ is strictly decreasing and concave on $(t_u,+\infty)$. In particular, if $t_u < 0$, then $P(u) <0$.
    \item [(iv)] $u\in S_a\mapsto s_u\in \mathbb{R}$ and $u\in S_a\mapsto t_u\in \mathbb{R}$ are of class $C^1$.
\end{itemize}
  \end{lemma}

 \begin{proof}
Since $s\star u\in \mathcal{P}\Leftrightarrow \psi_u'(s)=0$, then we first show that $\psi_u$ has at least two critical points.
Note
$$
J(u)\geq \frac{1}{2}|\nabla u|_2^2-\mu \frac{C_q^q}{q} a^{q(1-\gamma_q)}|\nabla u|_2^{q\gamma_q}-\frac{C_p^p}{p} a^{p(1-\gamma_p)}|\nabla u|_2^{p\gamma_p}:=h(|\nabla u|_2),
$$
then
$$
\psi_u(s)=J(s\star u)\geq h(|\nabla (s\star u)|_2)=h(e^s|\nabla u|_2).
$$
By Lemma \ref{K-Lem2.1}, we have $h(t)>0\Leftrightarrow t\in(R_0, R_1)$, which implies $\psi_u>0 $ on $(\log(\frac{R_0}{|\nabla u|_2}), \log(\frac{R_1}{|\nabla u|_2}))$. Furthermore, $\psi_u(-\infty)=0^-$ and $\psi_u(+\infty)=-\infty$. Hence, $\psi_u$ has at least two critical points $s_u<t_u$, where $s_u$ is a local minimum point on $(-\infty,\log(\frac{R_0}{|\nabla u|_2}))$ at negative level, and $t_u > s_u$ is a global maximum point at positive level.
Next, we prove that $\psi_u$ has at most two critical points. In other words, we show that $\psi_u'(s)=0$  has at most two  solutions.
Let $\psi_u'(s)=0$, namely,
$$
e^{s}\int_{\mathbb{R}^3}|\nabla u|^2dx+\frac{1}{4} \int_{\mathbb{R}^3} \int_{\mathbb{R}^3} \frac{|u(x)|^{2}|u(y)|^2} {|x-y|}dxdy
-\gamma_p e^{(p\gamma_p-1)s} \int_{\mathbb{R}^3}|u|^{p}dx
-\mu\gamma_q e^{(q\gamma_q-1)s}
\int_{\mathbb{R}^3}|u|^{q}dx=0.
$$
Define
$$
f_1(s)=\gamma_p e^{(p\gamma_p-1)s} \int_{\mathbb{R}^3}|u|^{p}dx
+\mu\gamma_q e^{(q\gamma_q-1)s}
\int_{\mathbb{R}^3}|u|^{q}dx-e^{s}\int_{\mathbb{R}^3}|\nabla u|^2dx,
$$
then
$$
f_1'(s)=e^sf_2(s),
$$
where
$$
f_2(s)=\gamma_p(p\gamma_p-1)e^{(p\gamma_p-2)s} \int_{\mathbb{R}^3}|u|^{p}dx+\mu\gamma_q(q\gamma_q-1)e^{(q\gamma_q-2)s}
\int_{\mathbb{R}^3}|u|^{q}dx-\int_{\mathbb{R}^3}|\nabla u|^{2}dx.
$$
Clearly, $f_2'(s)>0$ for all $s\in \mathbb{R}$, then $f_2$ is strictly increasing on $\mathbb{R}$.
Moreover, since $f_2(-\infty)=-\infty$ and $f_2(+\infty)=+\infty$, there exists a unique $\widetilde{s}\in \mathbb{R}$ such that $f_2(\widetilde{s})=0$, and then $f_1'(\widetilde{s})=0$. Furthermore, $f_1$ is strictly decreasing on $(-\infty, \widetilde{s})$ and is  strictly increasing on $( \widetilde{s}, +\infty)$. Since $f_1(-\infty)=+\infty$    and $f_1(+\infty)=+\infty$,  then the equation
$$
f_1(s)= \frac{1}{4} \int_{\mathbb{R}^3} \int_{\mathbb{R}^3} \frac{|u(x)|^{2}|u(y)|^2} {|x-y|}dxdy
$$
has at most two  solutions. That is,
$\psi_u'(s)=0$  has at most two  solutions. Therefore,  $\psi_u$ has exactly two critical points $s_u< t_u\in \mathbb{R}$. Moreover, $\psi_{s_u\star u}''(0)=\psi_{ u}''(s_u)\geq 0$ and $\mathcal{P}_0=\emptyset$ imply $s_u\star u\in \mathcal{P}_+$. Similarly, $t_u\star u\in \mathcal{P}_-$. The remains are similar to the proof of \cite[Lemma 5.3]{2020Soave}. We finish the proof.
\end{proof}

Let $A_k:=\left\{u\in \mathcal{H}: |\nabla u|_2\leq k\right\}$  and
\begin{align}\label{k91}
D_{\rho_0}=A_{\rho_0}\cap S_a.
\end{align}
We shall consider the local minimization problem
$$
 m(a):=\inf_{D_{\rho_0}}\mathcal{J}(u).
 $$
From Lemma \ref{K-Lem2.3}, we have

\begin{corollary}\label{K-Lem2.4}
Assume that $\mu>0$,  $a\in(0, \overline{a}_0)$, $q\in(2, \frac{8}{3})$ and $p\in (\frac{10}{3}, 6)$.
Then $\mathcal{P}_+\subset D_{\rho_0}$ and $\sup\limits_{\mathcal{P}_+} \mathcal{J}(u)\leq 0\leq \inf\limits_{\mathcal{P}_-} \mathcal{J}(u)$.
\end{corollary}

\begin{lemma}\label{K-Lem2.5}
Let $\mu>0$,  $a\in(0, \overline{a}_0)$, $q\in(2, \frac{8}{3})$ and $p\in (\frac{10}{3}, 6)$.
Then $-\infty<m(a)<0<\inf\limits_{\partial D_{\rho_0}}\mathcal{J}(u)$ and $m(a)=\inf\limits_{\mathcal{P}}\mathcal{J}(u)=\inf\limits_{\mathcal{P}_+}\mathcal{J}(u)$.
Moreover, $m(a)<\inf\limits_{\overline{D_{\rho_0}}\backslash D_{\rho_0-\epsilon}} \mathcal{J}(u)$ for $\epsilon>0$ small enough.
\end{lemma}
\begin{proof}
The proof is similar to \cite[Lemma 2.4]{2022-JMPA-jean} and \cite[Lemma 5.3]{2020Soave}, we omit the details.
\end{proof}

\section{The local minimizer of the case $\mu>0$, $q\in(2,\frac{8}{3})$ and $p\in(\frac{10}{3},6)$}\label{sec1}
To solve the case of $\mu>0$, $q\in(2, \frac{8}{3} )$ and $p\in (\frac{10}{3}, 6)$,  we use the ideas of \cite{2011JFA-BS}.
Firstly, we present two definitions, which are different from \cite{2011JFA-BS}
since  we do not directly consider the
minimization problem on $S_a$ as in \cite{2011JFA-BS}.

\begin{definition}\label{K-def1}
  Let $u\in \mathcal{H} $ and $u\neq0$. A continuous path $g_u: \theta\in \mathbb{R}^+\mapsto g_u(\theta)\in \mathcal{H}$ such that $g_u(1)=u$  is said to be a scaling path of $u$ if $|\nabla g_u(\theta)|_2^2\rightarrow |\nabla u|_2^{2}$ as $ \theta\rightarrow 1$,
$\Theta_{g_u}(\theta):=|g_u(\theta)|_2^2|u|_2^{-2}$ is differentiable and $\Theta_{g_u}'(1)\neq 0$.
 We denote with $\mathcal{G}_u$ the set of the scaling paths of $u$.
\end{definition}
\newpage

The set $\mathcal{G}_u$ is nonempty. For example, $g_u( \theta) = \theta u\in \mathcal{G}_u$, since $ \Theta_{g_u}(\theta)= \theta^ 2$.
Also $g_u( \theta)= u( \frac{x}{\theta })$ is an element of $\mathcal{G}_u$ since  $ \Theta_{g_u}(\theta)= \theta^3$.
As we will see, it is relevant to consider the family of scaling paths of $u$ parametrized with  $\iota\in \mathbb{R}$ given by
\begin{align*}
 \mathcal{G}_u^{\iota}=\left\{g_u(\theta): g_u(\theta)=\theta^{1-\frac{3}{2}\iota} u(\frac{x}{\theta^{\iota}}) \right\}\subset \mathcal{G}_u.
\end{align*}

\begin{definition}\label{K-def2}
 Let  $u\neq0$ be fixed and $g_u\in \mathcal{G}_u$. We say that   the scaling path $g_u$ is admissible for
the functional $\mathcal{J}$ if $h_{g_u}$ is a differentiable function, where
$h_{g_u}(\theta)= \mathcal{J}(g_u(\theta) )- \Theta_{g_u}(\theta)\mathcal{J}(u)$ for all $\theta\in \mathbb{R}^+$.
\end{definition}

The following result  is crucial for handling the dichotomy case.

\begin{lemma}\label{K-Lem2.7}
Assume that
 \begin{align}
 &-\infty<m(s)<0\ \text{for all} \ s\in(0,\overline{a}_0),\label{k51}\\
 &s\in(0,\overline{a}_0)\mapsto m(s)\ \text{is continuous},\label{k52}\\
 &\lim_{s\rightarrow 0}\frac{m(s)}{s^2}=0.\label{k53}
 \end{align}
Then for any $a\in (0, \overline{a}_0)$, there exists $a_1\in(0, a]$ such that $ m(a_1)<m(\beta)+m(\sqrt{a_1^2-\beta^2})$ for any $\beta\in (0, a_1)$.
\end{lemma}

\begin{proof}
 Let  $a\in (0, \overline{a}_0)$ fixed and define
 $$
 a_1=\min\left\{s\in[0,a]: \frac{m(s)}{s^2}=\frac{m(a)}{a^2}\right\}.
 $$
From \eqref{k52} and \eqref{k53}, $a_1>0$ follows. We claim the function $s\mapsto \frac{m(s)}{s^2}$ in the interval $[0, a_1]$ achieves minimum only at $s=a_1$.
If the claim holds, then for any $\beta\in (0, a_1)$, we get $\frac{\beta^2}{a_1^2}m(a_1)<m(\beta)$ and $\frac{a_1^2-\beta^2}{a_1^2}m(a_1)< m(\sqrt{a_1^2-\beta^2})$. Hence,
 $$
 m(a_1)=\frac{\beta^2}{a_1^2}m(a_1)+\frac{a_1^2-\beta^2}{a_1^2}m(a_1)<m(\beta)
 +m\Big(\sqrt{a_1^2-\beta^2}\Big).
 $$
 So, we next verify the claim. In fact, suppose that there exists $a_{\ast}<a_1$ such that $\frac{m(s)}{s^2}$ achieves minimum at $s=a_{\ast}$.
 Then we have $\frac{m(a_{\ast})}{a_{\ast}^2}< \frac{m(a_{1})}{a_{1}^2}= \frac{m(a)}{a^2}<0$. Thus, from \eqref{k51} and  \eqref{k53}, there exists $\overline{a}<a_1$ such that $ \frac{m(\overline{a})}{{\overline{a}}^2}=\frac{m(a_{1})}{a_{1}^2}= \frac{m(a)}{a^2} $, which contradicts the definition of $a_1$. Therefore, we complete the proof.
\end{proof}

To obtain the strong subadditivity inequality $m(a)<m(a_1)+m\big(\sqrt{a^2-a_1^2}\big)$ for any $a_1\in (0, a)$,
in the following we will show that the function $s\mapsto \frac{m(s)}{s^2}$  is monotone decreasing for $s\in(0, a]$.

\begin{lemma}\label{K-Lem2.8}(Avoiding dichotomy)
Let \eqref{k51}, \eqref{k52} and \eqref{k53}   hold.  Then  for any $a\in(0, \overline{a}_0)$,   the set
 \begin{align*}
M(a)=\cup_{\rho\in (0, a]}\left\{ u\in S_{\rho}\cap A_{\rho_0}: \mathcal{J}(u)=m(\rho)\right\}\neq\emptyset.
 \end{align*}
If in addition,
\begin{align}\label{k54}
\forall \ u\in M(a),\ \exists\ g_u\in \mathcal{G}_u\ \text{admissible, such that}\ \frac{d}{d \theta}h_{g_u}(\theta)|_{\theta=1}\neq 0,
\end{align}
then  the function $s\mapsto \frac{m(s)}{s^2}$  is monotone decreasing for $s\in(0, a]$.
\end{lemma}

\begin{proof}
By Lemma \ref{K-Lem2.7},   for any $a\in (0, \overline{a}_0)$, there exists $a_1\in(0, a]$ such that for any $\beta\in (0, a_1)$,
\begin{align}\label{k59}
m(a_1)<m(\beta)+m\Big(\sqrt{a_1^2-\beta^2}\Big).
\end{align}
We claim that
\begin{align}\label{k55}
\left\{u\in S_{a_1}\cap A_{\rho_0}: \mathcal{J}(u)=m(a_1)\right\}\neq\emptyset.
\end{align}
Indeed, taking a minimizing sequence $\{u_n\}\subset S_{a_1}\cap A_{\rho_0} $ of $m(a_1)$,
clearly, $\{u_n\}$ is bounded in $ \mathcal{H}$. Hence,  there exists some $  u\in \mathcal{H}$ satisfying
 \begin{align*}
 &u_n\rightharpoonup u \quad\quad\quad\quad \text{in}\   \mathcal{H};\\
 & u_n\rightarrow u \quad\quad \quad\quad\text{in}\  L_{loc}^t(\mathbb{R}^3), \ \forall\ t\in(2,6);\\
 &u_n(x)\rightarrow u(x) \quad \ \mbox{a.e.} \ \text{in}\ \mathbb{R}^3.
 \end{align*}
If $u\neq0$, the minimizing sequence $\{u_n\}$ is the desired.
If $u=0$, then the following two cases occur:
\begin{align*}
&(\mbox{i})\ \lim_{n\rightarrow\infty}\sup_{y\in \mathbb{ R}^3 }\int_{B_y{(1)}}|u_n|^2dx=0,\\
&(\mbox{ii}) \ \lim_{n\rightarrow\infty}\sup_{y\in \mathbb{ R}^3 }\int_{B_y{(1)}}|u_n|^2dx\geq\delta>0,
\end{align*}
If $(\mbox{i})$ holds,  by the well-known Lions' lemma \cite{1983Wi}, $u_n \rightarrow 0 $ in $L^t (\mathbb{R}^3)$ for any $t\in(2, 6)$ and so
$$
m(a_1)=\lim_{n\rightarrow\infty}\mathcal{ J}(u_n)=\frac{1}{2}\int_{\mathbb{R}^3}|\nabla u_n|^2dx+\frac{1}{4}\int_{\mathbb{R}^3} \int_{\mathbb{R}^3} \frac{|u_n(x)|^{2}|u_n(y)|^2} {|x-y|}dxdy\geq 0,
$$
which is contradict with $ m(a_1)<0 $ by \eqref{k51}.
Thus, $(\mbox{ii})$ holds.  In this case we choose $\{y_n\}  \subset\mathbb{ R}^3$ such that
$$
\int_{B_0(1)}|u_n(\cdot+y_n)|^2dx\geq \delta>0.
$$
Moreover, we deduce that
 the sequence $\{u_n(\cdot+y_n)\}\subset S_{a_1}\cap A_{\rho_0} $ is still a minimizing sequence for $m(a_1)$.  Hence, there exists $ \widehat{u}\in \mathcal{H}$  such that $ \widehat{u}\neq 0$ and
 \begin{align*}
 &u_n(\cdot+y_n)\rightharpoonup \widehat{u} \quad\quad\quad\quad \text{in}\   \mathcal{H};\\
 & u_n(\cdot+y_n)\rightarrow \widehat{u} \quad\quad \quad\quad\text{in}\  L_{loc}^t(\mathbb{R}^3),\ \forall\ t\in(2,6);\\
 &u_n(x+y_n)\rightarrow \widehat{u}(x) \quad  \quad \ \mbox{a.e.} \ \text{in}\ \mathbb{R}^3.
 \end{align*}
Thus we find a minimizing sequence $\{v_n\}\subset S_{a_1}\cap A_{\rho_0} $ of $m(a_1)$ such that  there is $v\in \mathcal{H}\backslash\{0\}$~satisfying
 \begin{align*}
v_n\rightharpoonup v \quad\quad\quad\quad  &\text{in}\   \mathcal{H};\\
  v_n\rightarrow v \quad\quad \quad\quad &\text{in}\  L_{loc}^t(\mathbb{R}^3)\ \text{with}\ t\in(2,6);\\
   v_n(x)\rightarrow v(x) \quad \  &\mbox{a.e.} \ \text{in}\ \mathbb{R}^3.
 \end{align*}
Let $w_n=v_n-v$, then $w_n\rightharpoonup 0$ in $   \mathcal{H}$. It follows from Br\'{e}zis-Lieb lemma \cite{1983Wi} that, as $n\rightarrow\infty$,
 \begin{align*}
&\|w_n\|^2=\|v_n\|^2-\|v\|^2+o_n(1);\\
 &|w_n|_s^s=|v_n|_s^s-|v|_s^s+o_n(1),
\end{align*}
where $s\in [2,6]$.
Then, by using \cite[Lemma 2.2]{2008-JMAA-ZHAO} we obtain, as $n\rightarrow\infty$,
 \begin{align*}
 \int_{\mathbb{R}^3}  \int_{\mathbb{R}^3} \frac{|w_n(x)|^2|w_n(y)|^2} {|x-y|}dxdy = \int_{\mathbb{R}^3}  \int_{\mathbb{R}^3} \frac{|v_n(x)|^2|v_n(y)|^2} {|x-y|}dx dy -\int_{\mathbb{R}^3}  \int_{\mathbb{R}^3} \frac{|v(x)|^2|v(y)|^2} {|x-y|}dxdy+o_n(1).
\end{align*}
Hence, we have
 \begin{align}\label{k56}
\mathcal{J}(v_n)=\mathcal{J}(w_n)+\mathcal{J}(v)+o_n(1).
 \end{align}
Let $|v|_2=a_2\in(0, a_1]$.  If $a_2\in(0, a_1)$, then $v\in S_{a_2}\cap A_{\rho_0}$. Hence,
 \begin{align}\label{k57}
\mathcal{J}(v)\geq m(a_2).
 \end{align}
Besides, since $|w_n|_2=|v_n-v|_2\rightarrow \sqrt{a_1^2-a_2^2}>0$ as $n\rightarrow\infty$ and $|\nabla w_n|_2\leq  \rho_0$ for $n$ large enough,  then $w_n\in S_{|w_n|_2}\cap A_{\rho_0}$, and   we deduce from \eqref{k52} that
\begin{align}\label{k58}
 \lim_{n\rightarrow\infty}\mathcal{J}(w_n)\geq \lim_{n\rightarrow\infty} m(|w_n|_2)=   m\Big(\sqrt{a_1^2-a_2^2}\Big).
 \end{align}
From \eqref{k56}-\eqref{k58}, we deduce that
 \begin{align*}
m(a_1)=\lim_{n\rightarrow\infty}\mathcal{J}(v_n)
=\lim_{n\rightarrow\infty}\mathcal{J}(w_n)+\mathcal{J}(v)\geq m\Big(\sqrt{a_1^2-a_2^2}\Big)+m(a_2),
 \end{align*}
which is contradict with \eqref{k59}. Consequently,  $a_2=a_1$.  Then we have $v\in S_{a_1}\cap A_{\rho_0}$ and so
$$
m(a_1)\leq \mathcal{J}(v)\leq \lim_{n\rightarrow\infty}\mathcal{J}(v_n)=m(a_1),
$$
which implies that $ \mathcal{J}(v)= m(a_1)$ and $\|v_n-v\|\rightarrow 0 $ as $n\rightarrow\infty$. Hence, \eqref{k55} holds. Then, $M(a)\neq \emptyset$.

To prove that the function $s\mapsto \frac{m(s)}{s^2}$  is monotone decreasing for $s\in(0, a]$, we only need to show that  the function $s\mapsto \frac{m(s)}{s^2}$ in every interval $[0, b]$ achieves its unique minimum in $s=b$, where $b\in(0, a]$.  Let $a\in(0, a_0)$ be fixed and $c:=\min_{[0, b]}\frac{m(s)}{s^2}<0$ by \eqref{k51}. Let
$$
b_0:=\min\Big\{ s\in[0,b]: \frac{m(s)}{s^2}=c\Big\}.
$$
We have to prove that $b_0=b$. It follows from \eqref{k51} and \eqref{k52} that $b_0>0$ and
\begin{align*}
\frac{m(b_0)}{b_0^2}< \frac{m(s)}{s^2} \quad \text{for all}\ s\in[0, b_0).
\end{align*}
That is, the function $s\mapsto \frac{m(s)}{s^2}$ in the interval $[0, b_0]$ achieves its unique minimum in $s=b_0$. Thus, for any $b_1\in (0, b_0)$, we get $\frac{b_1^2}{b_0^2}m(b_0)<m(b_1)$ and $\frac{b_0^2-b_1^2}{b_0^2}m(b_0)< m\big(\sqrt{b_0^2-b_1^2}\big)$. Hence, for any $b_1\in (0, b_0)$,
 $$
 m(b_0)=\frac{b_1^2}{b_0^2}m(b_0)+\frac{b_0^2-b_1^2}{b_0^2}m(b_0)<m(b_1)+m\Big(\sqrt{b_0^2-b_1^2}\Big).
 $$
 Similar to the proof of \eqref{k55}, we get
 \begin{align*}
 \left\{u\in S_{b_0}\cap A_{\rho_0}: \mathcal{J}(u)=m(b_0)\right\}\neq\emptyset.
 \end{align*}
 Hence, there exists $w\in S_{b_0}\cap A_{\rho_0}$ such that $\mathcal{J}(w)=m(b_0)$. Especially, $w\in M(a)$. Now assume that $b_0 <b$. Then fixed $ g_w\in \mathcal{G}_w$ with its associated $\Theta(\theta)$, by the Definition  \ref{K-def1}, for any $\varepsilon\in(0, \epsilon)$, where $\epsilon$ is given in Lemma \ref{K-Lem2.5}, there exists $\delta_{\varepsilon}'>0$ such that when $|\theta-1|< \delta_{\varepsilon}'$,
  \begin{align*}
 \left|\Theta(\theta) b_0^2-b_0^2 \right|=\left |\Theta(\theta) |w|_2^2-|w|_2^2 \right|=\left| |g_w(\theta)|_2^2- |w|_2^2\right|< \varepsilon.
 \end{align*}
 Due to the fact that  the function $s\mapsto \frac{m(s)}{s^2}$ in the interval $[0, b_0]$ achieves its unique minimum in $s=b_0$ and $b_0<b$,
  we have
   \begin{align*}
 \frac{m(b_0)}{b_0^2}\leq  \frac{m\big(\Theta(\theta) b_0^2\big)}{\Theta(\theta) b_0^2}\quad\quad \text{for all}\ \theta\in(1-\delta_{\varepsilon}', 1+\delta_{\varepsilon}')
 \end{align*}
 Moreover, for any $\varepsilon\in(0, \epsilon)$,   there is $\delta_{\varepsilon}''>0$ such that when $|\theta-1|< \delta_{\varepsilon}''$,
    \begin{align*}
 \left||\nabla g_w(\theta)|_2^2-|\nabla w |_2^2\right|<\varepsilon,
 \end{align*}
 which  and Lemma \ref{K-Lem2.5} show that
 $$
 |\nabla g_w(\theta)|_2^2\leq |\nabla w |_2^2+\varepsilon\leq \rho_0-\epsilon+\varepsilon\leq \rho_0 \quad\quad\text{for all}\ \theta\in(1-\delta_{\varepsilon}'', 1+\delta_{\varepsilon}'').
 $$
   Therefore, choosing $\delta_{\varepsilon}=\min\{\delta_{\varepsilon}', \delta_{\varepsilon}'' \}$, one has
 \begin{align*}
 \frac{\mathcal{J}(g_w(\theta,x))}{\Theta(\theta,x) b_0^2 }\geq\frac{m\big(\Theta(\theta) b_0^2\big)}{\Theta(\theta) b_0^2}\geq \frac{m(b_0)}{b_0^2}=\frac{\mathcal{J}(w)}{b_0^2}
 \quad\quad \text{for all}\ \theta\in(1-\delta_{\varepsilon}, 1+\delta_{\varepsilon}),
 \end{align*}
which gives that
$h_{g_w}(\theta)=\mathcal{J}(g_w(\theta))- \Theta(\theta) \mathcal{J}(w)$ is nonegative in  $(1-\delta_{\varepsilon}, 1+\delta_{\varepsilon})$ and has a global minimum in $\theta=1$ with $h_{g_w}(1)=0$. Then we get $\frac{dh_{g_w}(\theta)}{d \theta}|_{\theta=1}=0$. Since $g_w$ is arbitrary, this relation has to be true for every map  $g_w$. Thus,
we get a contradiction   with \eqref{k54}.  Hence, we deduce  $b_0=b$. That is,  the function $s\mapsto \frac{m(s)}{s^2}$ in every interval $[0, b]$ achieves its unique minimum in $s=b$, where $b\in(0, a]$.
Hence,  $s\mapsto \frac{m(s)}{s^2}$  is monotone decreasing for $s\in(0, a]$.   We complete the proof.
\end{proof}

According to Lemma \ref{K-Lem2.5}, it is easy to see that $ \eqref{k51}$ holds. In what follows, we check   that  \eqref{k52}, \eqref{k53} and \eqref{k54} are valid.

\begin{lemma}\label{K-Lem2.9}
 Let $\mu>0$, $q\in(2, \frac{8}{3})$ and $p\in (\frac{10}{3}, 6)$.
Then $a\in(0,\overline{a}_0)\mapsto m(a) $ is a continuous mapping.
\end{lemma}
\begin{proof}
Similarly as in \cite[Lemma 2.6]{2022-JMPA-jean}, let $a \in  (0, \overline{a}_0)$ be arbitrary and $\{a_n\} \subset(0, \overline{a}_0)$ be such that $a_n \rightarrow a$. From the definition of $m(a_n)$
and since $m(a_n) < 0$, for any  $\varepsilon> 0$ sufficiently small, there exists $u_n \in S_{a_n}\cap A_{\rho_0}$ such that
\begin{align}\label{k61}
\mathcal{J}(u_n)\leq m(a_n)+\varepsilon \quad\quad \text{and} \quad\quad\mathcal{J}(u_n)<0.
\end{align}
We set $v_n := \frac{a}{a_n} u_n$,
  and hence $v_n \in S_a$. We have that $v_n \in S_a\cap A_{\rho_0}$. Indeed, if $a_n \geq a$, then
  $$
  |\nabla v_n|_2=\frac{a}{a_n}|\nabla u_n|_2\leq |\nabla u_n|_2\leq \rho_0.
  $$
If $a_n \leq a$,   in view of Lemma \ref{KKK-Lem2.1} and $f(a, \rho_0)>f(a_0, \rho_0)=0$ we have
\begin{align}\label{k62}
f(a_n, \rho) \geq 0 \quad\quad \text{for any}\ \rho \in \big[\frac{a_n}{a}\rho_0,  \rho_0\big].
\end{align}
 Moreover, it follows from \eqref{k61} that
\begin{align*}
0>J(u_n)\geq& \frac{1}{2}|\nabla u_n|_2^2-\mu \frac{C_q^q}{q} a_n^{q(1-\gamma_q)}|\nabla u_n|_2^{q\gamma_q}-\frac{C_p^p}{p} a_n^{p(1-\gamma_p)}|\nabla u_n|_2^{p\gamma_p}\\
\geq& |\nabla u_n|_2^2f(a_n, |\nabla u_n|_2),
\end{align*}
which shows that $ f(a_n, |\nabla u_n|_2)<0$. Hence, by \eqref{k62}, we infer that $|\nabla u_n|_2< \frac{a_n}{a}\rho_0$. Then
$$
  |\nabla v_n|_2=\frac{a}{a_n}|\nabla u_n|_2\leq \frac{a}{a_n}\frac{a_n}{a}\rho_0=\rho_0.
$$
Due to $v_n\in S_a\cap A_{\rho_0}$, we infers that
$$
m(a)\leq \mathcal{J}(v_n)=\mathcal{J}(u_n)+[\mathcal{J}(v_n)- \mathcal{J}(u_n)],
$$
 where
\begin{align*}
 \mathcal{J}(v_n)- \mathcal{J}(u_n)=&\frac{1}{2}\Big[\big(\frac{a}{a_n}\big)^2-1\Big]|\nabla u_n|_2^2+\frac{1}{4}\Big[\big(\frac{a}{a_n}\big)^4-1\Big]\int_{\mathbb{R}^3} \int_{\mathbb{R}^3} \frac{|u_n(x)|^2|u_n(y)|^2} {|x-y|}dxdy\\
 &-\frac{1}{p}\Big[\big(\frac{a}{a_n}\big)^p-1\Big]|  u_n|_p^p
 -\mu\frac{1}{q}\Big[\big(\frac{a}{a_n}\big)^q-1\Big]|  u_n|_q^q.
\end{align*}
Since $|\nabla u_n|_2\leq \rho_0$, also $| u_n|_p$ and $|  u_n|_q$ are uniformly bounded,    we obtain, as $n\rightarrow\infty$,
 \begin{align}\label{k64}
m(a)\leq \mathcal{J}(v_n)=\mathcal{J}(u_n)+o(1)\leq m(a_n)+\varepsilon+o(1).
\end{align}
Besides,  let $u \in S_a\cap A_{\rho_0}$ be such that
 \begin{align}\label{k63}
\mathcal{J}(u)\leq m(a)+\varepsilon \quad\quad \text{and} \quad\quad\mathcal{J}(u)<0.
\end{align}
Let $u_n=\frac{a_n}{a}u$ and hence $u_n\in S_{a_n}$. Clearly, $|\nabla u|_2\leq \rho_0$ and $a_n\rightarrow a$ as $n\rightarrow\infty$ imply that $|\nabla u_n|_2\leq \rho_0$ for $n$ large enough. Thus $u_n\in S_{a_n}\cap A_{\rho_0}$. Moreover, $\mathcal{J}(u_n)\rightarrow \mathcal{J}(u)$ as $n\rightarrow\infty$, thus we deduce from \eqref{k63} that
 \begin{align}\label{k65}
m(a_n)\leq \mathcal{J}(u_n)=\mathcal{J}(u)+[\mathcal{J}(u_n)-\mathcal{J}(u)]\leq m(a)+\varepsilon+o(1).
\end{align}
Combining \eqref{k64} with \eqref{k65}, we get $m(a_n)\rightarrow m(a)$ as $n\rightarrow\infty$ for any $a\in(0, a_0)$.
\end{proof}

\begin{lemma}\label{K-Lem2.10}
 Let $\mu>0$, $a\in(0, \overline{a}_0)$, $q\in(2, \frac{8}{3})$ and $p\in (\frac{10}{3}, 6)$. Then
$\lim_{a\rightarrow 0}\frac{m(a)}{a^2}=0$.
\end{lemma}
\begin{proof}
Define
 \begin{align*}
I(u)=\frac{1}{2}\int_{\mathbb{R}^3}|\nabla u|^2dx
-\frac{1}{p}\int_{\mathbb{R}^3}|u|^{p}dx-\frac{\mu}{q}\int_{\mathbb{R}^3}|u|^{q}dx,
\end{align*}
and $\widetilde{m}(a):=\inf_{S_a\cap A_{\rho_0}} I$, where $a\in(0, \overline{a}_0)$.
Then $\mathcal{J} (u)\geq I(u) $ for any $u\in \mathcal{H}$. Hence $ 0> m(a)\geq \widetilde{m}(a)$,  which  implies that $ \frac{\widetilde{m}(a)}{a^2}\leq\frac{m(a)}{a^2}<0$. Hence, we only need to show that  $\lim_{a\rightarrow0}\frac{\widetilde{m}(a)}{a^2}=0$.   According to \cite{2020Soave}, when $q\in(2, \frac{8}{3})$ and $p\in (\frac{10}{3}, 6)$, we obtain that there exists $\widetilde{u}_{a}\in S_a\cap A_{\rho_0}$ such that $I(\widetilde{u}_{a})=\widetilde{m}(a)<0 $ for any $a\in(0, \overline{a}_0)$.  Then the sequence $\{\widetilde{u}_{a}\}_{a>0}$ is bounded in $D^{1,2}(\mathbb{R}^3)$. Since the minimizer $\widetilde{u}_{a}$ satisfies
\begin{align}\label{k79}
-\Delta \widetilde{u}_{a}-|\widetilde{u}_{a}|^{p-2}\widetilde{u}_{a}
-\mu|\widetilde{u}_{a}|^{q-2}\widetilde{u}_{a}=\omega_a  \widetilde{u}_{a},
\end{align}
 where $\omega_a$ is the Lagrange multiplier associated to the minimizer.
 We get
\begin{align}\label{k78}
 \frac{\omega_a}{2}=\frac{|\nabla  \widetilde{u}_{a}|_2^2-| \widetilde{u}_{a}|_p^p-\mu | \widetilde{u}_{a}|_q^q}{2| \widetilde{u}_{a}|_2^2}\leq\frac{\frac{1}{2}|\nabla  \widetilde{u}_{a}|_2^2-\frac{1}{p}| \widetilde{u}_{a}|_p^p-\mu\frac{1}{q} | \widetilde{u}_{a}|_q^q}{| \widetilde{u}_{a}|_2^2}=\frac{I(\widetilde{u}_{a} )}{a^2}< 0.
\end{align}
 Actually we prove that $\lim_{a\rightarrow0} \omega_a=0$, so by comparison in \eqref{k78} we get the lemma.
 To show that $\lim_{a\rightarrow0} \omega_a=0$,  suppose on the contrary that there exists a sequence $a_n \rightarrow 0$ such that $\omega_{a_n} <-\alpha$ for some $\alpha\in(0,1)$. Since the minimizers $\widetilde{u}_{a_n}\in S_{a_n}\cap A_{\rho_0}$ satisfy Eq.~\eqref{k79}, we get
\begin{align}
 |\nabla\widetilde{u}_{a_n} |_2^2-\omega_{a_n} |\widetilde{u}_{a_n} |_2^2-|\widetilde{u}_{a_n}|_p^p-\mu |\widetilde{u}_{a_n}|_q^q&=0,\label{k80}\\
 |\nabla\widetilde{u}_{a_n} |_2^2-\gamma_p|\widetilde{u}_{a_n}|_p^p-\mu\gamma_q|\widetilde{u}_{a_n}|_q^q&=0.\label{k81}
\end{align}
 By   $\eqref{k80}-\frac{1}{2}\eqref{k81}$, one infers that
 \begin{align*}
C\|\widetilde{u}_{a_n}\|^2
&\leq \frac{1}{2}|\nabla\widetilde{u}_{a_n} |_2^2+\alpha  |\widetilde{u}_{a_n} |_2^2\\
&\leq(1-\frac{1}{2}\gamma_p)| \widetilde{u}_{a_n} |_p^p+\mu(1-\frac{1}{2}\gamma_q)| \widetilde{u}_{a_n} |_q^q \\
  & \leq C(1-\frac{1}{2}\gamma_p)\| \widetilde{u}_{a_n} \|^p+\mu C(1-\frac{1}{2}\gamma_q)\| \widetilde{u}_{a_n} \|^q,
\end{align*}
 which implies that $\|\widetilde{u}_{a_n}\|\geq C $ due to $p, q>2$ and $\gamma_p, \gamma_q<1$. This shows that $ |\nabla\widetilde{u}_{a_n} |_2>C$ for $n$ large enough since $ |\widetilde{u}_{a_n} |_2=a_n\rightarrow 0$ as $n\rightarrow\infty$.
 Moreover, for any $u\in S_a$,
 \begin{align*}
I(u)=&\frac{1}{2}\int_{\mathbb{R}^3}|\nabla u|^2dx
-\frac{\mu}{q}\int_{\mathbb{R}^3}|u|^{q}dx-\frac{1}{p}\int_{\mathbb{R}^3}|u|^{p}dx\\
\geq& \frac{1}{2}|\nabla u|_2^2-\mu \frac{C_q^q}{q} a^{q(1-\gamma_q)}|\nabla u|_2^{q\gamma_q}-\frac{C_p^p}{p} a^{p(1-\gamma_p)}|\nabla u|_2^{p\gamma_p}.
\end{align*}
Hence, for $n$ large enough,
 \begin{align*}
0\geq I(\widetilde{u}_{a_n})
\geq\left(\frac{1}{2}-\mu \frac{C_q^q}{q} a_n^{q(1-\gamma_q)}|\nabla\widetilde{u}_{a_n}|_2^{q\gamma_q-2}-\frac{C_p^p}{p} a_n^{p(1-\gamma_p)}|\nabla\widetilde{u}_{a_n}|_2^{p\gamma_p-2}\right)|\nabla\widetilde{u}_{a_n}|_2^2
\geq \frac{1}{4}C>0,
\end{align*}
which is a contradiction. Therefore, $\lim_{a\rightarrow0} \omega_a=0$, which shows that $\lim_{a\rightarrow0}\frac{\widetilde{m}(a)}{a^2}=0$. Hence, we complete the proof.
\end{proof}

\begin{lemma}\label{K-Lem2.11}
 Let $\mu>0$, $q\in(2, \frac{8}{3})$, $p\in (\frac{10}{3}, 6)$ and $a\in(0, \overline{a}_0)$. Then,
 \begin{align*}
\forall \ u\in M(a),\ \exists\ g_u\in \mathcal{G}_u\ \text{admissible, such that}\ \frac{d}{d \theta}h_{g_u}(\theta)|_{\theta=1}\neq 0.
 \end{align*}
\end{lemma}

\begin{proof}
To prove this lemma,  we argue by contradiction   assuming that there exists a sequence $\{u_n\}\subset M(a) $
with $|\nabla u_n|_2\leq \rho_0$ and $a\geq |u_n|_2=a_n\rightarrow 0$ as $n\rightarrow\infty$ such that for all $\iota\in\mathbb{R}$  and    $g_{u_n}\in \mathcal{G}_{u_n}^\iota$,  $h_{g_{u_n}}'(1)=0$,  namely,
 \begin{align}\label{k66}
-\iota|\nabla u_n|_2^2+\frac{2-\iota}{4}\int_{\mathbb{R}^3} \int_{\mathbb{R}^3} \frac{|u_n(x)|^2|u_n(y)|^2} {|x-y|}dxdy
=\frac{p-2+\frac{6-3p}{2}\iota}{p}|u_n|_p^p
+\frac{q-2+\frac{6-3q}{2}\iota}{q}|u_n|_q^q.
 \end{align}
Moreover, since $\{u_n\}\subset M(a) $, then it holds that
 \begin{align}\label{k67}
|\nabla u_n|_2^2+\frac{1}{4}\int_{\mathbb{R}^3} \int_{\mathbb{R}^3} \frac{|u_n(x)|^2|u_n(y)|^2} {|x-y|}dxdy
-\gamma_p|u_n|_p^p
-\gamma_q|u_n|_q^q=0.
 \end{align}
Using  \eqref{k66} and \eqref{k67},   we have
\begin{align}\label{k70}
\frac{1}{2}\int_{\mathbb{R}^3} \int_{\mathbb{R}^3} \frac{|u_n(x)|^2|u_n(y)|^2} {|x-y|}dxdy-\frac{p-2}{p}|u_n|_p^p-\frac{q-2}{q}|u_n|_q^q=0.
 \end{align}
 Since $|\nabla u_n|_2\leq \rho_0$ and $ |u_n|_2\rightarrow 0$ as $n\rightarrow\infty$,  one has,  $|u_n|_p, |u_n|_q\rightarrow 0 $  as $n\rightarrow\infty$.
When $q\in(2, \frac{12}{5}]$, by \eqref{k70}, Lemma \ref{gle4} and  the interpolation inequality, we have
$$
 \frac{2(q-2)}{q}|u_n|_q^q\leq\int_{\mathbb{R}^3} \int_{\mathbb{R}^3} \frac{|u_n(x)|^2|u_n(y)|^2} {|x-y|}dxdy\leq C|u_n|_{\frac{12}{5}}^4\leq C|u_n|_{q}^{ \frac{6q}{6-q}}|u_n|_6^{4-\frac{6q}{6-q}},
$$
which is a contradiction since $\frac{6q}{6-q}>q$, $4-\frac{6q}{6-q}\geq0$, $|u_n|_{q}\rightarrow 0$ as $n\rightarrow\infty$ and $|u_n|_6\leq C$.
If $q\in(\frac{12}{5}, \frac{8}{3})$, it follows from \eqref{k70} and  the interpolation inequality that
\begin{align*}
 \frac{2(q-2)}{q}|u_n|_q^q&\leq\int_{\mathbb{R}^3} \int_{\mathbb{R}^3} \frac{|u_n(x)|^2|u_n(y)|^2} {|x-y|}dxdy\\
 &\leq C|u_n|_{\frac{12}{5}}^4\leq C|u_n|_{q}^{4(1-t)}|u_n|_2^{4t},
  \end{align*}
   where $ t=\frac{5q-12}{6q-2}$,
  a contradiction  since $q<4(1-t)$ and $|u_n|_{q}\xrightarrow{n}0$. Thus, we complete the proof.
\end{proof}

\begin{lemma}\label{K-Lem2.12}
Under the assumptions of Theorem \ref{K-TH1},    let $\{u_n\}\subset D_{\rho_0}$ be  a minimizing sequence of $m(a)$,  then there exists $u\in D_{\rho_0}$ such that $u_n\rightarrow u $  in $\mathcal{H}$ as $n\rightarrow\infty$ up to  translation and $\mathcal{J}(u)=m(a)$.
\end{lemma}

\begin{proof}
From Lemma \ref{K-Lem2.8},  the function $s\mapsto \frac{m(s)}{s^2}$  is monotone decreasing for $s\in(0, a]$.
Then for any $a_1\in (0, a)$, we get $\frac{a_1^2}{a^2}m(a)<m(a_1)$ and $\frac{a^2-a_1^2}{a^2}m(a)< m(\sqrt{a^2-a_1^2})$. Hence, for any $a_1\in (0, a)$,
 $$
 m(a)=\frac{a_1^2}{a^2}m(a)+\frac{a^2-a_1^2}{a^2}m(a)<m(a_1)+m\Big(\sqrt{a^2-a_1^2}\Big).
 $$
  The remain is similar to the proof of \eqref{k55}. Hence, we complete the proof.
\end{proof}

 \noindent{\bf Proof of Theorem \ref{K-TH1}.}
Under the assumption of Theorem \ref{K-TH1},  from Lemma \ref{K-Lem2.12}, there exists $u\in D_{\rho_0}$ such that $ \mathcal{J}(u)=m(a)$.  Since if $\{u_n\}\subset D_{\rho_0}$  is a minimizing sequence of $m(a)$, then $\{|u_n|\}\subset D_{\rho_0}$  is also a minimizing sequence of $m(a)$. Hence,
 in view of Lemma \ref{K-Lem2.5}, we obtain that
 there is $\lambda\in \mathbb{R}$ such that $(\lambda, u)\in \mathbb{R}\times \mathcal{H}$  is a ground state normalized solutions for
 Eq.~\eqref{k1}, where $u$ is  a real-valued non-negative function. Moreover,    the elliptic $L^s$ estimate implies that $u \in W^{2,s}_{loc} (\mathbb{R}^3) $ with $s\geq 2$.  Then, by Sobolev embedding theorem, we have $u\in C^{1,\alpha}_{loc}(\mathbb{R}^3) $  for $0 <  \alpha < 1$. Hence, it follows from the strong maximum principle   that $u > 0$.
  Besides,   Corollary \ref{K-Lem2.4}  and  Lemma \ref{K-Lem2.5} imply that any ground state for   Eq.~\eqref{k1} is a local minimizer of $\mathcal{J}$ on $D_{\rho_0}$.  $\hfill\Box$

\section{The second solution of the case $\mu>0$, $q\in(2, \frac{12}{5} ]$ and $p\in(\frac{10}{3},6)$}\label{sec3}
  In this part,  we address the case $\mu>0$, $q\in(2, \frac{12}{5} ]$ and $p\in (\frac{10}{3}, 6)$
  and consider the problem in $\mathcal{H}_r$. Let $S_{a,r}= \mathcal{H}_r\cap S_a  $, $\mathcal{P}_{\pm,r}= \mathcal{H}_r\cap\mathcal{P}_{\pm} $ and $m(a,r)=\inf_{\mathcal{P}_{+,r} } \mathcal{J}(u)$. All conclusions in Sec. \ref{sec2} are still valid.
  Using a similar method in the proof of Theorem \ref{K-TH1}, we can obtain that there exists $(\lambda, u)\in \mathbb{R}\times S_{a,r}$  is a   solution for Eq.~\eqref{k1},  where  $\mathcal{J}(u)= m(a,r)<0$. Now, we show that Eq.~\eqref{k1} has a second  solution (Mountain Pass type) in this case.
 Firstly, we verify the Mountain Pass geometry of $\mathcal{J} $ on manifold $S_{a,r}$. Let
\begin{align*}
\Gamma:=\left\{\zeta\in C([0,1], S_{a,r}): \zeta(0)\in \mathcal{P}_{+,r}\ \ \text{and}\ \ \mathcal{J}(\zeta(1))<2m(a,r)\right\},
\end{align*}
and
$$
c_a:=\inf_{\zeta\in \Gamma}\max_{u\in \zeta([0,1])}\mathcal{J}(u).
$$
Taking $v\in S_{a,r}$, there is $s_v\in \mathbb{R}$ such that $s_v\star v\in \mathcal{P}_{+,r}$  by Lemma \ref{K-Lem2.3}. When $k>1$ large enough,
$$
\zeta_v(\tau)=[(1-\tau)s_v+\tau k] \star v\in \Gamma,
$$
which implies that $\Gamma\neq \emptyset $. Next, we claim that
$$
c_a=\sigma:=\inf_{u\in\mathcal{P}_{-,r}}\mathcal{J}(u)>0.
$$
On the one hand, assuming that there exists $\widetilde{v}\in \mathcal{P}_{-,r}$ such that $\mathcal{J}(\widetilde{v})<c_a$. Setting $s\star \widetilde{v}= e^{\frac{3}{2}s} \widetilde{v}(e^s x) $. Let $s_1>1$ sufficient large be such that $\mathcal{J}(s_1\star \widetilde{v} )<2m(a,r) $. Hence,  for any $\tau\in [0,1]$,
$$
 \zeta_{\widetilde{v}}(\tau)=[(1-\tau)s_{\widetilde{v}}+\tau s_1]\star \widetilde{v}\in \Gamma,
$$
which and Lemma \ref{K-Lem2.3}-(iii) show that
$$
c_a\leq \max_{\tau\in [0,1]}\mathcal{J}(\zeta_{\widetilde{v}}(\tau))=\max_{\tau\in [0,1]}\mathcal{J}([(1-\tau)s_{\widetilde{v}}+\tau s_1]\star \widetilde{v})=\mathcal{J}(\widetilde{v} )< c_a,
$$
a contradiction. Thus, for any $u\in  \mathcal{P}_{-,r}$, one infers  $\mathcal{J}(u)\geq c_a $, which implies that $\sigma\geq c_a$. On the other hand, for any $\zeta\in \Gamma$, since $\zeta(0)\in \mathcal{P}_{+,r}$, then there exists $t_{\zeta(0)}> s_{\zeta(0)}=0$ satisfying $t_{\zeta(0)}\star \zeta(0)\in   \mathcal{P}_{-,r} $. Moreover, since $\mathcal{J}(\zeta(1))\leq 2m(a,r) $, then similar to \cite[Lemma 5.6]{2020Soave}, we know that $t_{\zeta(1)}<0$. Therefore, there exists $\overline{\tau}\in(0,1)$ such that $t_{\zeta(\overline{\tau})}=0$ by the continuity. Hence, $\zeta(\overline{\tau})\in  \mathcal{P}_{-,r}$. So,
$$
\max_{u\in \zeta([0,1])}\mathcal{J}(u)\geq \mathcal{J}(\zeta(\overline{\tau}))\geq \inf_{\mathcal{P}_{-,r} }\mathcal{J}(u)=\sigma,
 $$
which implies that $c_a\geq \sigma$. Hence, we deduce that $c_a= \sigma$. Furthermore, let $t_{\max}$  is the maximum point of function $h$. For any $u\in  \mathcal{P}_{-,r}$, there is $\tau_u\in \mathbb{R}$ such that $|\nabla (\tau_u \star u)|_2=t_{\max}$. For every $u\in \mathcal{P}_{-,r}$, one has
$$
\mathcal{J}(u)\geq \mathcal{J}(\tau_u\star u)\geq h(|\nabla (\tau_u \star u) |_2)=h(t_{\max} )>0,
$$
which gives that $\sigma>0$. That is, $c_a= \sigma>0$.

For any $a\in(0, \overline{a}_0)$ fixed, it is easy to verify that the set
$$
L:=\{u\in \mathcal{P}_{-,r}:\mathcal{J}(u)\leq c_a+1\}
$$
is bounded. Then let $M_0>0$ be such that $L\subset B(0, M_0)$, where
$$
B(0, M_0):=\{u\in \mathcal{H}_r: \|u\|\leq M_0\}.
$$

In order to prove the following lemma, we need to develop a deformation argument on $S_a$. Following \cite{1983-BL-II},  we recall that, for any $a > 0$, $S_a$ is a submanifold of $\mathcal{H}$ with codimension $1$ and the
tangent space at a point $\overline{u}\in S_a$ is defined as
$$
T_{\overline{u}}=\{v\in \mathcal{H}: (\overline{u}, v)_{L^2}=0 \}.
$$
The restriction $\mathcal{J}_{|_{S_a}}: S_a\rightarrow \mathbb{R}$ is a $C^1$ functional on $S_a$ and for any $\overline{u}\in S_a$ and any $v\in T_{\overline{u} }$,
 $$
 \langle\mathcal{J}_{|_{S_a}}'(\overline{u} ), v\rangle= \langle\mathcal{J}'(\overline{u}), v\rangle.
 $$
We use the notation   $\|d\mathcal{J}_{|_{S_a}}(\overline{u} ) \|$  to indicate the norm in the cotangent space $ T_{\overline{u}}'$, that is,
the dual norm induced by the norm of $T_{\overline{u}}$, that is,
$$
\|d\mathcal{J}_{|_{S_a}}(\overline{u} ) \|:= \sup_{\|v\|\leq 1, v\in T_{\overline{u}}}| \langle d\mathcal{J}(\overline{u} ), v \rangle|.
$$
Let $\widetilde{S}_a:=\{u\in S_a: d\mathcal{J}_{|_{S_a}}(u )\neq 0\}$. We know from \cite{1983-BL-II} that there exists a locally
Lipschitz pseudo gradient vector field $Y \in C^1(\widetilde{S}_a, T(S_a))$ (here $T(S_a)$ is the tangent bundle) such that
\begin{align}\label{k41}
\|Y(u)\|\leq 2 \|d\mathcal{J}_{|_{S_a}}(u)\|
\end{align}
and
\begin{align}\label{k42}
\langle\mathcal{J}_{|_{S_a}}'(u), Y(u)\rangle\geq \|d\mathcal{J}_{|_{S_a}}(\overline{u} ) \|^2,
\end{align}
for any $u\in \widetilde{S}_a$. Note that $\|Y(u)\|\neq 0$ for $u\in \widetilde{S}_a$ thanks to \eqref{k42}. Now an arbitrary but fixed $\delta > 0$, we consider the sets
\begin{align*}
&\widetilde{N}_{\delta}:=\{u\in S_a: |\mathcal{J}(u)-c_a |\leq \delta,\ dist(u, \mathcal{P}_{-,r})\leq 2 \delta, \ \|Y(u)\|\geq 2\delta\},\\
& \widetilde{N}_{\delta}:=\{u\in S_a: |\mathcal{J}(u)-c_a|< 2\delta\},
\end{align*}
where $dist(x, \mathcal{A}):=\inf\{\|x-y\|: y\in \mathcal{A}\}$. Assuming that  $ \widetilde{N}_{\delta}$ is
nonempty there exists a locally Lipschitz function $g : S_a \rightarrow [0, 1]$ such that
\begin{align}\label{k104} g(u)=
\begin{cases}
1,\quad\quad &\text{on}\ \widetilde{N}_{\delta},\\
0,\quad\quad &\text{on}\ N_{\delta}^c.
\end{cases}
\end{align}
We also define on $S_a$ the vector field $W$ by
\begin{align}\label{k44}W(u)=
\begin{cases}
-g(u)\frac{Y(u)}{\|Y(u)\|},\quad\quad &\text{if}\ u\in \widetilde{S}_a,\\
0,\quad\quad &\text{if}\ u\in S_a\backslash \widetilde{S}_a,
\end{cases}
\end{align}
and  the pseudo gradient flow
\begin{align}\label{k45}
\begin{cases}
 \frac{d}{dt}\eta(t,u)=W(\eta(t,u)),\\
\eta(0,u)=u.
\end{cases}
\end{align}
The existence of a unique solution $\eta(t, \cdot)$ of \eqref{k44} defined for all $t\in \mathbb{R}$ follows from standard
arguments and we refer to \cite[Lemma 5]{1983-BL-II} for this. Let us recall some of its basic properties:
\begin{itemize}
  \item [(i)] $\eta(t, \cdot)$ is a homeomorphism of $S_a$;
  \item [(ii)]$\eta(t, u)=u$ for all $t\in \mathbb{R}$ if $|\mathcal{J}(u)-c_a |\geq 2\delta$;
  \item [(iii)] $ \frac{d}{dt} \mathcal{J}(\eta(t, u))=\langle \mathcal{J}'(\eta(t, u)), W(\eta(t, u) ) \rangle\leq 0$ for all $t\in \mathbb{R}$ and $u\in S_a$.
\end{itemize}

The following results help to obtain a special Palais-Smail sequence, inspired by \cite{2011JFA-BS}.

\begin{lemma}\label{K-Lem3.1}
Let $\mu>0$, $q\in (2, \frac{12}{5}]$, $p\in (\frac{10}{3}, 6)$, $a\in(0, \overline{a}_0)$ and
\begin{align*}
\Omega_{\delta}:=\left\{u\in S_{a,r}:\ |\mathcal{J}(u)-c_a  |\leq\delta, \ dist(u, \mathcal{P}_{-,r})\leq 2\delta, \ \|\mathcal{J}'_{|_{S_{a,r}}}(u)\|_{\mathcal{H}_r^{-1}}\leq 2 \delta\right\},
\end{align*}
then for any $\delta>0$, $\Omega_{\delta}\cap B(0, 3M_0)$ is nonempty.
\end{lemma}

\begin{proof}
Define
$$
\Lambda_{\delta}:=\{u\in S_{a,r}: \ |\mathcal{J}(u)-c_a  |\leq\delta, \ dist(u, \mathcal{P}_{-,r})\leq 2\delta\}.
$$
Suppose on the contrary   that there is $\overline{\delta}\in (0, \frac{c_a}{2})$ such that
\begin{align}\label{k105}
u\in \Lambda_{\overline{\delta}}\cap B(0, 3M_0)\Rightarrow \|\mathcal{J}'_{|_{S_{a,r}}}(u)\|_{\mathcal{H}_r^{-1}}> 2 \overline{\delta}.
\end{align}
From \eqref{k42},
\begin{align}\label{k106}
u\in \Lambda_{\overline{\delta}}\cap B(0, 3M_0)\Rightarrow u\in \widetilde{N}_{\overline{\delta}}.
\end{align}
Note that, by  \eqref{k45}, for any $u\in S_{a,r}$, it holds  that $\big\|\frac{d}{dt}\eta(t,u)\big\|\leq 1$  for all $t\geq 0$,
then there exists $s'>0$, depending on $\overline{\delta}>0$, such that, for all $s\in(0, s')$,
\begin{align}\label{k46}
u\in \Lambda_{\frac{\overline{\delta}}{2}}\cap B(0, 2M_0)\Rightarrow  \eta(s,u)\in B(0, 3M_0)\quad \quad \text{and}\quad\quad dist(\eta(s,u), \mathcal{P}_{-,r} )\leq 2 \overline{\delta}.
\end{align}
We claim that, taking $\varepsilon>0$ small enough, we can construct a path $\zeta_\varepsilon(t)\in \Gamma$ satisfying
$$
\max_{t\in[0,1]}\mathcal{J}(\zeta_\varepsilon(t))\leq c_a+\varepsilon,
$$
and
\begin{align}\label{k47}
\mathcal{J}(\zeta_\varepsilon(t))\geq c_a\Rightarrow\zeta_\varepsilon(t)\in \Lambda_{\frac{ \overline{\delta}}{2}}\cap B(0, 2M_0).
\end{align}
In fact, for $\varepsilon>0$ small, let $u_\varepsilon \in \mathcal{P}_{-,r}$ be such that $\mathcal{J}(u_\varepsilon)\leq c_a+\varepsilon $, and considering the path
\begin{align*}
\zeta_{\varepsilon}(t)=[(1-t)s_{u_{\varepsilon}}+t k]\star u_{\varepsilon},
\end{align*}
where $ t\in [0,1] $, $ k>0 $ large enough and   $s_{u_{\varepsilon}}<0$.
Clearly,
$$
\max_{t\in[0,1]}\mathcal{J}(\zeta_\varepsilon(t))=\mathcal{J}(u_\varepsilon) \leq c_a+\varepsilon.
$$
Since  $u_\varepsilon \in \mathcal{P}_{-,r}$,
similar to \eqref{k11},  one has
 \begin{align*}
 |\nabla u_\varepsilon|_2^2
 \leq\gamma_p\left(\frac{p\gamma_p-q\gamma_q}{2-q\gamma_q}\right)
 \int_{\mathbb{R}^3}|u_\varepsilon|^{p}dx
 \leq C_p^p\gamma_p\left(\frac{p\gamma_p-q\gamma_q}{2-q\gamma_q}\right)a^{p(1-\gamma_p)}|\nabla u_\varepsilon|_2^{p\gamma_p},
 \end{align*}
which shows that
\begin{align*}
|\nabla u_\varepsilon|_2\geq \Big(\frac{1}{C_p^p \gamma_p a^{p(1-\gamma_p)}}\Big)^{\frac{1}{p \gamma_p-2}} \Big( \frac{2-q\gamma_q}{p\gamma_p-q\gamma_q}  \Big)^{\frac{1}{p \gamma_p-2}}.
\end{align*}
Moreover, it is easy to see that $\{u_{\varepsilon}\}_{\varepsilon}$ is bounded in $\mathcal{H}_r$. Then, let $\varepsilon\rightarrow 0$,
\begin{align*}
&A(u_{\varepsilon}):=\int_{\mathbb{R}^3}|\nabla u_\varepsilon|^2dx\rightarrow A>0, \quad \quad B(u_{\varepsilon}):=\int_{\mathbb{R}^3} \int_{\mathbb{R}^3} \frac{|u_\varepsilon(x)|^{2}|u_\varepsilon(y)|^2} {|x-y|}dxdy\rightarrow B\geq0,\\
&D(u_{\varepsilon}):=\int_{\mathbb{R}^3}|u_\varepsilon|^{p}dx\rightarrow D>0, \quad\quad  \ E(u_{\varepsilon}):=\int_{\mathbb{R}^3}|u_\varepsilon|^{q}dx\rightarrow E\geq0.
\end{align*}
Since  $u_\varepsilon \in \mathcal{P}_{-,r}$ and $\lim_{\varepsilon\rightarrow 0}\mathcal{J}(u_\varepsilon)=c_a$, then we have
\begin{align}
\lim_{\varepsilon\rightarrow 0}P(u_\varepsilon)&=A+\frac{1}{4}B
-\gamma_pD-\mu\gamma_qE=0, \label{k25}\\
\lim_{\varepsilon\rightarrow 0}\psi_{u_{\varepsilon}}''(0)&=2A+\frac{1}{4}B
-p\gamma_p^2D
-\mu q \gamma_q^2E\leq0,\label{k26}\\
\lim_{\varepsilon\rightarrow 0}\mathcal{J} (u_{\varepsilon})&=\frac{1}{2}A+\frac{1}{4}B
-\frac{1}{p}D-\frac{\mu}{q}E=c_a.\label{k27}
 \end{align}
 Considering the function
\begin{align*}
 L(t)=\frac{1}{2}A t^2+\frac{1}{4}B t-\frac{1}{p}Dt^{p \gamma_p} -\frac{\mu}{q}Et^{q\gamma_q},\quad\quad \text{for all}\ t>0.
 \end{align*}
 We claim that the function  $ L(t)$ has   a unique global maximum point at $t=1$. Indeed, by simple calculate and from \eqref{k25},  we have
 \begin{align*}
 L'(t)&= A t+\frac{1}{4}B -\gamma_pDt^{p \gamma_p-1} - \mu \gamma_qEt^{q\gamma_q-1} \\
  &=\frac{1}{4}(1-t)B
+\gamma_pD(t-t^{p \gamma_p-1})+\mu\gamma_qE(t-t^{q\gamma_q-1}).
 \end{align*}
 Then
  \begin{align*}
 L''(t)&=-B
+\gamma_pD(1-(p \gamma_p-1)t^{p \gamma_p-2})+\mu\gamma_qE(1-(q\gamma_q-1)t^{q\gamma_q-2}),\\
 L'''(t)&=
 -\gamma_pD(p \gamma_p-1)(p \gamma_p-2)t^{p \gamma_p-3}-\mu\gamma_qE(q\gamma_q-1)(q\gamma_q-2)t^{q\gamma_q-3}.
 \end{align*}
Notice that  $L'''(t)<0$ for all $t>0$. Thus, the function $ L''(t)$ is strictly decrease on  $t>0$. Since $\lim_{t\rightarrow 0}L''(t)=+\infty$, $\lim_{t\rightarrow \infty}L''(t)=-\infty$ and
  \begin{align*}
  L''(1)&=-B
+\gamma_pD(2-p \gamma_p )+\mu\gamma_qE(2-q\gamma_q)\leq 0,
  \end{align*}
 where \eqref{k25} and \eqref{k26} are applied,
 then there exits $0<\widetilde{t}\leq 1$ such that $ L''(\widetilde{t})=0$. Therefore, the function $L'(t)$ is increase on $(0, \widetilde{t})$ and decrease on $(\widetilde{t}, +\infty)$. It is easy to see from $\lim_{t\rightarrow0} L(t)=0^-$ and $L(1)=c_a>0$ that
$L'(\widetilde{t})>0$. Furthermore,   $L'(1)=0$ implies that $\widetilde{t}< 1$. Since $\lim_{t\rightarrow0}L'(\widetilde{t})=-\infty $, $L'(\widetilde{t})$ has only two zero points, denoted by $t^*$ and 1. Hence, the function $L(t)$ has a local minimize   $t^*$ and global maximum $1$. By \eqref{k27}, we infer that
$$
\max_{t>0}L(t)=L(1)=c_a.
$$
Setting $y_{\varepsilon}=  (1-t)s_{u_{\varepsilon}}+t k \in (s_{u_{\varepsilon}}, k )$, assume that $e^{y_{\varepsilon}}\rightarrow l\geq 0$ as $\varepsilon\rightarrow 0$.
Since
 \begin{align*}
c_a \leq \lim_{\varepsilon\rightarrow 0} \mathcal{J}(\zeta_\varepsilon(t))
 =\frac{1}{2}l^2A+\frac{1}{4} l B - \frac{1}{p}l^{p\gamma_p}D-
 -\frac{\mu}{q}l^{q\gamma_q}E\leq L(1)=c_a,
\end{align*}
which implies that $l=1$ by the uniqueness of global maximum for the function $L(t)$. That is, $y_{\varepsilon}\rightarrow 0$ as $\varepsilon\rightarrow 0$.
 Hence, let $\varepsilon\in (0, \frac{1}{4} \overline{\delta}s')$ sufficiently small,   we get $\zeta_{\varepsilon}(t)\in \Lambda_{ \frac{\overline{\delta}}{2}} \cap B(0, 2M_0) $. Besides, applying the pseudo-gradient flow on $ \zeta_{\varepsilon}(t)$,   we see that
\begin{align}\label{k22}
\eta(s, \zeta_{\varepsilon}(\cdot))\in \Gamma\quad\quad \text{for all} \ s>0,
\end{align}
because $ \eta(s, u)=u$ for all $s>0$ if $| \mathcal{J}(u)-c_a |\geq 2 \overline{\delta}$. Next, we claim that, taking $s^*:=\frac{4 \varepsilon}{\overline{\delta}}<s'$,
\begin{align}\label{k21}
\max_{t\in [0,1]} \mathcal{J}( \eta( s^*,\zeta_{\varepsilon}(t)))<c_a .
\end{align}
Indeed, for simplicity, setting $w=\zeta_{\varepsilon}(t)$ for $t\in [0,1]$,

(1) If $ \mathcal{J}(w)<c_a $, then $\mathcal{J}( \eta( s^*, w))\leq  \mathcal{J}(w)<c_a$.

(2)  If $ \mathcal{J}(w)\geq c_a  $,   assume by contradiction that
\begin{align*}
\mathcal{J}( \eta( s, w))\geq c_a  \quad\quad \text{for all}\ s\in [0, s^*].
\end{align*}
 Since  $ \mathcal{J}( \eta( s, w))\leq  \mathcal{J}(w)\leq c_a+\varepsilon$,
 it follows from \eqref{k46} and  \eqref{k47} that $\eta( s, w)\in \Lambda_{\overline{\delta}}\cap B(0, 3M_0) $ for all $s\in [0, s^*]$. Moreover, we can see from \eqref{k42}, \eqref{k104}, \eqref{k105} and \eqref{k106} that   $\|Y(\eta( s, w) )\|\geq 2 \overline{\delta}$ and $ g(\eta( s, w) )=1$ for all $s\in [0, s^*]$. Thus, by \eqref{k44}, \eqref{k45} and \eqref{k105}, one has
\begin{align*}
\frac{d}{ds} \mathcal{J}( \eta( s, w))= \Big\langle d\mathcal{J}( \eta( s, w)), -\frac{Y(\eta(t, w)  }{\|Y(\eta(t, w)\|}    \Big\rangle.
\end{align*}
By integration, using $s^*=\frac{4\varepsilon}{\overline{\delta}}$, \eqref{k41}, \eqref{k42} and the fact that $\|Y(\eta( s, w) )\|\geq 2 \overline{\delta}$, we deduce that
$$
\mathcal{J}( \eta( s^*, w))\leq\mathcal{J}(w)-s^* \frac{\overline{\delta}}{2}\leq c_a +\varepsilon-2\varepsilon<c_a -\varepsilon,
$$
which is contradict with $\mathcal{J}( \eta( s^*, w))\geq c_a $. Thus, \eqref{k21} holds. It follows from  \eqref{k22} and  \eqref{k21} that
$$
c_a \leq \max_{t\in[0,1]}\mathcal{J}( \eta( s^*,  \gamma_{\varepsilon}(t)))<c_a ,
$$
a contradiction. So we complete the proof.
\end{proof}

\begin{lemma}\label{K-Lem3.2}
Let $\mu>0$, $q\in (2, \frac{12}{5}]$ and $p\in (\frac{10}{3}, 6)$. Then there exists a sequence $\{u_n\}\subset S_{a,r} $ and a constant $\alpha>0$ satisfying
$$
P(u_n)=o(1), \quad  \mathcal{J}(u_n)=c_a +o(1), \quad \|\mathcal{J}'_{|_{S_{a,r}}}(u_n)\|_{\mathcal{H}_{r}^{-1}}= o(1), \quad \|u_n\|\leq \alpha.
$$
\end{lemma}
\begin{proof}
We know from Lemma \ref{K-Lem3.1} that there exists $\{u_n\}\subset S_{a,r}$ satisfying $\{u_n\}\subset B(0, 3M_0)$ and
$$
dist(u_n, \mathcal{P}_{-,r})=o(1),\quad \quad |\mathcal{J}(u_n)-c_a |=o(1),\quad\quad \|\mathcal{J}_{|_{S_{a,r}}}'(u_n) \|_{H^{-1}}=o(1).
$$
In what follows, we show that $P(u_n)=o(1)$. Since $\|dP\|_{H^{-1}}$ is bounded on any bounded set of $\mathcal{H}_r$. Now, for any $n\in \mathbb{N}$ and any $w\in \mathcal{P}_{-,r}$, we have
$$
P(u_n)=P(w)+ dP(\beta u_n+(1-\beta)w)(u_n-w),
$$
where $\beta\in [0,1]$. Since $P(w)=0$,  then
\begin{align}\label{k23}
|P(u_n)|\leq \max_{u\in B(0, 3M_0)} \|dP\|_{H^{-1}} \|u_n-w\|.
\end{align}
Choosing $\{w_m\}\subset \mathcal{P}_{-,r}$ such that
\begin{align}\label{k24}
\|u_n-w_m\|\rightarrow dist(u_n, \mathcal{P}_{-,r} )
\end{align}
as $m\rightarrow\infty$. Since  $dist(u_n, \mathcal{P}_{-,r})\rightarrow 0 $ as $n\rightarrow\infty$,  \eqref{k23} and \eqref{k24} give  that $P(u_n)\rightarrow 0$ as $n\rightarrow\infty$.
\end{proof}

\begin{lemma}\label{K-Lem2.6}
Let $\mu>0$, $q\in (2, \frac{12}{5}]$, $p\in (\frac{10}{3}, 6)$ and \eqref{k3} be hold. Assume that $\{u_n\}\subset S_{a,r}$ is a Palais-Smail sequence at level $c\neq 0$ and $P(u_n)\rightarrow 0$ as $n\rightarrow\infty$. Then there exists $u\in S_{a,r}$ such that $u_n\rightarrow u$ in $\mathcal{H}_r$ and $(\lambda, u )\in \mathbb{R}^+ \times  \mathcal{H}_{r}$ solves Eq.~\eqref{k1} with $\mathcal{J} (u)=c $.
\end{lemma}

\begin{proof} Let us prove this result in three steps.

Step 1: $\{u_n\}$ is bounded in $\mathcal{H}_r$.
 Since $P(u_n)\rightarrow 0$ as $n\rightarrow\infty$,  then
   \begin{align}\label{k96}
     c+o(1)
 =&\mathcal{J} (u_n)-\frac{1}{p\gamma_p} P(u_n)\nonumber\\
 =&\left(\frac{1}{2}-\frac{1}{p\gamma_p}\right)\int_{\mathbb{R}^3}|\nabla u_n|^2dx+\frac{1}{4}\left(1-\frac{1}{p\gamma_p}\right)\int_{\mathbb{R}^3} \int_{\mathbb{R}^3} \frac{|u_n(x)|^{2}|u_n(y)|^2} {|x-y|}dxdy\nonumber\\
&+\mu\left(\frac{\gamma_q}{p\gamma_p}-\frac{1}{q}\right)\int_{\mathbb{R}^3}|u_n|^{q}dx\nonumber\\
 \geq& \left(\frac{1}{2}-\frac{1}{p\gamma_p}\right)\int_{\mathbb{R}^3}|\nabla u_n|^2dx-\mu  \frac{p\gamma_p-q\gamma_q}{pq\gamma_p} C_q^qa^{(1-\gamma_q)q}|\nabla u_n|_2^{q \gamma_q},
\end{align}
which shows that $\{u_n\}$ is bounded in $\mathcal{H}_r$ because of $q\gamma_q<2$.

Step 2: Since $\{u_n\}$ is bounded in $\mathcal{H}_r$, then there exists $u\in \mathcal{H}_r$ such that, up to a subsequence, as $n\rightarrow\infty$,
 \begin{align*}
   & u_n\rightharpoonup u \quad\quad \text{in}\ \mathcal{H}_r;\\
   & u_n\rightarrow u\quad\quad \text{in}\ L^t(\mathbb{R}^3)\ \text{with}\ t\in(2, 6);\\
   & u_n\rightarrow u \quad\quad  a.e \ \text{in}\ \mathbb{R}^3.
\end{align*}
Since $\mathcal{J}'_{|_{S_a}}(u_n)\rightarrow 0$,  then there exists $\lambda_n\in \mathbb{R} $ satisfying
\begin{align} \label{k19}
-\Delta u_n+\lambda_n u_n+ (|x|^{-1}\ast |u_n|^2)u_n=\mu |u_n|^{q-2}u_n+|u_n|^{p-2}u_n+o(1).
\end{align}
 Multiplying the above equation by $u_n$ and integrating in $\mathbb{R}^3$,
\begin{align} \label{k16}
 \lambda_n a^2= -|\nabla u_n|_2^2-\int_{\mathbb{R}^3} \int_{\mathbb{R}^3} \frac{|u_n(x)|^{2}|u_n(y)|^2} {|x-y|}dxdy
+\int_{\mathbb{R}^3}|u_n|^{p}dx+\mu \int_{\mathbb{R}^3}|u_n|^{q}dx+o( \|u_n\|).
\end{align}
Since $\{u_n\}\subset \mathcal{H}_r$ is bounded, then it follows from \eqref{k16} that  $\{\lambda_n\}$  is bounded in $\mathbb{R}$. Hence, there is $\lambda\in \mathbb{R}$ satisfying $\lambda_n\rightarrow \lambda$ as $n\rightarrow\infty$. Then we have
\begin{align} \label{k20}
-\Delta u+\lambda u+(|x|^{-1}\ast |u|^2)u=\mu |u|^{q-2}u+|u|^{p-2}u.
\end{align}
Clearly, $u\neq0$. If not, assume that $u=0$, then by Lemma \ref{gle4} and the fact that $\mathcal{H}_r\hookrightarrow L^t(\mathbb{R}^3 )$ with $t\in (2,6)$ is compact, we deduce
\begin{align*}
c=\lim_{n\rightarrow\infty}\mathcal{J}(u_n)
&=\lim_{n\rightarrow\infty}\Big(\mathcal{J}(u_n)- \frac{1}{2} P(u_n)\Big)\\
&=\lim_{n\rightarrow\infty}\bigg[\frac{1}{8} \int_{\mathbb{R}^3} \int_{\mathbb{R}^3} \frac{|u_n(x)|^2|u_n(y)|^2} {|x-y|}dxdy+(\gamma_p-\frac{1}{p})|u_n|_p^p+\mu(\gamma_q-\frac{1}{q})|u_n|_q^q\bigg]\\
&=0,
\end{align*}
which is a contradiction. Besides, we show that
\begin{align}\label{k40}
\lambda>0.
\end{align}
 Since $(\lambda, u)\in \mathbb{R}\times \mathcal{H}_r$ satisfies \eqref{k20},
then  the following equalities hold:
\begin{align}
|\nabla u|_2^2+ \lambda \int_{\mathbb{R}^3} |u|^{2}dx  +\int_{\mathbb{R}^3} \int_{\mathbb{R}^3} \frac{|u(x)|^{2}|u(y)|^2} {|x-y|}dxdy
-\int_{\mathbb{R}^3}|u|^{p}dx-\mu \int_{\mathbb{R}^3}|u|^{q}dx&=0,\label{k17}\\
 |\nabla u|_2^2+ \frac{1}{4}\int_{\mathbb{R}^3} \int_{\mathbb{R}^3} \frac{|u(x)|^{2}|u(y)|^2} {|x-y|}dxdy-\gamma_p \int_{\mathbb{R}^3}|u|^{p}dx-\mu \gamma_q \int_{\mathbb{R}^3}|u|^{q}dx&=0.\label{k18}
\end{align}
  We argue by contradiction   assuming that $\lambda\leq 0$. Then it follows from  \eqref{k17} and \eqref{k18} that, by eliminating $\int_{\mathbb{R}^3} \int_{\mathbb{R}^3} \frac{|u(y)|^2|u(x)|^{2}} {|x-y|}dydx$,
\begin{align*}
\frac{3}{4}|\nabla u|_2^2-\frac{\lambda}{4}  \int_{\mathbb{R}^3} |u|^{2}dx
+(\frac{1}{4}-\gamma_p)\int_{\mathbb{R}^3}|u|^{p}dx+\mu (\frac{1}{4}-\gamma_q) \int_{\mathbb{R}^3}|u|^{q}dx=0,
\end{align*}
which implies that
\begin{align*}
\frac{3}{4}|\nabla u|_2^2 \leq&
 (\gamma_p-\frac{1}{4})\int_{\mathbb{R}^3}|u|^{p}dx+\mu (\gamma_q-\frac{1}{4}) \int_{\mathbb{R}^3}|u|^{q}dx\\
 \leq& (\gamma_p-\frac{1}{4})\int_{\mathbb{R}^3}|u|^{p}dx\\
 \leq & (\gamma_p-\frac{1}{4}) C_p^p|\nabla u|_2^{p\gamma_p} a^{p(1-\gamma_p)}
\end{align*}
by using the fact that $q\in(2, \frac{12}{5}]$, $p\in(\frac{10}{3}, 6)$, $ |u|_2\leq a$ and the Gagliardo-Nirenberg inequality \eqref{k10}.
Then, since $p\gamma_p>2$ and $|\nabla u|_2\neq 0$,  we get
\begin{align}\label{k48}
 |\nabla u|_2^{p\gamma_p-2}\geq \frac{3}{4 (\gamma_p-\frac{1}{4})C_p^p a^{p(1-\gamma_p)} }.
\end{align}
Moreover, by eliminating $\int_{\mathbb{R}^3}|u|^{p}dx$
  from  \eqref{k17} and \eqref{k18},
  \begin{align*}
(\gamma_p-1)|\nabla u|_2^2+ \lambda \gamma_p \int_{\mathbb{R}^3} |u|^{2}dx
+(\gamma_p-\frac{1}{4}) \int_{\mathbb{R}^3} \int_{\mathbb{R}^3} \frac{|u(x)|^{2}|u(y)|^2} {|x-y|}dxdy +\mu ( \gamma_q-\gamma_p) \int_{\mathbb{R}^3}|u|^{q}dx=0.
\end{align*}
  Then,
    \begin{align*}
  0\leq-\lambda \gamma_p \int_{\mathbb{R}^3} |u|^{2}dx =&(\gamma_p-1)|\nabla u|_2^2+\big(\gamma_p-\frac{1}{4}\big) \int_{\mathbb{R}^3} \int_{\mathbb{R}^3} \frac{|u(x)|^{2}|u(y)|^2} {|x-y|}dxdy +\mu ( \gamma_q-\gamma_p) \int_{\mathbb{R}^3}|u|^{q}dx\\
  \leq & (\gamma_p-1)|\nabla u|_2^2+\big(\gamma_p-\frac{1}{4}\big) \int_{\mathbb{R}^3} \int_{\mathbb{R}^3} \frac{|u(x)|^{2}|u(y)|^2} {|x-y|}dxdy\\
  \leq &  (\gamma_p-1)|\nabla u|_2^2+ \big(\gamma_p-\frac{1}{4}\big)  C_{\frac{12}{5}}^{\frac{12}{5}}a^3|\nabla u|_2,
\end{align*}
which gives that
  \begin{align}\label{k49}
(1-\gamma_p) |\nabla u|_2\leq \big(\gamma_p-\frac{1}{4}\big)  C_{\frac{12}{5}}^{\frac{12}{5}}a^3.
\end{align}
Combining \eqref{k48} and \eqref{k49}, using the assumption \eqref{k3},
   we get a contradiction. Hence,   $\lambda>0$.

Step 3:  Multiplying \eqref{k19} and \eqref{k20} by $u_n-u$ and integrating in $\mathbb{R}^3$,
  $$
 \left\langle \mathcal{J}'(u_n)-\mathcal{J}'(u), u_n-u\right\rangle +\lambda_n \int_{\mathbb{R}^3} u_n(u_n-u)dx-\lambda \int_{\mathbb{R}^3} u(u_n-u)dx=o(1),
  $$
 by the fact that   $\lambda_n\rightarrow \lambda$ as $n\rightarrow\infty$ and $\mathcal{H}_r\hookrightarrow L^t( \mathbb{R}^3)$ with $t\in (2,6)$ is compact, we deduce
 $$
 |\nabla (u_n-u)|_2^2+ \lambda |u_n-u|_2^2=o(1),
 $$
 which shows that $\|u_n-u\|\rightarrow 0$ as $n\rightarrow\infty$ because  $\lambda>0$. Thus, we infer that $( \lambda, u)\in \mathbb{R}^+\times \mathcal{H}_r$ is a solution for Eq.~\eqref{k1}, which satisfies that $u_n\rightarrow u$ in $\mathcal{H}_r$ and $\mathcal{J}(u)=c$. So, we complete the proof.
\end{proof}

 \noindent {\bf{Proof of Theorem \ref{K-TH2}}.}
In view of Lemmas \ref{K-Lem3.2} and \ref{K-Lem2.6}, we obtain that there is $( \widehat{\lambda}, \widehat{u})\in \mathbb{R}^+ \times \mathcal{H}_r  $ such that $(\widehat{\lambda}, \widehat{u})$ solves Eq.~\eqref{k1} and $\mathcal{J}(\widehat{u})=c_a>m(a,r)\geq m(a)$. Similarly as the proof of Theorem \ref{K-TH1}, we can find that $\widehat{u}>0$.

\section{The exisence of the case $\mu\leq0$, $  2<q<\frac{8}{3}$ and $ \frac{10}{3}<p<6$}\label{sec6}
In this section, we consider the case $\mu\leq0$, $  2<q<\frac{8}{3}$ and $ \frac{10}{3}<p<6$ in $\mathcal{H}_{r}$. Let $S_{a,r}=\mathcal{H}_{r} \cap S_a$, $ \mathcal{P}_{0,r}=\mathcal{H}_{r}\cap \mathcal{P}_{0}$ and $\mathcal{P}_r=\mathcal{H}_{r}\cap \mathcal{P} $.

\begin{lemma}\label{K-Lem4.1}
  Let  $\mu\leq0$, $  2<q<\frac{8}{3}$ and $ \frac{10}{3}<p<6$. Then $ \mathcal{P}_{0,r}=\emptyset$, and $\mathcal{P}_r$ is a smooth manifold of codimension 2 in $\mathcal{H}_r$.
  \end{lemma}
 \begin{proof}
  Suppose on the contrary that there exists $u\in  \mathcal{P}_{0,r}$. By \eqref{k8} and \eqref{k4} one has
 \begin{align}
P(u)=\int_{\mathbb{R}^3}|\nabla u|^2dx+\frac{1}{4}\int_{\mathbb{R}^3} \int_{\mathbb{R}^3} \frac{|u(x)|^{2}|u(y)|^2} {|x-y|}dxdy
-\gamma_p\int_{\mathbb{R}^3}|u|^{p}dx
-\mu\gamma_q\int_{\mathbb{R}^3}|u|^{q}dx=0,\label{kk5}\\
\psi_u''(0)=2\int_{\mathbb{R}^3}|\nabla u|^2dx+\frac{1}{4}\int_{\mathbb{R}^3} \int_{\mathbb{R}^3} \frac{|u(x)|^{2}|u(y)|^2} {|x-y|}dxdy
-p\gamma_p^2\int_{\mathbb{R}^3}|u|^{p}dx
-\mu q \gamma_q^2\int_{\mathbb{R}^3}|u|^{q}dx=0.\label{kk6}
 \end{align}
  By eliminating  $|\nabla u|_2^2$ from \eqref{kk5}  and   \eqref{kk6}, we get
  \begin{align*}
\mu  (2-q\gamma_q)\gamma_q\int_{\mathbb{R}^3}|u|^{q}dx
= (p \gamma_p-2)\gamma_p\int_{\mathbb{R}^3}|u|^{p}dx
+\frac{1}{4}\int_{\mathbb{R}^3} \int_{\mathbb{R}^3} \frac{|u(x)|^{2}|u(y)|^2} {|x-y|}dxdy.
 \end{align*}
Since $\mu\leq0$ and $q \gamma_q < 2< p\gamma_p$, we get $u\equiv0$, which is contradict with $u\in S_{a,r}$. Moreover,  for the remaining proofs, similar to \cite[Lemma 5.2]{2020Soave}, we obtain that   $\mathcal{P}_r$ is a smooth manifold of codimension 2 in $\mathcal{H}_r$.
     Thus we complete the proof.
 \end{proof}

\begin{lemma}\label{K-Lem4.2}
  Let  $\mu\leq0$, $  2<q<\frac{8}{3}$ and  $\frac{10}{3}<p<6$. For every $u\in S_{a,r}$, there exists a unique $t_u\in \mathbb{R}$  such that  $t_u\star u\in \mathcal{P}_r$, where $t_u$ is the unique
critical point of the function  $\psi_u$, and is a strict maximum point of the function  $\psi_u$ at positive level. Moreover,
\begin{itemize}
  \item [$(a)$] $  \mathcal{P}_r= \mathcal{P}_{-,r}$;
  \item [$(b)$] $\psi_u$ is strictly decreasing and concave on $(t_u, +\infty)$, and $t_u < 0  \Leftrightarrow P(u)<0$.
  \item [$(c)$] The map $u \in S_{a,r}\mapsto t_u \in \mathbb{R}$ is of class $C^1$.
\end{itemize}
  \end{lemma}
  \begin{proof}
  Notice that,  for every $u \in S_{a,r}$,  $\psi_u(s)\rightarrow 0^+$  as $s \rightarrow -\infty$, and $\psi_u(s)\rightarrow -\infty$  as $s \rightarrow +\infty$.    Therefore, $\psi_u(s)$ has a global maximum point at positive level. To show that
this is the unique critical point of $\psi_u(s)$, we observe that  $\psi_u'(s)=0$ if and only if
\begin{align}\label{k95}
f(s)=\mu\gamma_q
\int_{\mathbb{R}^3}|u|^{q}dx
\end{align}
has only one solution,
where
$$
f(s)=e^{(2-q\gamma_q)s}\int_{\mathbb{R}^3}|\nabla u|^2dx+\frac{1}{4} e^{(1-q\gamma_q)s}\int_{\mathbb{R}^3} \int_{\mathbb{R}^3} \frac{|u(x)|^{2}|u(y)|^2} {|x-y|}dxdy
-\gamma_p e^{(p\gamma_p-q\gamma_q)s} \int_{\mathbb{R}^3}|u|^{p}dx.
$$
Indeed, since
\begin{align*}
f'(s)=&(2-q\gamma_q)e^{(2-q\gamma_q)s}\int_{\mathbb{R}^3}|\nabla u|^2dx+\frac{1}{4}(1-q\gamma_q) e^{(1-q\gamma_q)s}\int_{\mathbb{R}^3} \int_{\mathbb{R}^3} \frac{|u(x)|^{2}|u(y)|^2} {|x-y|}dxdy \\
&-\gamma_p(p\gamma_p-q\gamma_q) e^{(p\gamma_p-q\gamma_q)s} \int_{\mathbb{R}^3}|u|^{p}dx\\
=&e^{(1-q\gamma_q)s} f_1(s),
\end{align*}
where
$$
f_1(s)=(2-q\gamma_q)e^{s}\int_{\mathbb{R}^3}|\nabla u|^2dx+\frac{1-q\gamma_q}{4}  \int_{\mathbb{R}^3} \int_{\mathbb{R}^3} \frac{|u(x)|^{2}|u(y)|^2} {|x-y|}dxdy
-\gamma_p(p\gamma_p-q\gamma_q) e^{(p\gamma_p-1)s} \int_{\mathbb{R}^3}|u|^{p}dx,
$$
then, by simple analysis,  we obtain that $f_1(s)\rightarrow \frac{1}{4}(1-q\gamma_q) \int_{\mathbb{R}^3} \int_{\mathbb{R}^3} \frac{|u(y)|^2|u(x)|^{2}} {|x-y|}dydx >0$ as $s\rightarrow-\infty$,   $f_1(s)\rightarrow -\infty$ as $s\rightarrow+\infty$, and there exists $s_0\in \mathbb{R}$ such that $f_1(s) $ is increasing on $(-\infty, s_0)$ and decreasing on $( s_0, +\infty)$.
Hence, there is $s_1>s_0$ such that $f_1(s)>0 $ on $(-\infty, s_1)$ and  $f_1(s)<0 $ on $( s_1, +\infty)$, which shows that $f'(s)>0 $ on $(-\infty, s_1)$ and  $f'(s)<0 $ on $( s_1, +\infty)$. That is,  $f(s) $ is increasing on $(-\infty, s_1)$ and  decreasing   on $( s_1, +\infty)$. Moreover, $f(s)\rightarrow 0^+$ as $s\rightarrow -\infty$ and $f(s)\rightarrow -\infty$ as $s\rightarrow +\infty$. Since $\mu\leq0$, by \eqref{k95}, we get that $\psi_u'(s)=0$
 has only one solution, which implies that $\psi_u(s)$ has the unique critical point, denoted as $t_u$.  In the same way, one can also check that  $\psi_u$ has only one inflection point. Since
 $\psi_u'(s)<0$ if and only if $s>t_u$,
we have that $P(u) =   \psi_u(0) < 0 $ if and only if $t_u < 0$. Finally, for point $(c)$ we argue as in \cite[Lemma 5.3]{2020Soave}.
   \end{proof}

 \begin{lemma}\label{K-Lem4.3}
  Let  $\mu\leq0$, $  2<q<\frac{8}{3}$ and  $\frac{10}{3}<p<6$.
 It holds that
$$
\widetilde{m}_{a,r}:=\inf_{u\in \mathcal{P}_r}\mathcal{ J}(u)>0,
$$
  and there exists $k > 0$ sufficiently small such that
\begin{align*}
  0<\sup_{\overline{D}_{k,r}} \mathcal{ J}<\widetilde{m}_{a,r} \quad\quad\text{and} \quad\quad u\in D_{k,r}\Rightarrow \mathcal{ J}(u), \ P(u)>0,
  \end{align*}
  where $D_{k,r}=D_k\cap \mathcal{H}_r$ and $D_k$ is defined in \eqref{k91}.
  \end{lemma}
\begin{proof}
 Similar to \cite[Lemmas 7.3 and  7.4]{2020Soave}, using \eqref{k96} we get the results.
\end{proof}

Now,  we prove that the functional $\mathcal{J}$ has the Mountain Pass geometry. Let
\begin{align*}
\widetilde{\Gamma}:=\{\zeta\in C([0,1], S_{a,r}): \zeta(0)\in D_{k,r}  \quad \text{and} \quad \mathcal{J}(\zeta(1))<0\},
\end{align*}
and
$$
\widetilde{c}_a:=\inf_{\zeta\in \Gamma}\max_{u\in \zeta([0,1])}\mathcal{J}(u).
$$
Taking $v\in S_{a,r}$, then  $|\nabla (s\star u)|\rightarrow 0^+$  as $s\rightarrow -\infty$, and  $ |\mathcal{J} (s\star u)|\rightarrow  -\infty$  as $s\rightarrow+\infty$. Hence there exist
$s_0<< -1$ and $s_1 >> 1$ such that
\begin{align}\label{k97}
\zeta_v(\tau)=[(1-\tau)s_0+\tau s_1]\star v\in \widetilde{\Gamma},
\end{align}
which implies that $\widetilde{\Gamma}\neq \emptyset $. Next, we claim that
$$
\widetilde{c}_a=\widetilde{m}_{a,r}:=\inf_{u\in\mathcal{P}_r}\mathcal{J}(u)>0.
$$
On the one hand, if $u\in \mathcal{P}_r$, then by \eqref{k97}, $\zeta_u(\tau)$ is a path in  $\widetilde{\Gamma}$. Hence, one obtains from Lemma \ref{K-Lem4.2} that
$$
 \mathcal{J}(u)=\max_{\tau\in[0,1]}\mathcal{J}(\zeta_u(\tau))\geq \widetilde{c}_a,
$$
which gives that
$$
\widetilde{m}_{a,r}\geq \widetilde{c}_a.
$$
On the other hand, for any $\zeta\in \widetilde{\Gamma}$, $P( \zeta(0))>0$ by Lemma \ref{K-Lem4.3}. We claim that $P(\zeta(1))<0$. In fact, since $\psi_{\zeta(1) }(s)>0$ for any $s\in(-\infty, t_{\zeta(1)})$, and $\psi_{\zeta(1)}(0)=\mathcal{J}(\zeta(1))<0$, we deduce that $t_{\zeta(1)}<0 $, which implies that $P(\zeta(1))<0$ due to Lemma \ref{K-Lem4.2}. Hence, there exists $\tau\in(0,1)$ such that $P(\zeta(\tau))=0$ by the continuity. That is, $\zeta(\tau)\in \mathcal{P}_r$. Thus,
\begin{align*}
\max_{t\in[0,1]} J(\zeta(t))\geq J(\zeta(\tau))\geq \inf_{\mathcal{P}_r}J\geq \widetilde{m}_{a,r},
\end{align*}
which shows that $\widetilde{c}_a\geq  \widetilde{m}_{a,r}$ in view of the arbitrariness of  $\zeta$.  This and Lemma \ref{K-Lem4.3} imply   $ \widetilde{c}_a= \widetilde{m}_{a,r}>0$.

Similarly as the proof of Lemmas \ref{K-Lem3.1} and \ref{K-Lem3.2}, we obtain
\begin{lemma}\label{K-Lem4.4}
Let $\mu\leq0$, $q\in (2, \frac{8}{3})$ and $p\in (\frac{10}{3}, 6)$. Then there exists a sequence $\{u_n\}\subset S_{a,r} $ and a constant $\alpha>0$ satisfying
$$
P(u_n)=o(1), \quad  \mathcal{J}(u_n)=\widetilde{c}_a +o(1), \quad \|\mathcal{J}'_{|_{S_a}}(u_n)\|_{\mathcal{H}_{r}^{-1}}= o(1), \quad \|u_n\|\leq \alpha.
$$
\end{lemma}

\begin{lemma}\label{K-Lem4.5}
Let $\mu\leq0$, $\lambda\in \mathbb{R}$,   $q\in (2, \frac{8}{3})$ and $p\in (\frac{10}{3}, 6)$. If $v \in \mathcal{H}$ is a weak solution of
\begin{align}\label{k98}
 - \Delta v+\lambda v+  (|x|^{-1}\ast|v|^2)v=|v|^{p-2}v +\mu|v|^{q-2}v\quad \quad \text{in} \ \mathbb{ R}^3,
\end{align}
then $P(v)=0$. Moreover, if $\lambda\leq 0$, then there exists a constant $\widetilde{a}_0>0$, independent on $\lambda \in \mathbb{R}$,
such that the only solution of \eqref{k98} fulfilling  $|v|_2\leq\widetilde{a}_0$
  is the null function.
\end{lemma}

\begin{proof}
 Since $v$ satisfies
\begin{align}
|\nabla v|_2^2+ \lambda \int_{\mathbb{R}^3} |v|^{2}dx  +\int_{\mathbb{R}^3} \int_{\mathbb{R}^3} \frac{|v(x)|^{2}|v(y)|^2} {|x-y|}dxdy
-\int_{\mathbb{R}^3}|v|^{p}dx-\mu \int_{\mathbb{R}^3}|v|^{q}dx&=0,\label{k100}\\
\frac{1}{2}|\nabla v|_2^2 +\frac{3\lambda }{2} \int_{\mathbb{R}^3}  |v|^2dx
+\frac{5}{4}\int_{\mathbb{R}^3} \int_{\mathbb{R}^3} \frac{|v(x)|^{2}|v(y)|^2} {|x-y|}dxdy-
\frac{3}{p}\int_{\mathbb{R}^3}|v|^{p}dx-\frac{3\mu}{q} \int_{\mathbb{R}^3}|v|^{q}dx&=0,
\label{k101}
\end{align}
then it holds that
\begin{align}\label{k99}
P(v)=\int_{\mathbb{R}^3}|\nabla v|^2dx+\frac{1}{4}\int_{\mathbb{R}^3} \int_{\mathbb{R}^3} \frac{|v(x)|^{2}|v(y)|^2} {|x-y|}dxdy
-\gamma_p\int_{\mathbb{R}^3}|v|^{p}dx
-\mu\gamma_q\int_{\mathbb{R}^3}|v|^{q}dx=0.
 \end{align}
 Suppose that $(\lambda, v)\in \mathbb{R}\times \mathcal{H}$ with $\lambda\leq 0$ and $|v|_2\leq\widetilde{a}_0$ satisfies \eqref{k98}, then
using \eqref{k99}, $\mu\leq0$   and the Gagliardo-Nirenberg inequality \eqref{k10}, one infers
\begin{align*}
 |\nabla v|_2^2 \leq\gamma_p\int_{\mathbb{R}^3}|v|^{p}dx
 \leq  \gamma_p C_p^p|\nabla v|_2^{p\gamma_p} |v|_2^{p(1-\gamma_p)},
\end{align*}
which shows that
\begin{align}\label{kk48}
 |\nabla v|_2^{p\gamma_p-2}\geq \frac{1}{\gamma_pC_p^p |v|_2^{p(1-\gamma_p)} }.
\end{align}
 Then, \eqref{kk48} implies that  when $|v|_2$ is small enough,  $ |\nabla v|_2$ must be large.
Moreover, by eliminating $\int_{\mathbb{R}^3}|v|^{p}dx$ from  \eqref{k100} and \eqref{k101}, since $\mu\leq0$   and the Gagliardo-Nirenberg inequality \eqref{k10}, one has
  \begin{align*}
  \lambda \gamma_p \int_{\mathbb{R}^3} |v|^{2}dx&=(1-\gamma_p)|\nabla v|_2^2+(\frac{1}{4}-\gamma_p) \int_{\mathbb{R}^3} \int_{\mathbb{R}^3} \frac{|v(x)|^{2}|v(y)|^2} {|x-y|}dxdy+\mu ( \gamma_p-\gamma_q) \int_{\mathbb{R}^3}|v|^{q}dx\\
  &\geq(1-\gamma_p)|\nabla v|_2^2 - (\gamma_p-\frac{1}{4})C_{\frac{12}{5}}^{\frac{12}{5}} |\nabla v|_2|v|_2^3-|\mu| ( \gamma_p-\gamma_q) C_q^q |\nabla v|_2^{q\gamma_q}|v|_2^{q(1-\gamma_q)},
\end{align*}
which gives a contradiction  by $\lambda\leq 0$ and \eqref{kk48} when $|v|_2$  is small enough.
Thus we finish the proof.
\end{proof}

 \noindent {\bf{Proof of Theorem \ref{K-TH5}}.}
From Lemma \ref{K-Lem4.4}, we get a sequence  $\{u_n\}\subset S_{a,r}$ satisfying
$$
P(u_n)=o(1), \quad  \mathcal{J}(u_n)=\widetilde{c}_a +o(1), \quad \|\mathcal{J}'_{|_{S_a}}(u_n)\|_{\mathcal{H}_{r}^{-1}}= o(1).
$$
Similar to the proof of Lemma  \ref{K-Lem2.6},      we deduce that $\{u_n\}$ is bounded in $\mathcal{H}_r$, and there exists $0\neq \widetilde{u}\in \mathcal{H}_r$ such that, up to a subsequence,
 \begin{align*}
   & u_n\rightharpoonup \widetilde{u} \quad\quad \text{in}\ \mathcal{H}_r;\\
   & u_n\rightarrow \widetilde{u}\quad\quad \text{in}\ L^t(\mathbb{R}^3)\ \text{with}\ t\in(2, 6);\\
   & u_n\rightarrow \widetilde{u} \quad\quad  a.e \ \text{in}\ \mathbb{R}^3.
\end{align*}
Since $\mathcal{J}'_{|_{S_{a,r}}}(u_n)\rightarrow 0$,  then there exists $\widetilde{\lambda}_n\in \mathbb{R} $ satisfying
\begin{align*} 
-\Delta u_n+\widetilde{\lambda}_n u_n+  (|x|^{-1}\ast |u_n|^2)u_n=\mu |u_n|^{q-2}u_n+|u_n|^{p-2}u_n+o(1).
\end{align*}
 Multiplying the above equation by $u_n$ and integrating in $\mathbb{R}^3$,
\begin{align} \label{kk16}
\widetilde{ \lambda}_n a^2= -|\nabla u_n|_2^2-\int_{\mathbb{R}^3} \int_{\mathbb{R}^3} \frac{|u_n(x)|^{2}|u_n(y)|^2} {|x-y|}dxdy
+\int_{\mathbb{R}^3}|u_n|^{p}dx+\mu \int_{\mathbb{R}^3}|u_n|^{q}dx+o( \|u_n\|).
\end{align}
Since $\{u_n\}\subset \mathcal{H}_r$ is bounded, then it follows from \eqref{kk16} that  $\{\widetilde{\lambda}_n\}$  is bounded in $\mathbb{R}$. Hence, there is $\widetilde{\lambda}\in \mathbb{R}$ satisfying $\widetilde{\lambda}_n\rightarrow \widetilde{\lambda}$ as $n\rightarrow\infty$. Then we have
\begin{align*} 
-\Delta \widetilde{u}+\widetilde{\lambda} \widetilde{u}+ (|x|^{-1}\ast |\widetilde{u}|^2) \widetilde{u}=\mu |\widetilde{u}|^{q-2}\widetilde{u}+|\widetilde{u}|^{p-2}\widetilde{u}.
\end{align*}
Then, by Lemma  \ref{K-Lem4.5}, there exists a $\widetilde{a}_0>0$   such that  $\widetilde{\lambda}>0$ if $a\in (0, \widetilde{a}_0)$.
Similar to  the Step 3 of Lemma   \ref{K-Lem2.6}, we deduce that $\|u_n-\widetilde{u}\|\rightarrow 0$ in $\mathcal{H}_r$ as $n\rightarrow\infty$. Thus, we infer that $(\widetilde{\lambda}, \widetilde{u} )\in \mathbb{R}^+  \times \mathcal{H}_r$ solves Eq.~\eqref{k1},  where $\mathcal{J}(\widetilde{u})=\widetilde{c}_a$. So, we complete the proof.$\hfill\Box$

 \section{The  case   $q\in(2, \frac{8}{3} )$ and $p= \frac{10}{3}$}\label{sec5}
In this section, we consider the case $q\in(2, \frac{8}{3} )$ and $p=\overline{p}=\frac{10}{3}$, and then we give some existence and nonexistence results for Eq.~\eqref{k1}. Defined
$$
E_0(u)=\frac{1}{2}\int_{\mathbb{R}^3}|\nabla u|^2dx-\frac{1}{p}\int_{\mathbb{R}^3}|u|^{p}dx.
$$
By Gagliardo-Nirenberg inequality \eqref{k10}, it holds that, for any $u\in S_a$ and $p=\overline{p}=\frac{10}{3}$,
$$
E_0(u)=\frac{1}{2}\int_{\mathbb{R}^3}|\nabla u|^2dx-\frac{1}{p}\int_{\mathbb{R}^3}|u|^{p}dx
\geq \Big(\frac{1}{2}- \frac{C_p^p}{p}a^{\frac{4}{3}}\Big) \int_{\mathbb{R}^3}|\nabla u|^2dx.
$$
From the above analysis, when $0<a\leq a^*=(\frac{\overline{p}}{2 C_{\overline{p}}^{\overline{p}}})^{\frac{3}{4}}$, we have $E_0(u)\geq 0$.\\

 \noindent{\bf{Proof of Theorem \ref{K-TH7}}.} Let $q\in(2, \frac{8}{3} )$ and $p=\overline{p}=\frac{10}{3}$.
For any $\mu<0$, if $0<a\leq a^*$, assume that $u$ is a solution of
Eq.~\eqref{k1}, then $P(u)=0$, that is
\begin{align*}
P(u)=\int_{\mathbb{R}^3}|\nabla u|^2dx+\frac{1}{4}\int_{\mathbb{R}^3} \int_{\mathbb{R}^3} \frac{|u(x)|^{2}|u(y)|^2} {|x-y|}dxdy
-\gamma_p\int_{\mathbb{R}^3}|u|^{p}dx
-\mu\gamma_q\int_{\mathbb{R}^3}|u|^{q}dx=0,
 \end{align*}
which and   $\inf_{S_a} E_0 \geq 0$  due to $a\leq a^*$   show  that
$$
0>\mu\gamma_q\int_{\mathbb{R}^3}|u|^{q}dx-\frac{1}{4}\int_{\mathbb{R}^3} \int_{\mathbb{R}^3} \frac{|u(x)|^{2}|u(y)|^2} {|x-y|}dxdy=2E_0(u)\geq2 \inf_{S_a}E_0(u)\geq 0.
$$
This is a contradiction. Hence,  Eq.~\eqref{k1} has no solution at all. Clearly,   $\inf_{S_a} \mathcal{J}=0$.

For any $\mu\in \mathbb{R}$, if $a> a^*$, then there exists $w\in S_a$ such that $E_0(w)<0$. Hence, since $q\gamma_q<1$,
\begin{align*}
 \mathcal{J}(s\star w )
  = e^{2s} E_0(w)+\frac{e^{s}}{4} \int_{\mathbb{R}^3} \int_{\mathbb{R}^3} \frac{|w(x)|^{2}|w(y)|^2} {|x-y|}dxdy-\mu\frac{e^{q\gamma_qs}}{q}
\int_{\mathbb{R}^3}|w|^{q}dx
 \rightarrow -\infty
\end{align*}
as $s\rightarrow +\infty$. Thus, $\inf_{S_a}\mathcal{J}=-\infty $.

For any $\mu>0$, if $0<a\leq a^*$, we first claim  that $\mathcal{J}$ is bounded below on $S_a$ and $e(a)<0$. In fact, since in this case  $\inf_{S_a} E_0 \geq 0$, we deduce that, for any $u\in S_a$,
\begin{align*}
 \mathcal{J}(s\star u )
 & = e^{2s} E_0(u)+\frac{e^{s}}{4} \int_{\mathbb{R}^3} \int_{\mathbb{R}^3} \frac{|u(x)|^{2}|u(y)|^2} {|x-y|}dxdy-\mu\frac{e^{q\gamma_qs}}{q}
\int_{\mathbb{R}^3}|u|^{q}dx\\
&\geq \frac{e^{s}}{4} \int_{\mathbb{R}^3} \int_{\mathbb{R}^3} \frac{|u(x)|^{2}|u(y)|^2} {|x-y|}dxdy-\mu\frac{e^{q\gamma_qs}}{q}
\int_{\mathbb{R}^3}|u|^{q}dx.
\end{align*}
In view of $q\gamma_q<1$, we get   $ \mathcal{J}(s\star u )\rightarrow+\infty$ as $s\rightarrow+\infty$. Hence, $\mathcal{J}$ is coercive on $S_a$, and $e(a):= \inf_{S_a} \mathcal{J} >-\infty$. Furthermore,  due to $q\gamma_q<1$, for any $u\in S_a$,  $\mathcal{J}(s\star u )\rightarrow 0^-$ as $s\rightarrow-\infty$. Hence, $e(a)<0$. In particular, we further prove that
\begin{align*}
 \text{when}\  0<a<a^*,  \ e(a) \text{ is attained by a real-valued
function}\ u\in S_a.
 \end{align*}
In other words, we have to verify that, taking a minimizing sequence $\{u_n\}\subset S_a$ of $e(a)$, along a subsequence, there exists $u\in S_a$  such that $u_n \rightarrow u$  in $\mathcal{H}$  up to translations. Indeed, similar to Sec. \ref{sec1}, we    change the definition of $g_u$ to the following statement
\begin{definition}
  Let $u\in \mathcal{H}\backslash \{0\} $. A continuous path $g_u: \theta\in \mathbb{R}^+\mapsto g_u(\theta)\in \mathcal{H}$ such that $g_u(1)=u$  is said to be a scaling path of $u$ if
 \begin{align*}
 \Theta_{g_u}(\theta):=|g_u(\theta)|_2^2|u|_2^{-2}\ \ \text{is differentiable}  \quad\quad \text{and}\quad \quad  \Theta_{g_u}'(1)\neq 0.
 \end{align*}
 We denote with $\mathcal{G}_u$ the set of the scaling paths of $u$.
\end{definition}
 Here,   if the following statements hold: for any $a\in (0, a^*)$ and $\mu>0$,
 \begin{itemize}
   \item [(i)] the map $a \mapsto e(a) $  is continuous;
   \item [(ii)]
 $\lim_{a\rightarrow 0}\frac{e(a) }{a^2}=0$,
 \end{itemize}
 then similar to Sec. \ref{sec1} and  the proof of Theorem \ref{K-TH1}, we can obtain the desire result.  That is,
  there exists $u\in S_a$ such that $ \mathcal{J}(u)=e(a)$.  Since if $\{u_n\}\subset S_a$  is a minimizing sequence of $e(a)$, then $\{|u_n|\}\subset S_a$  is also a minimizing sequence of $e(a)$. Hence,
 there is  a real-valued non-negative function $u\in S_a$ such that $ \mathcal{J}(u)=e(a)$.

  Now, we verify (i).  Let $a \in  (0, a^*)$ be arbitrary and $\{a_n\} \subset(0, \overline{a}_0)$ be such that $a_n \rightarrow a$ as $n\rightarrow\infty$. From the definition of $e(a_n)$
and since $e(a_n) < 0$, for any  $\varepsilon> 0$ sufficiently small, there exists $u_n \in S_{a_n}$ such that $\mathcal{J}(u_n)\leq e(a_n)+\varepsilon$ and $\mathcal{J}(u_n)<0$.
Taking $n$ large enough, we have $a_n \leq a+\varepsilon<a^*$.
Since  $q\gamma_q<1 $ and
\begin{align}\label{k113}
0>\mathcal{J}(u_n)\geq& \frac{1}{2}|\nabla u_n|_2^2-\mu \frac{C_q^q}{q} a_n^{q(1-\gamma_q)}|\nabla u_n|_2^{q\gamma_q}-\frac{C_{\overline{p}}^{\overline{p}}}{{\overline{p}}} (a_n)^{ \frac{4}{3}}|\nabla u_n|_2^{2}\nonumber\\
\geq&\Big( \frac{1}{2}- \frac{C_{\overline{p}}^{\overline{p}}}{{\overline{p}}} (a_n)^{ \frac{4}{3}}\Big) |\nabla u_n|_2^{2} -\mu \frac{C_q^q}{q} a_n^{q(1-\gamma_q)}|\nabla u_n|_2^{q\gamma_q},
\end{align}
we deduce that $\{u_n\} $ is bounded in $D^{1,2}(\mathbb{R}^3)$.
On the one hand,  setting $v_n := \frac{a}{a_n} u_n\in S_a$, by the boundedness of $\{u_n\}$ and $a_n\rightarrow a$ as $n\rightarrow\infty$,
\begin{align}\label{k112}
e(a)\leq \mathcal{J}(v_n)=&\mathcal{J}(u_n)+\frac{1}{2}\Big[\big(\frac{a}{a_n}\big)^2-1\Big]|\nabla u_n|_2^2+\frac{1}{4}\Big[\big(\frac{a}{a_n}\big)^4
-1\Big]\int_{\mathbb{R}^3} \int_{\mathbb{R}^3} \frac{|u_n(x)|^2|u_n(y)|^2} {|x-y|}dxdy\nonumber\\
 &-\frac{1}{p}\Big[\big(\frac{a}{a_n}\big)^p-1\Big]|  u_n|_p^p
 -\mu\frac{1}{q}\Big[\big(\frac{a}{a_n}\big)^q-1\Big]|  u_n|_q^q\nonumber\\
 =&\mathcal{J}(u_n)+o(1)\nonumber\\
 \leq & e(a_n)+o(1).
\end{align}
 On the other hand, let $u \in S_a$ be such that $\mathcal{J}(u)\leq e(a)+\varepsilon$  and  $\mathcal{J}(u)<0.$
Let $w_n=\frac{a_n}{a}u\in S_{a_n}$, then,
 \begin{align}\label{k111}
e(a_n)\leq \mathcal{J}(w_n)=\mathcal{J}(u)+[\mathcal{J}(w_n)-\mathcal{J}(u)]\leq e(a) +o(1).
\end{align}
 Thus, we get from \eqref{k112} and \eqref{k111} that $e(a_n)\rightarrow e(a)$ as $n\rightarrow\infty$ for any $a\in(0, a^*)$.

Next, we show that (ii) holds. Define
 \begin{align*}
\mathcal{I}(u)=\frac{1}{2}\int_{\mathbb{R}^3}|\nabla u|^2dx
-\frac{1}{\overline{p}}\int_{\mathbb{R}^3}|u|^{\overline{p}}dx-\mu\frac{1}{q}\int_{\mathbb{R}^3}|u|^{q}dx,
\end{align*}
and $\overline{e}(a):=\inf_{S_a}  \mathcal{I}$, where $a\in(0, a^*)$.
Then $\mathcal{J} (u)\geq \mathcal{I}(u) $ for any $u\in S_a$. Hence, $0> e(a)\geq \overline{e}(a),$
implies that $ \frac{\overline{e}(a)}{a^2}\leq  \frac{ e(a)}{a^2} <0$. Hence, we only need to show that  $\lim_{a\rightarrow0}\frac{\overline{e}(a)}{a^2}=0$.   According to \cite{2020Soave}, when $a\in(0, a^*)$, $\mu>0$, $q\in(2, \frac{8}{3})$ and $p=\overline{p}=\frac{10}{3}$, we obtain that there exists $\overline{u}_{a}\in S_a$ such that $ \mathcal{I} (\overline{u}_{a})=\overline{e}(a)<0 $.  Then, similar to \eqref{k113}, the sequence $\{\overline{u}_{a}\}_{a>0}$ is bounded in $D^{1,2}(\mathbb{R}^3)$. Since the minimizer $\overline{u}_{a}$ satisfies
\begin{align}\label{kk79}
-\Delta \overline{u}_{a}-|\overline{u}_{a}|^{\overline{p}-2}\overline{u}_{a}
-\mu|\overline{u}_{a}|^{q-2}\overline{u}_{a}=\omega_a \overline{u}_{a},
\end{align}
 where $\omega_a$ is the Lagrange multiplier associated to the minimizer,
then
\begin{align*}
 \frac{\omega_a}{2}=\frac{|\nabla  \overline{u}_{a}|_2^2-| \overline{u}_{a}|_{\overline{p}}^{\overline{p}}-\mu |\overline{u}_{a}|_q^q}{2| \overline{u}_{a}|_2^2}\leq\frac{\frac{1}{2}|\nabla  \overline{u}_{a}|_2^2-\frac{1}{{\overline{p}}}|\overline{u}_{a}|_{\overline{p}}^{\overline{p}}-\mu\frac{1}{q} | \overline{u}_{a}|_q^q}{| \overline{u}_{a}|_2^2}=\frac{\overline{\mathcal{I}}(\overline{u}_{a} )}{a^2}< 0.
\end{align*}
If  $\lim_{a\rightarrow0} \omega_a=0$, we get the conclusion.
 To show that $\lim_{a\rightarrow0} \omega_a=0$,    by contradiction,  assuming that there exists a sequence $a_n \rightarrow 0$ such that $\omega_{a_n} <-\alpha$ for some $\alpha\in(0,1)$. Since the minimizers $\overline{u}_{a_n}\in S_{a_n}$ satisfies Eq.~\eqref{kk79},  one has
\begin{align}
 |\nabla\overline{u}_{a_n} |_2^2-\omega_{a_n} |\overline{u}_{a_n} |_2^2-|\overline{u}_{a_n}|_{\overline{p}}^{\overline{p}}-\mu |\overline{u}_{a_n}|_q^q&=0,\label{kkkk80}\\
 |\nabla\overline{u}_{a_n} |_2^2-\gamma_{\overline{p}}|\overline{u}_{a_n}|_{\overline{p}}^{\overline{p}}-\mu\gamma_q|\overline{u}_{a_n}|_q^q&=0.\label{kkk81}
\end{align}
It follows from  $\eqref{kkkk80}-\frac{1}{2}\eqref{kkk81}$  that
 \begin{align*}
C\|\overline{u}_{a_n}\|^2\leq \frac{1}{2}|\nabla\overline{u}_{a_n} |_2^2+\alpha  |\overline{u}_{a_n} |_2^2
&\leq(1-\frac{1}{2}\gamma_{\overline{p}})| \overline{u}_{a_n} |_{\overline{p}}^{\overline{p}}+\mu(1-\frac{1}{2}\gamma_q)| \overline{u}_{a_n} |_q^q \\
  & \leq C(1-\frac{1}{2}\gamma_{\overline{p}})\| \overline{u}_{a_n} \|^{\overline{p}}+\mu C(1-\frac{1}{2}\gamma_q)\| \overline{u}_{a_n} \|^q,
\end{align*}
 which implies that $\|\overline{u}_{a_n}\|\geq C $ due to ${\overline{p}}, q>2$ and $\gamma_{\overline{p}}, \gamma_q<1$. It gives that $ |\nabla\overline{u}_{a_n} |_2>C$ for $n$ large enough in virtue of $ |\overline{u}_{a_n} |_2=a_n\rightarrow 0$ as $n\rightarrow\infty$.
 Moreover, for any $u\in S_a$,
 \begin{align*}
\mathcal{I}(u)=&\frac{1}{2}\int_{\mathbb{R}^3}|\nabla u|^2dx
-\frac{1}{\overline{p}}\int_{\mathbb{R}^3}|u|^{\overline{p}}dx-\mu\frac{1}{q}\int_{\mathbb{R}^3}|u|^{q}dx\\
\geq& \frac{1}{2}|\nabla u|_2^2-\mu \frac{C_q^q}{q} a^{q(1-\gamma_q)}|\nabla u|_2^{q\gamma_q}-\frac{C_{\overline{p}}^{\overline{p}}}{\overline{p}} a^{\overline{p}(1-\gamma_{\overline{p}})}|\nabla u|_2^{2}.
\end{align*}
Thus, for $n$ large enough,
 \begin{align*}
0\geq  \mathcal{I}(\overline{u}_{a_n})
\geq\Big(\frac{1}{2}-\mu \frac{C_q^q}{q} a_n^{q(1-\gamma_q)}|\nabla\overline{u}_{a_n}|_2^{q\gamma_q-2}-\frac{C_{\overline{p}}^{\overline{p}}}{{\overline{p}}} a_n^{{\overline{p}}(1-\gamma_{\overline{p}})} \Big)|\nabla\overline{u}_{a_n}|_2^2
\geq \frac{1}{4}C>0,
\end{align*}
which is a contradiction. Therefore, $\lim_{a\rightarrow0} \omega_a=0$, which shows that  $\lim_{a\rightarrow0}\frac{\overline{e}(a)}{a^2}=0$.

In order to complete the proof of Theorem \ref{K-TH7}, we finally need to prove that when $\mu>0,  a=a^*, q\in(2, \frac{12}{5} ]$ and $p=\overline{p}=\frac{10}{3}$, $e(a^*)$ is attained by a real-valued
function $ u\in S_a$. In fact,  let us consider a minimizing sequence $\{|v_n|\}\subset S_a$ for $e(a^*)$,   the Ekeland's
variational principle yields in a standard way the existence of a new minimizing sequence $\{u_n\}\subset S_a$
  for $e(a^*)$, with the property that $ \mathcal{J}_{|_{S_a}}'(u_n)\rightarrow 0$ in $\mathcal{H}_{r}^{-1}$ and $u_n$ is a real-value nonnegative function for every $n\in \mathbb{N}$. Since $\mathcal{J}$ is coercive on $S_a$, $\{u_n\}$ is bounded in  $\mathcal{H}_{r}$. Then there exists $u\in \mathcal{H}_r$ such that, up to a subsequence, as $n\rightarrow\infty$,
 \begin{align*}
   & u_n\rightharpoonup u \quad\quad \text{in}\ \mathcal{H}_r;\\
   & u_n\rightarrow u\quad\quad \text{in}\ L^t(\mathbb{R}^3)\ \text{with}\ t\in(2, 6);\\
   & u_n\rightarrow u \quad\quad  a.e \ \text{in}\ \mathbb{R}^3.
\end{align*}
Since $\mathcal{J}'_{|_{S_a}}(u_n)\rightarrow 0$,  then there exists $\lambda_n\in \mathbb{R} $ satisfying
\begin{align*} 
-\Delta u_n+\lambda_n u_n+ (|x|^{-1}\ast |u_n|^2)u_n=\mu |u_n|^{q-2}u_n+|u_n|^{\overline{p}-2}u_n+o(1).
\end{align*}
 Multiplying the above equation by $u_n$ and integrating in $\mathbb{R}^3$,
\begin{align} \label{kkk16}
 \lambda_n a^2= -|\nabla u_n|_2^2-\int_{\mathbb{R}^3} \int_{\mathbb{R}^3} \frac{|u_n(x)|^{2}|u_n(y)|^2} {|x-y|}dxdy
+\int_{\mathbb{R}^3}|u_n|^{\overline{p}}dx+\mu \int_{\mathbb{R}^3}|u_n|^{q}dx+o( \|u_n\|).
\end{align}
Since $\{u_n\}\subset \mathcal{H}_r$ is bounded, then it follows from \eqref{kkk16} that  $\{\lambda_n\}$  is bounded in $\mathbb{R}$. Hence, there is $\lambda\in \mathbb{R}$ satisfying $\lambda_n\rightarrow \lambda$ as $n\rightarrow\infty$. Then we have
\begin{align} \label{kkkk20}
-\Delta u+\lambda u+ (|x|^{-1}\ast |u|^2)u=\mu |u|^{q-2}u+|u|^{\overline{p}-2}u.
\end{align}
Clearly, $u\neq0$. If not, assume that $u=0$, then by Lemma \ref{gle4} and the fact that $\mathcal{H}_r\hookrightarrow L^t(\mathbb{R}^3 )$ with $t\in (2,6)$ is compact, we deduce
\begin{align*}
e(a^*)=\lim_{n\rightarrow\infty}\mathcal{J}(u_n)
 =\lim_{n\rightarrow\infty} \frac{1}{2}  \int_{\mathbb{R}^3} |\nabla u_n|^2dx\geq 0,
\end{align*}
which is   contradict with $e(a^*)<0$. Besides, we show that $\lambda>0$. It follows from \eqref{kkkk20} that
\begin{align}
|\nabla u|_2^2+ \lambda \int_{\mathbb{R}^3} |u|^{2}dx  +\int_{\mathbb{R}^3} \int_{\mathbb{R}^3} \frac{|u(x)|^{2}|u(y)|^2} {|x-y|}dxdy
-\int_{\mathbb{R}^3}|u|^{\overline{p}}dx-\mu \int_{\mathbb{R}^3}|u|^{q}dx&=0,\label{kk17}\\
 |\nabla u|_2^2+ \frac{1}{4}\int_{\mathbb{R}^3} \int_{\mathbb{R}^3} \frac{|u(x)|^{2}|u(y)|^2} {|x-y|}dxdy-\gamma_p \int_{\mathbb{R}^3}|u|^{\overline{p}}dx-\mu \gamma_q \int_{\mathbb{R}^3}|u|^{q}dx&=0.\label{kk18}
\end{align}
  We argue by contradiction   assuming that $\lambda\leq 0$. Then it follows from  \eqref{kk17} and \eqref{kk18} that, by eliminating $\int_{\mathbb{R}^3} \int_{\mathbb{R}^3} \frac{|u(y)|^2|u(x)|^{2}} {|x-y|}dydx$,
\begin{align*}
\frac{3}{4}|\nabla u|_2^2-\frac{\lambda}{4}  \int_{\mathbb{R}^3} |u|^{2}dx
+(\frac{1}{4}-\gamma_{\overline{p}})\int_{\mathbb{R}^3}|u|^{\overline{p}}dx+\mu (\frac{1}{4}-\gamma_q) \int_{\mathbb{R}^3}|u|^{q}dx=0,
\end{align*}
which implies that
\begin{align}\label{k103}
\frac{3}{4}|\nabla u|_2^2
&\leq
 \big(\gamma_{\overline{p}}-\frac{1}{4}\big)\int_{\mathbb{R}^3}|u|^{\overline{p}}dx+\mu \big(\gamma_q-\frac{1}{4}\big) \int_{\mathbb{R}^3}|u|^{q}dx\nonumber\\
& \leq \big(\gamma_{\overline{p}}-\frac{1}{4}\big)\int_{\mathbb{R}^3}|u|^{\overline{p}}dx
 \leq  \big(\gamma_{\overline{p}}-\frac{1}{4}\big) C_{\overline{p}}^{\overline{p}}|\nabla u|_2^{2} (a^*)^{ \frac{4}{3}}
\end{align}
by using the fact that $\mu>0$, $q\in(2, \frac{12}{5}]$, $\overline{p}=\frac{10}{3}$, $ |u|_2\leq a^*$ and  \eqref{k10}.  In view of \eqref{k103}, we infer that $a^*\geq\big(\frac{15}{7 C_{\overline{p}}^{\overline{p}}}\big)^\frac{3}{4}$,
which is contradict with the definition of $a^* $. Hence,   $\lambda>0$.
  The following is similar to    the Step 3 of Lemma    \ref{K-Lem2.6},
 we can obtain that $\|u_n-u\|\rightarrow 0$ in $\mathcal{H}_r$ as $n\rightarrow\infty$. Then, we infer that $( \lambda, u)\in \mathbb{R}^+\times \mathcal{H}_r$ is a solution for Eq.~\eqref{k1}, where $\mathcal{J}(u)=e(a^*)$ and $u\in S_a$ is a nonnegative real-value function. So, we complete the proof.
$\hfill\Box$

\section{The case $\mu>0$, $q=\frac{10}{3}$ and   $p\in (\frac{10}{3}, 6)$}\label{sec7}
In this part, we are going to solve  the case $\mu>0$, $q=\overline{q}=\frac{10}{3}$ and   $p\in (\frac{10}{3}, 6)$.
 \begin{lemma}\label{KJ-Lem2.2}
  Let $\mu>0$, $q=\overline{q}=\frac{10}{3}$ and   $p\in (\frac{10}{3}, 6)$. Then $ \mathcal{P}_0=\emptyset$, and $\mathcal{P}$ is a smooth manifold of codimension 2 in $\mathcal{H}$.
  \end{lemma}
 \begin{proof}
  Suppose on the contrary that there exists $u\in  \mathcal{P}_0$. By \eqref{k8} and \eqref{k4} one has
 \begin{align}
P(u)=\int_{\mathbb{R}^3}|\nabla u|^2dx+\frac{1}{4}\int_{\mathbb{R}^3} \int_{\mathbb{R}^3} \frac{|u(x)|^{2}|u(y)|^2} {|x-y|}dxdy
-\gamma_p\int_{\mathbb{R}^3}|u|^{p}dx
-\mu\gamma_q\int_{\mathbb{R}^3}|u|^{\overline{q}}dx=0,\label{k205}\\
\psi_u''(0)=2\int_{\mathbb{R}^3}|\nabla u|^2dx+\frac{1}{4}\int_{\mathbb{R}^3} \int_{\mathbb{R}^3} \frac{|u(x)|^{2}|u(y)|^2} {|x-y|}dxdy
-p\gamma_p^2\int_{\mathbb{R}^3}|u|^{p}dx
-\mu q \gamma_q^2\int_{\mathbb{R}^3}|u|^{\overline{q}}dx=0.\label{k206}
 \end{align}
  By eliminating  $|\nabla u|_2^2$ from \eqref{k205}  and   \eqref{k206}, we get
  \begin{align*}
0
= (p \gamma_p-2)\gamma_p\int_{\mathbb{R}^3}|u|^{p}dx
+\frac{1}{4}\int_{\mathbb{R}^3} \int_{\mathbb{R}^3} \frac{|u(x)|^{2}|u(y)|^2} {|x-y|}dxdy,
 \end{align*}
  which implies that $|u|_p=0$. This is not possible
since $u \in  S_a$.  The remain is similar to \cite[Lemma 5.2]{2020Soave}.
     Thus we complete the proof.
  \end{proof}

 \begin{lemma}\label{K-J1}
Let  $\mu>0$, $q=\overline{q} =\frac{10}{3}$ and $p\in (\frac{10}{3}, 6)$. Assume that \eqref{k201} holds.
   For any $u\in S_a$,  there exists a unique $t_u \in \mathbb{R}$ such that $t_u\star u\in \mathcal{P}$. $t_u$ is the unique
critical point of the function $\psi_u$, and is a strict maximum point at positive level. Moreover:
\begin{itemize}
  \item [(i)]  $\mathcal{P}=\mathcal{P}_-$;
  \item [(ii)]  the function $\psi_u$ is strictly decreasing and concave on $(t_u,+\infty)$. In particular,  $t_u < 0\Leftrightarrow P(u) <0$;
    \item [(iii)]   the map $u\in S_a\mapsto t_u\in \mathbb{R}$ is of class $C^1$.
\end{itemize}
  \end{lemma}

 \begin{proof}
Since
\begin{align*}
 \psi_u(s)
&=\frac{e^{2s}}{2}\int_{\mathbb{R}^3}|\nabla u|^2dx+\frac{e^{s}}{4} \int_{\mathbb{R}^3} \int_{\mathbb{R}^3} \frac{|u(x)|^{2}|u(y)|^2} {|x-y|}dxdy
-\frac{e^{p\gamma_ps}}{p}\int_{\mathbb{R}^3}|u|^{p}dx
-\frac{\mu e^{2s}}{\overline{q}}
\int_{\mathbb{R}^3}|u|^{\overline{q}}dx \\
&= \left(\frac{1}{2}\int_{\mathbb{R}^3}|\nabla u|^2dx-\frac{\mu }{\overline{q}}
\int_{\mathbb{R}^3}|u|^{\overline{q}}dx\right) e^{2s}+ \frac{e^{s}}{4} \int_{\mathbb{R}^3} \int_{\mathbb{R}^3} \frac{|u(x)|^{2}|u(y)|^2} {|x-y|}dxdy
-\frac{e^{p\gamma_ps}}{p}\int_{\mathbb{R}^3}|u|^{p}dx
\end{align*}
then, by the fact that $s\star u\in \mathcal{P}\Leftrightarrow \psi_u'(s)=0$,  to prove existence and uniqueness of $t_u$, together with monotonicity
and convexity of $\psi_u$, we have only to show that the term inside the brackets is positive. This is
clearly satisfied, since
\begin{align*}
\frac{1}{2}\int_{\mathbb{R}^3}|\nabla u|^2dx-\frac{\mu }{\overline{q}}
\int_{\mathbb{R}^3}|u|^{\overline{q}}dx\geq \left( \frac{1}{2}-\frac{\mu}{\overline{q}}C_{\overline{q}}^{\overline{q}}a^{\frac{4}{3}}\right)|\nabla u|_2^2>0
\end{align*}
by the assumption \eqref{k201}.
The remains are similar to the proof of \cite[Lemmas 5.3 and   6.1 ]{2020Soave}. Thus, we complete the proof.
  \end{proof}

\begin{lemma} \label{K-J2}
 Assume that $\mu>0$,
  $q=\overline{q}= \frac{10}{3} $, $p\in (\frac{10}{3}, 6)$ and  \eqref{k201} hold. Then
$ \sigma_{\overline{q}}:=\inf\limits_{\mathcal{P}} \mathcal{J}(u)>0$.
\end{lemma}

\begin{proof}
Since $P(u)=0$, then
\begin{align*}
 |\nabla u|_2^2+ \frac{1}{4}\int_{\mathbb{R}^3} \int_{\mathbb{R}^3} \frac{|u(x)|^{2}|u(y)|^2} {|x-y|}dxdy-\gamma_p \int_{\mathbb{R}^3}|u|^{p}dx-\mu \gamma_{\overline{q}} \int_{\mathbb{R}^3}|u|^{\overline{q}}dx=0,
\end{align*}
which and   the Gagliardo-Nirenberg inequality \eqref{k10}  imply that
\begin{align}\label{k200}
\left(1-\frac{3}{5}\mu C_{\overline{q}}^{\overline{q}}a^{\frac{4}{3}}\right)|\nabla u|_2^2\leq C_{p}^p \gamma_pa^{p(1-\gamma_p)} |\nabla u|_2^{p\gamma_p}.
\end{align}
We deduce from the above and \eqref{k201} that $\inf_{\mathcal{P}} |\nabla u|_2 >0.$
Hence,
\begin{align*}
   \mathcal{J} (u)=\mathcal{J} (u)-\frac{1}{p\gamma_p} P(u)
 \geq \frac{1}{2}\Big(1- \frac{2}{p\gamma_p}\Big)\big(1-\frac{3}{5}  C_{\overline{q}}^{\overline{q}}\mu a^{\frac{4}{3}}\big)|\nabla u|_2^2>0,
\end{align*}
where the assumption \eqref{k201} is applied.
Thus, the thesis follows.
 \end{proof}

 \begin{lemma} \label{K-J3}
Assume that $\mu>0$,
  $q=\overline{q}= \frac{10}{3} $, $p\in (\frac{10}{3}, 6)$ and \eqref{k201} hold. Then
there exists $k > 0$ sufficiently small such that
\begin{align*}
  0<\sup_{ \overline{A}_k\cap S_a} \mathcal{J}< \sigma_{\overline{q}}\quad \text{and} \quad u\in \overline{A}_k\cap S_a\Rightarrow \mathcal{J}(u), P(u)>0,
\end{align*}
where $A_k:=\{u\in \mathcal{H}: |\nabla u|_2\leq k\}$.
\end{lemma}
  \begin{proof}
  The proof is similar to \cite[Lemma 6.4]{2020Soave}, we omit it here.
 \end{proof}

In what follows we address the case $\mu>0$, $q= \overline{q}=\frac{10}{3}  $ and $p\in (\frac{10}{3}, 6)$
  and consider the problem in $\mathcal{H}_r$. We prove the existence of ground state of mountain pass type at level $\sigma_{\overline{q},r}:=  \inf\limits_{\mathcal{P}_r} \mathcal{J}(u)>0$, where $\mathcal{P}_r=\mathcal{H}_r \cap \mathcal{P}$.
Let $S_{a,r}= \mathcal{H}_r\cap S_a  $ and $\mathcal{P}_{\pm,r}= \mathcal{H}_r\cap\mathcal{P}_{\pm} $.

 Firstly, we verify the Mountain Pass geometry of $\mathcal{J} $ on manifold $S_{a,r}$. Let
\begin{align*}
\Gamma_{\overline{q}}:=\big\{\zeta\in C([0,1], S_{a,r}): \zeta(0)\in \overline{A}_k\cap S_{a,r}  \quad \text{and}\quad \mathcal{J}(\zeta(1))<0\big\},
\end{align*}
with associated minimax level
$$
c_{\overline{q},r}:=\inf_{\zeta\in \Gamma_{\overline{q}}}\max_{u\in \zeta([0,1])}\mathcal{J}(u).
$$
Taking $v\in S_{a,r}$, since $|\nabla (s\star v)|_2\rightarrow 0^+$ as $ s\rightarrow -\infty $, and $\mathcal{J}( s\star v)\rightarrow -\infty$ as $s\rightarrow +\infty $. Thus, there exist $s_0<<-1$ and $s_1>> 1$ such that
$$
\zeta_v(\tau)=[(1-\tau)s_0+\tau s_1]\star v\in \Gamma_{\overline{q}},
$$
which implies that $\Gamma_q\neq \emptyset $. Next, we claim that
$$
c_{\overline{q},r}=\sigma_{\overline{q},r}:=\inf_{u\in\mathcal{P}_{r}}\mathcal{J}(u)>0.
$$
On the one hand, assuming that there exists $\widetilde{v}\in \mathcal{P}_{r}$ such that $\mathcal{J}(\widetilde{v})<c_{a,\overline{q}}$.  Then there exist $\widetilde{s}_0<<-1$ and $\widetilde{s}_1>> 1$ such that $\widetilde{s}_0 \star \widetilde{v}\in \overline{A}_k $ and $\mathcal{J}(\widetilde{s}_1\star\widetilde{v})<0 $. Hence,
$ \zeta_{\widetilde{v}}(\tau)=[(1-\tau)\widetilde{s}_0+\tau \widetilde{s}_1]\star \widetilde{v}\in \Gamma_{\overline{q}}$, which  and Lemma \ref{K-J1} show that
 \begin{align*}
 c_{\overline{q},r}\leq \max_{\tau\in[0,1]}\mathcal{J}(\zeta_{\widetilde{v}}(\tau) )=\max_{\tau\in[0,1]}\mathcal{J}([(1-\tau)\widetilde{s}_0+\tau \widetilde{s}_1]\star \widetilde{v})=\mathcal{J}(\widetilde{v} )<c_{\overline{q},r}.
\end{align*}
a contradiction. Thus, for any $u\in  \mathcal{P}_{r}$, one infers  $\mathcal{J}(u)\geq  c_{\overline{q},r} $, which implies that $\sigma_{\overline{q},r}\geq  c_{\overline{q},r}$.
On the other hand, for any $\zeta\in \Gamma_{\overline{q}}$, since $\zeta(0)\in  \overline{A}_k\cap S_{a,r}$, then by Lemma \ref{K-J3}, we have $P( \zeta(0))>0$. Thus, $t_{\zeta(0)}> 0$  where $t_{\zeta(0)}$  is defined in Lemma \ref{K-J1} satisfying $t_{\zeta(0)}\star \zeta(0)\in   \mathcal{P}_{r} $. Moreover, since
$\psi_{\zeta(1)}(s)>0$ for any $s\in(-\infty, t_{ \zeta(1)})$  and $\psi_{\zeta(1)}(0)=\mathcal{J}(\zeta(1))<0 $, where $t_{ \zeta(1)} $ satisfying $t_{\zeta(1)}\star \zeta(1)\in   \mathcal{P}_{r} $,
we know that $t_{\zeta(1)}<0$.
Hence, from Lemma \ref{K-J1}-(iii), there exists $\overline{\tau}\in(0,1)$ such that $t_{\zeta(\overline{\tau})}=0$ by the continuity. Hence, $\zeta(\overline{\tau})\in  \mathcal{P}_{r}$. So,
$$
\max_{u\in \zeta([0,1])}\mathcal{J}(u)\geq \mathcal{J}(\zeta(\overline{\tau}))\geq \inf_{\mathcal{P}_{r} }\mathcal{J}(u)=\sigma_{\overline{q},r},
 $$
which implies that $ c_{\overline{q},r}\geq \sigma_{\overline{q},r}$. Hence, we deduce that $c_{\overline{q},r}=\sigma_{\overline{q},r}$. Furthermore, it follows from Lemma \ref{K-J2}
 that $c_{\overline{q},r}=\sigma_{\overline{q},r}>0$. Then the claim follows.

For any $a>0$ satisfying \eqref{k201}  fixed, it is easy to verify that the set
$$
L:=\{u\in \mathcal{P}_{r}:\mathcal{J}(u)\leq c_{\overline{q},r}+1\}
$$
is bounded. Then let $M_0>0$ be such that $L\subset B(0, M_0)$, where
$$
B(0, M_0):=\{u\in \mathcal{H}_r: \|u\|\leq M_0\}.
$$

  Similar to Lemmas \ref{K-Lem3.1} and  \ref{K-Lem3.2}, we have the following two theses, which are crucial to derive a special Palais-Smail sequence.

\begin{lemma}\label{KJ-Lem3.1}
Assume that $\mu>0$, $q=\overline{q}=\frac{10}{3}$, $p\in (\frac{10}{3}, 6)$ and \eqref{k201} hold.
Let
\begin{align*}
\Omega_{\delta}:=\left\{u\in S_{a,r}:\ |\mathcal{J}(u)-c_{\overline{q},r}  |\leq\delta, \ dist(u, \mathcal{P}_{r})\leq 2\delta, \ \|\mathcal{J}'_{|_{S_{a,r}}}(u)\|_{\mathcal{H}_r^{-1}}\leq 2 \delta\right\},
\end{align*}
then for any $\delta>0$, $\Omega_{\delta}\cap B(0, 3M_0)$ is nonempty.
\end{lemma}

\begin{lemma}\label{KJ-Lem3.2}
Let $\mu>0$, $q=\overline{q}=\frac{10}{3}$,  $p\in (\frac{10}{3}, 6)$ and \eqref{k201} be hold.  Then there exists a sequence $\{u_n\}\subset S_{a,r} $ and a constant $\alpha>0$ satisfying
$$
P(u_n)=o(1), \quad  \mathcal{J}(u_n)=c_{\overline{q},r} +o(1)  \quad \text{and} \quad \|\mathcal{J}'_{|_{S_{a,r}}}(u_n)\|_{\mathcal{H}_{r}^{-1}}= o(1).
$$
\end{lemma}

\begin{lemma}\label{KJ-Lem2.6}
Let $\mu>0$, $q=\overline{q}=\frac{10}{3}$,  $p\in (\frac{10}{3}, 6)$ and \eqref{k201} be hold.
Assume that $\{u_n\}\subset S_{a,r}$ is a Palais-Smail sequence at level $c\neq 0$ and $P(u_n)\rightarrow 0$ as $n\rightarrow\infty$. Then there exists $u\in S_{a,r}$ such that $u_n\rightarrow u$ in $\mathcal{H}_r$ and $(\lambda, u )\in \mathbb{R}^+ \times  \mathcal{H}_{r}$ solves Eq.~\eqref{k1}.
\end{lemma}

\begin{proof}
  According to \eqref{k96}, due to the fact that $P(u_n)\rightarrow 0$ as $n\rightarrow\infty$ and  \eqref{k201},  we can easily get that $\{u_n\}$ is bounded in $\mathcal{H}_r$.
Then there exists $u\in \mathcal{H}_r$ such that, up to a subsequence, as $n\rightarrow\infty$,
 \begin{align*}
   & u_n\rightharpoonup u \quad\quad \text{in}\ \mathcal{H}_r;\\
   & u_n\rightarrow u\quad\quad \text{in}\ L^t(\mathbb{R}^3)\ \text{with}\ t\in(2, 6);\\
   & u_n\rightarrow u \quad\quad  a.e \ \text{in}\ \mathbb{R}^3.
\end{align*}
Since $\mathcal{J}'_{|_{S_a}}(u_n)\rightarrow 0$,  then there exists $\lambda_n\in \mathbb{R} $ satisfying
\begin{align*} 
-\Delta u_n+\lambda_n u_n+(|x|^{-1}\ast |u_n|^2)u_n  =\mu |u_n|^{\overline{q}-2}u_n+|u_n|^{p-2}u_n+o(1).
\end{align*}
 Multiplying the above equation by $u_n$ and integrating in $\mathbb{R}^3$,
\begin{align} \label{k215}
 \lambda_n a^2= -|\nabla u_n|_2^2-\int_{\mathbb{R}^3} \int_{\mathbb{R}^3} \frac{|u_n(x)|^{2}|u_n(y)|^2} {|x-y|}dxdy
+\int_{\mathbb{R}^3}|u_n|^{p}dx+\mu \int_{\mathbb{R}^3}|u_n|^{\overline{q}}dx+o( \|u_n\|).
\end{align}
Since $\{u_n\}\subset \mathcal{H}_r$ is bounded, then it follows from \eqref{k215} that  $\{\lambda_n\}$  is bounded in $\mathbb{R}$. Hence, there is $\lambda\in \mathbb{R}$ satisfying $\lambda_n\rightarrow \lambda$ as $n\rightarrow\infty$. Then we have
\begin{align} \label{k220}
-\Delta u+\lambda u+ (|x|^{-1}\ast |u|^2)u =\mu |u|^{\overline{q}-2}u+|u|^{p-2}u.
\end{align}
Clearly, $u\neq0$.  Besides, we show that $\lambda>0$.
 Since $(\lambda, u)\in \mathbb{R}\times \mathcal{H}_r$ satisfies \eqref{k220},
then  the following equalities hold:
\begin{align}
|\nabla u|_2^2+ \lambda \int_{\mathbb{R}^3} |u|^{2}dx  +\int_{\mathbb{R}^3} \int_{\mathbb{R}^3} \frac{|u(x)|^{2}|u(y)|^2} {|x-y|}dxdy
-\int_{\mathbb{R}^3}|u|^{p}dx-\mu \int_{\mathbb{R}^3}|u|^{\overline{q}}dx&=0,\label{k217}\\
 |\nabla u|_2^2+ \frac{1}{4}\int_{\mathbb{R}^3} \int_{\mathbb{R}^3} \frac{|u(x)|^{2}|u(y)|^2} {|x-y|}dxdy-\gamma_p \int_{\mathbb{R}^3}|u|^{p}dx-\mu \gamma_q \int_{\mathbb{R}^3}|u|^{\overline{q}}dx&=0.\label{k218}
\end{align}
 Similar to \eqref{k200},  one gets from \eqref{k218} that
\begin{align}\label{k202}
|\nabla u|_2\geq \Big(\frac{1}{C_p^p \gamma_p a^{p(1-\gamma_p)}}\Big)^{\frac{1}{p \gamma_p-2}} \Big( 1-\frac{3}{5}\mu C_{\overline{q}}^{\overline{q}}a^{\frac{4}{3}}\Big)^{\frac{1}{p \gamma_p-2}}>0.
\end{align}
 We argue by contradiction   assuming that $\lambda\leq 0$.
By eliminating $\int_{\mathbb{R}^3}|u|^{p}dx$
  from  \eqref{k217} and \eqref{k218},
  \begin{align*}
(\gamma_p-1)|\nabla u|_2^2+ \lambda \gamma_p \int_{\mathbb{R}^3} |u|^{2}dx
+\big(\gamma_p-\frac{1}{4}\big) \int_{\mathbb{R}^3} \int_{\mathbb{R}^3} \frac{|u(x)|^{2}|u(y)|^2} {|x-y|}dxdy +\mu ( \gamma_{\overline{q}}-\gamma_p) \int_{\mathbb{R}^3}|u|^{\overline{q}}dx=0.
\end{align*}
  Then,
    \begin{align*}
  0\leq-\lambda \gamma_p \int_{\mathbb{R}^3} |u|^{2}dx =&(\gamma_p-1)|\nabla u|_2^2+(\gamma_p-\frac{1}{4}) \int_{\mathbb{R}^3} \int_{\mathbb{R}^3} \frac{|u(x)|^{2}|u(y)|^2} {|x-y|}dxdy +\mu ( \gamma_{\overline{q}}-\gamma_p) \int_{\mathbb{R}^3}|u|^{\overline{q}}dx\\
  \leq & (\gamma_p-1)|\nabla u|_2^2+\big(\gamma_p-\frac{1}{4}\big) \int_{\mathbb{R}^3} \int_{\mathbb{R}^3} \frac{|u(x)|^{2}|u(y)|^2} {|x-y|}dxdy\\
  \leq &  (\gamma_p-1)|\nabla u|_2^2+ (\gamma_p-\frac{1}{4})  C_{\frac{12}{5}}^{\frac{12}{5}}a^3|\nabla u|_2,
\end{align*}
which gives that
  \begin{align}\label{kk149}
(1-\gamma_p) |\nabla u|_2\leq (\gamma_p-\frac{1}{4})  C_{\frac{12}{5}}^{\frac{12}{5}}a^3.
\end{align}
Combining \eqref{k202} and \eqref{kk149},
   we get a contradiction  by  \eqref{k201}. Thus,   $\lambda>0$. Similarly as the Step 3 of Lemma \ref{K-Lem2.6}, we have $\|u_n-u \|\rightarrow 0$ in $\mathcal{H}_r$  as $n\rightarrow\infty$. Hence, one infers that $( \lambda, u)\in \mathbb{R}^+\times \mathcal{H}_r$ is a solution for Eq.~\eqref{k1}, which satisfies $\mathcal{J}(u)=c_{\overline{q},r}$. We complete the proof.
\end{proof}

 \noindent{\bf{Proof of Theorem \ref{K-TH15}.}} In view of Lemmas \ref{KJ-Lem3.2} and \ref{KJ-Lem2.6}, we obtain that there is $( \lambda_{\overline{q}}, u_{\overline{q}} )\in \mathbb{R}^+ \times \mathcal{H}_r $ such that $( \lambda_{\overline{q}}, u_{\overline{q}} )$ solves Eq.~\eqref{k1} and $\mathcal{J}( u_{\overline{q}})=c_{\overline{q},r}$  with  $|u_{\overline{q}} |_2^2=a^2$. Similar  to the proof of Theorem \ref{K-TH1}, we know that $u_{\overline{q}} >0$.

\section{ Properties of solutions}\label{sec4}

\subsection{Asymptotic behavior of solutions as $\mu\rightarrow 0$ }
   Within a section, since we consider the dependence of  $\mu$, hence for $\mu\geq 0$, we may writing $ J_{\mu}$,   $\mathcal{P}_{\mu, \pm}$ and $c_{a, \mu}$ to denote   the functional $ J$, the Poho\v{z}aev manifold    $\mathcal{P}_{\pm}$ and the mountain pass level $c_{a}$ in Theorem \ref{K-TH2}, respectively. 
  Under the assumption of Theorem \ref{K-TH6}, let $ (\widehat{\lambda}_{\mu}, \widehat{u}_{\mu})\in \mathbb{R}^+ \times \mathcal{H}_r $ be the mountain pass type solutions for
  Eq.~\eqref{k1}, obtained in Theorem \ref{K-TH2}, where
 \begin{align}\label{k108}
  c_{a, \mu}= \inf_{u\in\mathcal{P}_{\mu, -,r}}\mathcal{J}_{\mu}(u) = \mathcal{J}_{\mu}(\widehat{u}_{\mu}).
 \end{align}
Moreover, according to Theorem \ref{K-TH5}, let  $( \widetilde{\lambda}, \widetilde{u}  )\in \mathbb{R}^+ \times \mathcal{H}_r $ be the   solution  for Eq.~\eqref{k1},   where
  \begin{align*}
  c_{a, 0}= \inf_{u\in\mathcal{P}_{0, -,r}}\mathcal{J}_{0}(u) = \mathcal{J}_{0}(\widetilde{u}).
 \end{align*}

\begin{lemma}\label{K-Lem6.1}
 Under the assumption of Theorem \ref{K-TH6}, it holds that
$c_{a, \mu}=\inf_{S_{a, r}}\max_{s\in \mathbb{R}} J_{\mu} (s\star u)$,  and $c_{a, 0}=\inf_{S_{a, r}}\max_{s\in \mathbb{R}} J_{0} (s\star u)$.
\end{lemma}
\begin{proof}
 On the one hand, by using \eqref{k108} and Lemma \ref{K-Lem2.3}, we know that
 $$
 c_{a, \mu}=J_{\mu} (\widehat{u}_{\mu})=\max_{s\in \mathbb{R}} J_{\mu} (s\star \widehat{u}_{\mu})\geq\inf_{S_{a, r}}\max_{s\in \mathbb{R}} J_{\mu} (s\star u).
 $$
  On the other hand, recalling that $c_{a, \mu}= \inf_{u\in\mathcal{P}_{\mu, -,r}}\mathcal{J}_{\mu}(u)$, then it follows from Lemma \ref{K-Lem2.3} that, for any $u\in S_{a, r}$, we have $t_u \star u \in \mathcal{P}_{\mu,-,r}$, and hence
  $$
  \max_{s\in \mathbb{R}} J_{\mu} (s\star u)\geq  J_{\mu} (t_u\star u)\geq\inf_{u\in\mathcal{P}_{\mu, -,r}}\mathcal{J}_{\mu}(u)=c_{a, \mu}.
  $$
 Similarly, we can prove  $c_{a, 0}=\inf_{S_{a, r}}\max_{s\in \mathbb{R}} J_{0} (s\star u)$. Thus, we complete the proof.
\end{proof}

\begin{lemma}\label{K-Lem6.2}
 Under the assumption of Theorem \ref{K-TH6}, for any $0< \mu_1<\mu_2$,  we have $c_{a, \mu_2} \leq c_{a, \mu_1}\leq c_{a, 0}$.
\end{lemma}

\begin{proof}
From Lemma \ref{K-Lem6.1},  for any $0< \mu_1<\mu_2$, we  deduce that
\begin{align*}
c_{a, \mu_2}&\leq \max_{s\in \mathbb{R}} \mathcal{J}_{\mu_2}(s\star \widehat{u}_{ \mu_1})\\
&= \max_{s\in \mathbb{R}}\Big[\mathcal{J}_{\mu_1}(s\star \widehat{u}_{ \mu_1})+\frac{e^{q \gamma_q s}}{q}(\mu_1-\mu_2)|\widehat{u}_{ \mu_1}|_q^q  \Big]\\
&\leq \max_{s\in \mathbb{R}}\mathcal{J}_{\mu_1}(s\star \widehat{u}_{ \mu_1})\\
&= \mathcal{J}_{\mu_1}( \widehat{u}_{ \mu_1})\\
&=c_{a, \mu_1}.
\end{align*}
 Similarly,  $c_{a, \mu_1}\leq c_{a, 0}$ can be obtained. Hence, we complete the proof.
\end{proof}

  \noindent{\bf{Proof of Theorem \ref{K-TH6}}.} Assume that $\mu>0$, $q\in(2, \frac{12}{5} ]$, $p\in (\frac{10}{3}, 6)$, $a\in(0, \min \{\overline{a}_0, \widetilde{a}_0\})$ and \eqref{k3} hold.  let $ \{(\widehat{\lambda}_{\mu}, \widehat{u}_{\mu})\}\subseteq \mathbb{R}^+ \times \mathcal{H}_r $ be the sequence of solutions for Eq.~\eqref{k1}, where $\mathcal{J}_{\mu}(\widehat{u}_{\mu})=c_{a, \mu}$.
Now, we prove that $\{\widehat{u}_{\mu}\}$ is bounded in $\mathcal{H}_r$. Since $P_{\mu}(\widehat{u}_{\mu} )=0$ and Lemma \ref{K-Lem6.2}, then   we have
\begin{align*}
    c_{a, 0}\geq c_{a, \mu}
 =&\mathcal{J}_{\mu} (\widehat{u}_{\mu})-\frac{1}{p\gamma_p} P_{\mu}(\widehat{u}_{\mu} )\nonumber\\
 =&\left(\frac{1}{2}-\frac{1}{p\gamma_p}\right)\int_{\mathbb{R}^3}|\nabla \widehat{u}_{\mu} |^2dx+\frac{1}{4}\left(1-\frac{1}{p\gamma_p}\right)\int_{\mathbb{R}^3} \int_{\mathbb{R}^3} \frac{|\widehat{u}_{\mu}(x) |^{2}|\widehat{u}_{\mu} (y)|^2} {|x-y|}dxdy\\
&+\mu\left(\frac{\gamma_q}{p\gamma_p}-\frac{1}{q}\right)\int_{\mathbb{R}^3}|\widehat{u}_{\mu} |^{q}dx\nonumber\\
 \geq &\left(\frac{1}{2}-\frac{1}{p\gamma_p}\right)\int_{\mathbb{R}^3}|\nabla \widehat{u}_{\mu} |^2dx-\mu  \frac{p\gamma_p-q\gamma_q}{pq\gamma_p} C_q^qa^{(1-\gamma_q)q}|\nabla u_n|_2^{q \gamma_q},
\end{align*}
which shows that $\{\widehat{u}_{\mu} \}$ is bounded in $\mathcal{H}_r$ because of $q\gamma_q<2$. Therefore,  there exists $\widehat{u}_0\in \mathcal{H}_r$ such that, up to a subsequence, as $\mu\rightarrow 0$,
 \begin{align*}
   & \widehat{u}_{\mu}\rightharpoonup \widehat{u}_0 \quad\quad \text{in}\ \mathcal{H}_r;\\
   & \widehat{u}_{\mu}\rightarrow \widehat{u}_0\quad\quad \text{in}\ L^t(\mathbb{R}^3)\ \text{with}\ t\in(2, 6);\\
   & \widehat{u}_{\mu}\rightarrow \widehat{u}_0\quad\quad  a.e \ \text{in}\ \mathbb{R}^3.
\end{align*}
Since  $(\widehat{\lambda}_{\mu}, \widehat{u}_{\mu})\in \mathbb{R}^+ \times \mathcal{H}_{r}$   satisfies
\begin{align} \label{kkk191}
-\Delta \widehat{u}_{\mu}+\widehat{\lambda}_{\mu}\widehat{u}_{\mu}+  (|x|^{-1}\ast|\widehat{u}_{\mu}|^2 )\widehat{u}_{\mu}  =\mu |\widehat{u}_{\mu}|^{q-2}\widehat{u}_{\mu}+|\widehat{u}_{\mu}|^{p-2}\widehat{u}_{\mu},
\end{align}
 then multiplying the above equation by $\widehat{u}_{\mu}$ and integrating in $\mathbb{R}^3$,
\begin{align*}
\widehat{\lambda}_{\mu} a^2= -|\nabla \widehat{u}_{\mu}|_2^2-\int_{\mathbb{R}^3} \int_{\mathbb{R}^3} \frac{|\widehat{u}_{\mu}(x)|^{2}|\widehat{u}_{\mu}(y)|^2} {|x-y|}dxdy
+\int_{\mathbb{R}^3}|\widehat{u}_{\mu}|^{p}dx+\mu \int_{\mathbb{R}^3}|\widehat{u}_{\mu}|^{q}dx,
\end{align*}
which implies that  $\{\widehat{\lambda}_{\mu}\}$  is bounded in $\mathbb{R}$. Hence, there is $\lambda_0 \geq 0$ satisfying $\lambda_{\mu}\rightarrow \widehat{\lambda}_0$ as $n\rightarrow\infty$. Moreover, when $\mu\rightarrow 0$, from \eqref{kkk191} we have
\begin{align} \label{kkk20}
-\Delta \widehat{u}_0+\widehat{\lambda}_0 \widehat{u}_0+(|x|^{-1}\ast |\widehat{u}_0|^2)\widehat{u}_0
= |\widehat{u}_0|^{p-2}\widehat{u}_0.
\end{align}
Clearly, $\widehat{u}_0\neq 0$. Moreover, using the assumption \eqref{k3}, similar to the proof of \eqref{k40}, we have     the    Lagrange multiplier $\widehat{\lambda}_0>0$.
Then, using a similar methods of the Step 3 in Lemma \ref{K-Lem2.6},
  we conclude  that $\widehat{u}_{\mu}\rightarrow \widehat{u}_{0}$ in  $\mathcal{H}_{r}$ and $\widehat{\lambda}_{\mu}\rightarrow \widehat{\lambda}_{0}$ in $\mathbb{R} $  as $\mu\rightarrow0$  up to a subsequence, where $(\widehat{\lambda}_{0}, \widehat{u}_{0})\in \mathbb{R}^+ \times \mathcal{H}_r  $  is the solution for the following equation
\begin{align*}
  \begin{cases}
\displaystyle - \Delta u+\lambda u+  (|x|^{-1}\ast|u|^2)u=|u|^{p-2}u  \quad \quad \text{in} \ \mathbb{ R}^3,\\
\displaystyle \int_{\mathbb{R}^3}|u|^2dx=a^2.\\
  \end{cases}
 \end{align*}
Thus, we complete the proof. $\hfill\Box$

\subsection{Limit behavior of solutions as $q\rightarrow  \frac{10}{3}$ }
In this subsection,    in order to study  the limit behavior of  normalized solutions as $q\rightarrow  \frac{10}{3}$, we use  $\mathcal{P}_{q,r}$, $\mathcal{J}_q$ and $P_q$   instead of  $\mathcal{P}_{r}$, $\mathcal{J}$ and $P$  to indicate the dependence on $q$ for convenience.
 Similar to Lemma \ref{K-J1}, we have
 \begin{lemma}\label{KK-J1}
Let  $\mu>0$, $\frac{10}{3}=\overline{q} <q< p<6$.
   For any $u\in S_a$,  there exists a unique $t_u \in \mathbb{R}$ such that $t_u\star u\in \mathcal{P}_{q, r}$. $t_u$ is the unique
critical point of the function $\psi_u$, and is a strict maximum point at positive level. Moreover:
\begin{itemize}
  \item [(i)]  $\mathcal{P}_{q,r}=\mathcal{P}_{q,r,-}$;
  \item [(ii)]  the function $\psi_u$ is strictly decreasing and concave on $(t_u,+\infty)$. Particularly,  $t_u < 0\Leftrightarrow P_q(u) <0$;
    \item [(iii)]   the map $u\in S_a\mapsto t_u\in \mathbb{R}$ is of class $C^1$.
\end{itemize}
  \end{lemma}
It is easy to know that $
c_{q,r}=\sigma_{q,r}:=\inf_{u\in\mathcal{P}_{q,r}}\mathcal{J}_q(u)>0,
$
where $c_{q,r}$ is the mountain pass level defined in Theorem \ref{K-T1}.

\begin{lemma}\label{KJ-Lem1}
Assume that $\mu>0$,  $\frac{10}{3}=\overline{q}<q<p<6$ and $a\in (0, \widetilde{\kappa})$ satisfying \eqref{k201}.  It holds that
\begin{align}\label{k223}
\liminf_{q\rightarrow \overline{q}} c_{q, r}>0,
\end{align}
and
\begin{align}\label{k224}
\limsup_{q\rightarrow \overline{q}}c_{q, r}\leq c_{\overline{q}, r}.
\end{align}
\end{lemma}

\begin{proof}
The proof of \eqref{k223} is similar to \cite[Lemma 3.1]{2023-JDE-Qi}, so we omit it here.  Now, we prove \eqref{k224}. For any $\varepsilon\in (0, \frac{1}{2})$, there exists $u\in \mathcal{P}_{\overline{q},r}$ such that $ \mathcal{J}_{\overline{q}}(u)\leq c_{\overline{q},r}+\varepsilon$. Then, due to $\overline{q}<p<6 $, there is $T\in \mathbb{R}$ large enough satisfying
\begin{align*}
 \mathcal{J}_{\overline{q}}(T\star u )
=&\frac{e^{2T}}{2}\int_{\mathbb{R}^3}|\nabla u|^2dx+\frac{e^{T}}{4} \int_{\mathbb{R}^3} \int_{\mathbb{R}^3} \frac{|u(x)|^{2}|u(y)|^2} {|x-y|}dxdy \nonumber\\
&-\frac{e^{p\gamma_pT}}{p}\int_{\mathbb{R}^3}|u|^{p}dx
-\frac{\mu e^{2T}}{\overline{q}}
\int_{\mathbb{R}^3}|u|^{\overline{q}}dx\nonumber\\
\leq &  -1.
\end{align*}
  By the continuity of $ \frac{e^{q\gamma_qs}}{q}\int_{\mathbb{R}^3}|u|^{q}dx$ on $(s,q)\in(-\infty, T]\times [\overline{q}, 6)$  there exists $\eta > 0$ such that
\begin{align}\label{k225}
| \mathcal{J}_{q}(s\star u)- \mathcal{J}_{\overline{q}}(s\star u)|=\bigg|\frac{e^{q\gamma_qs}}{q}\int_{\mathbb{R}^3}|u|^{q}dx
-\frac{e^{2s}}{\overline{q}}\int_{\mathbb{R}^3}|u|^{\overline{q}}dx \bigg|<\varepsilon
\end{align}
for all $\overline{q}<q<\overline{q}+\eta$ and $s\in(-\infty, T]$, which shows that $\mathcal{J}_{q}(s\star u)\leq -\frac{1}{2}$ for all $\overline{q}<q<\overline{q}+\eta$. Besides, when $s\in \mathbb{R}$ small enough, one gets
$$
 \mathcal{J}_{q}(s\star u )>0.
$$
Hence, there is $s_0\in (-\infty, T)$ such that $\frac{d}{ds} \mathcal{J}_{q}(s\star u)|_{s=s_0}=0$. This gives us that $s_0\star u\in \mathcal{P}_{q, r}$. Thus, by \eqref{k225}, Lemmas \ref{K-J1} and \ref{KK-J1} we have
$$
c_{q, r}\leq \mathcal{J}_{q}(s_0\star u)\leq  \mathcal{J}_{\overline{q}}(s_0\star u)+\varepsilon\leq \mathcal{J}_{\overline{q}}(  u)+\varepsilon\leq c_{\overline{q},r}+2\varepsilon
$$
for all $\overline{q}<q<\overline{q}+\eta$. Therefore, we complete the proof.
\end{proof}

\begin{lemma}\label{KJ-Lem2}
Assume that $\mu>0$, $\frac{10}{3}=\overline{q}<q<p<6$ and \eqref{k201} hold. Let $(\lambda_q, u_q )$ be the radial mountain pass type normalized solution of Eq.~\eqref{k1} at the level $c_{q,r}$,  then
there is a constant $ C > 0$ independent of $q$ such that
\begin{align*}
\lim_{q\rightarrow \overline{q}}\|\nabla u_q\|_2^2\leq C.
\end{align*}
\end{lemma}
\begin{proof}
Assume to the contrary that there is a subsequence, denoted still by $\{u_q\}$, such that
\begin{align}\label{k229}
\|\nabla u_q\|_2^2\rightarrow\infty\quad\quad \text{as}\ q\rightarrow \overline{q}.
\end{align}
Since $u_q$ is a solution of Eq.~\eqref{k1},  then
\begin{align}\label{k227}
 P_q(u_q)=\int_{\mathbb{R}^3}|\nabla u_q|^2dx+\frac{1}{4}\int_{\mathbb{R}^3} \int_{\mathbb{R}^3} \frac{|u_q(x)|^{2}|u_q(y)|^2} {|x-y|}dxdy
-\gamma_p\int_{\mathbb{R}^3}|u_q|^{p}dx
-\mu\gamma_q\int_{\mathbb{R}^3}|u_q|^{q}dx=0.
\end{align}
Then, we deduce
\begin{align*}
c_{q,r}=&\mathcal{J}_q(u_q)-\frac{1}{q \gamma_q}  P_q(u_q)\\
=& \left(\frac{1}{2}-\frac{1}{q\gamma_q}\right)\int_{\mathbb{R}^3}|\nabla u_q|^2dx+\frac{1}{4}\left(1-\frac{1}{q\gamma_q}\right)\int_{\mathbb{R}^3} \int_{\mathbb{R}^3} \frac{|u_q(x)|^{2}|u_q(y)|^2} {|x-y|}dxdy\\
&+ \left(\frac{\gamma_p}{q\gamma_q}-\frac{1}{p}\right)\int_{\mathbb{R}^3}|u_q|^{p}dx\nonumber\\
\geq&\left(\frac{\gamma_p}{q\gamma_q}-\frac{1}{p}\right)\int_{\mathbb{R}^3}|u_q|^{p}dx,
\end{align*}
which and \eqref{k224} imply that
\begin{align}\label{k228}
\lim_{q\rightarrow \overline{q}}\int_{\mathbb{R}^3}|u_q|^{p}dx\leq \frac{pq \gamma_q}{p\gamma_p-q\gamma_q} c_{\overline{q},r}.
\end{align}
Moreover, due to $\overline{q}<q<p$, by   the interpolation inequality and \eqref{k10}, we have
\begin{align}\label{k226}
\int_{\mathbb{R}^3}|u|^{q}dx&\leq\Big(\int_{\mathbb{R}^3}|u|^{\overline{q}}dx\Big)^{\frac{p-q}{p-\overline{q}}}
\Big(\int_{\mathbb{R}^3}|u|^{p}dx\Big)^{\frac{q-\overline{q}}{p-\overline{q}}}\nonumber\\
&\leq C_{\overline{q}}^{\frac{\overline{q}(p-q)}{p-\overline{q}}} |u|_2^{\frac{4(p-q)}{3(p-\overline{q})}}|\nabla u|_2^{\frac{2(p-q)}{p-\overline{q}}} \Big(\int_{\mathbb{R}^3}|u|^{p}dx\Big)^{\frac{q-\overline{q}}{p-\overline{q}}}.
\end{align}
Then, it follows from \eqref{k227} and \eqref{k226} that
\begin{align*}
c_{q,r}
=&\mathcal{J}_q(u_q)-\frac{1}{p \gamma_p}  P_q(u_q)\\
=& \left(\frac{1}{2}-\frac{1}{p\gamma_p}\right)\int_{\mathbb{R}^3}|\nabla u_q|^2dx+\frac{1}{4}\left(1-\frac{1}{p\gamma_p}\right)\int_{\mathbb{R}^3} \int_{\mathbb{R}^3} \frac{|u_q(x)|^{2}|u_q(y)|^2} {|x-y|}dxdy\\
&+\mu\left(\frac{\gamma_q}{p\gamma_p}-\frac{1}{q}\right)\int_{\mathbb{R}^3}|u_q|^{q}dx\nonumber\\
 \geq & \left(\frac{1}{2}-\frac{1}{p\gamma_p}\right)\int_{\mathbb{R}^3}|\nabla u_q|^2dx-\mu  \frac{p\gamma_p-q\gamma_q}{pq\gamma_p} C_{\overline{q}}^{\frac{\overline{q}(p-q)}{p-\overline{q}}} |u_q|_2^{\frac{4(p-q)}{3(p-\overline{q})}}|\nabla u_q|_2^{\frac{2(p-q)}{p-\overline{q}}} \Big(\int_{\mathbb{R}^3}|u_q|^{p}dx\Big)^{\frac{q-\overline{q}}{p-\overline{q}}},
\end{align*}
which and \eqref{k228} show that
$$
c_{\overline{q},r}\geq \lim_{q\rightarrow \overline{q}}c_{q,r}\geq\frac{[3(p-2)-4](5-3\mu C_{\overline{q}}^{\overline{q}} a^{\frac{4}{3}} )}{30(p-2)}\lim_{q\rightarrow \overline{q}} |\nabla u_q|_2^2.
$$
This gives a contradiction due to \eqref{k229} and \eqref{k201}. Hence, we complete the proof.
\end{proof}

\begin{lemma}\label{KJ-Lem3}
Assume that $\mu>0$,  $\frac{10}{3}=\overline{q}<q<p<6$ and $a\in (0, \widetilde{\kappa})$ satisfying \eqref{k201}. Let $(\lambda_q, u_q )$ be the radial mountain pass type normalized solution of Eq.~\eqref{k1} at the level $c_{q,r}$, then there exists a constant
$  \Lambda> 0$ independent of $q$ such that  $0 < \lambda_q <\Lambda $
 for any $q$ tending to $\overline{q}$.
\end{lemma}
\begin{proof}
Firstly, since $u_q$ is a solution of Eq.~\eqref{k1}, then $u_q$ satisfies
\begin{align}
|\nabla u_q|_2^2+ \lambda \int_{\mathbb{R}^3} |u_q|^{2}dx  +\int_{\mathbb{R}^3} \int_{\mathbb{R}^3} \frac{|u_q(x)|^{2}|u_q(y)|^2} {|x-y|}dxdy
-\int_{\mathbb{R}^3}|u_q|^{p}dx-\mu \int_{\mathbb{R}^3}|u_q|^{q}dx&=0,\label{k230}\\
 P_q(u_q)=|\nabla u_q|_2^2+ \frac{1}{4}\int_{\mathbb{R}^3} \int_{\mathbb{R}^3} \frac{|u_q(x)|^{2}|u_q(y)|^2} {|x-y|}dxdy-\gamma_p \int_{\mathbb{R}^3}|u_q|^{p}dx-\mu \gamma_q \int_{\mathbb{R}^3}|u_q|^{q}dx&=0.\label{k231}
\end{align}
 By eliminating $\int_{\mathbb{R}^3}|u|^{q}dx$
  from  \eqref{k230} and \eqref{k231}, one infers that
  \begin{align*}
(1-\gamma_q)|\nabla u_q|_2^2= \lambda \gamma_q \int_{\mathbb{R}^3} |u_q|^{2}dx
+\big(\gamma_q-\frac{1}{4}\big) \int_{\mathbb{R}^3} \int_{\mathbb{R}^3} \frac{|u_q(x)|^{2}|u_q(y)|^2} {|x-y|}dxdy +\mu ( \gamma_p-\gamma_q) \int_{\mathbb{R}^3}|u|^{p}dx,
\end{align*}
which and Lemma \ref{KJ-Lem2} show that
  \begin{align*}
\lim_{q\rightarrow \overline{q}}\lambda_q \leq \frac{2}{3 a^2}\lim_{q\rightarrow \overline{q}}|\nabla u_q|_2^2\leq \frac{2}{3 a^2}C:=\Lambda.
\end{align*}
Next, we prove that $\lim_{q\rightarrow \overline{q}}\lambda_q >0$. Assume to the contrary that, up to a subsequence,
\begin{align}\label{k232}
\lim_{q\rightarrow \overline{q}}\lambda_q =0.
\end{align}
From \eqref{k231} and the Gagliardo-Nirenberg inequality, we know
$$
 |\nabla u_q|_2^2 -\gamma_pC_p^pa^{p(1-\gamma_p)} |\nabla u_q|_2^{p\gamma_p}  -\mu \gamma_q C_q^q a^{q(1-\gamma_q)} |\nabla u_q|_2^{q\gamma_q}\leq 0,
 $$
 which implies that
 \begin{align}\label{k233}
 (1-\frac{3}{5}\mu C_{\overline{q}}^{\overline{q}}a^{\frac{4}{3}})\lim_{q\rightarrow \overline{q}}|\nabla u_q|_2^2\leq \gamma_pC_p^p a^{p(1-\gamma_p)} \lim_{q\rightarrow \overline{q}}|\nabla u_q|_2^{p\gamma_p}.
 \end{align}
 We claim that $\lim_{q\rightarrow \overline{q}}|\nabla u_q|_2^2>0$. In fact, if we assume that $\lim_{q\rightarrow \overline{q}}|\nabla u_q|_2^2=0$, then by eliminating $\int_{\mathbb{R}^3} \int_{\mathbb{R}^3} \frac{|u(x)|^{2}|u(y)|^2} {|x-y|}dxdy$ from \eqref{k230} and  \eqref{k231}, in view  of \eqref{k232}, one has
\begin{align*}
0=\frac{3}{4} \lim_{q\rightarrow \overline{q}}|\nabla u_q|_2^2
=\big(\gamma_p-\frac{1}{4}\big)\lim_{q\rightarrow \overline{q}}\int_{\mathbb{R}^3}|u_q|^{p}dx+\mu  \frac{7}{20} \lim_{q\rightarrow \overline{q}}\int_{\mathbb{R}^3}|u_q|^{q}dx,
\end{align*}
 which implies that
 $$
 \lim_{q\rightarrow \overline{q}}\int_{\mathbb{R}^3}|u_q|^{p}dx=0 \quad\quad \text{and} \quad\quad \lim_{q\rightarrow \overline{q}}\int_{\mathbb{R}^3}|u_q|^{q}dx=0.
 $$
 And then,  we deduce from \eqref{k231} that
 $ \lim_{q\rightarrow\overline{q}}\int_{\mathbb{R}^3} \int_{\mathbb{R}^3} \frac{|u_q(x)|^{2}|u_q(y)|^2} {|x-y|}dxdy=0$. Hence, $ \lim_{q\rightarrow\overline{q}}c_{q, r}=\lim_{q\rightarrow\overline{q}} J_q(u_q)=0$, which contradicts \eqref{k223}.  Then the claim follows.
Hence, it follows from \eqref{k201} and \eqref{k233} that
  \begin{align}\label{k234}
 \lim_{q\rightarrow\overline{q}} |\nabla u_q|_2^{p\gamma_p-2}\geq \frac{1-\mu \gamma_{\overline{q}}C_{\overline{q}}^{\overline{q}} a^{\frac{4}{3}}}{\gamma_p C_p^p a^{p(1-\gamma_p)}}>0.
\end{align}
By eliminating $\int_{\mathbb{R}^3}|u|^{p}dx$
  from  \eqref{k230} and \eqref{k231}, due to  Lemma \ref{gle4},   \eqref{k10},  \eqref{k234} and the Young's inequality, one infers that
  \begin{align*}
 0\leq\mu ( \gamma_p-\gamma_q) \int_{\mathbb{R}^3}|u_q|^{q}dx=&(\gamma_p-1)|\nabla u_q|_2^2
+(\gamma_p-\frac{1}{4}) \int_{\mathbb{R}^3} \int_{\mathbb{R}^3} \frac{|u_q(x)|^{2}|u_q(y)|^2} {|x-y|}dxdy\\
 \leq&(\gamma_p-1)|\nabla u_q|_2^2
+C|\nabla u_q|_2 |u_q|_2^3\\
 \leq& -\frac{3(6-p)}{2(p-2)} |\nabla u_q|_2^2 +C a^6\\
  \leq& -\frac{3(6-p)}{2(p-2)} \Big(\frac{1-\mu \gamma_{\overline{q}}C_{\overline{q}}^{\overline{q}} a^{\frac{4}{3}}}{\gamma_p C_p^p a^{p(1-\gamma_p)}} \Big)^{\frac{p\gamma_p-2}{2}}+C a^6.
\end{align*}
 When $a\in(0, \widetilde{\kappa})$ small enough, the right-hand side of the above is negative, which is a contradiction. Then, we get $\lim_{q\rightarrow \overline{q}}\lambda_q >0$. Hence, we complete the proof.
\end{proof}

\begin{lemma}\label{KJ-Lem4}
Assume that $\mu>0$ and $\frac{10}{3}=\overline{q}<q<p<6$. Let $(\lambda_q, u_q )$ be the radial mountain pass type normalized solution of Eq.~\eqref{k1} at the level $c_{q,r}$,
 there holds
\begin{align}\label{k221}
\limsup_{q\rightarrow\overline{q}}\|u_q\|_{\infty}\leq C.
\end{align}
Moreover,
\begin{align}\label{k222}
\liminf_{q\rightarrow\overline{q}}\|u_q\|_{\infty}>0.
\end{align}
\end{lemma}
\begin{proof}
Let $u_q$ be the radial mountain pass type normalized solution of Eq.~\eqref{k1} at the level $c_{q,r}$. Without loss of generality, we assume $u_q(0)=\max_{x\in \mathbb{R}^3}u_q(x)$. Since the multiplier $\lambda_q >0$,  by the regularity theory of elliptic partial differential equations, we have  $u_q\in C_{loc}^{1, \alpha}(\mathbb{R}^3)$  for some $\alpha\in (0, 1)$. Assume to the
contrary that there is a subsequence, denoted still by $\{u_q\}$, such that
\begin{align*}
\limsup_{q\rightarrow\overline{q}}\|u_q\|_{\infty}=\infty.
\end{align*}
Define
\begin{align*}
v_q=\frac{1}{\|u_q\|_{\infty}}u_q(\|u_q\|_{\infty}^{\frac{2-p}{2}}x) \quad \quad \text{for any}\ x\in \mathbb{R}^3,
\end{align*}
then $v_q(0)=1$ and $\|v_q\|_{\infty}\leq 1$. By a direct calculation, $v_q$ satisfies
\begin{align}\label{k235}
-\Delta v_q+\left(\lambda_q |u_q|_{\infty}^{2-p}+|u_q|_{\infty}^{6-2p}(|x|^{-1}\ast |v_q|^2)-\mu|u_q|_{\infty}^{q-p}|v_q|^{q-2}\right)v_q=|v_q|^{p-2}v_q.
\end{align}
In view of Lemma \ref{KJ-Lem3}, we have
\begin{align*}
\max_{x\in \mathbb{R}^3}\Big| \lambda_q |u_q|_{\infty}^{2-p}+|u_q|_{\infty}^{6-2p}(|x|^{-1}\ast |v_q|^2)-\mu|u_q|_{\infty}^{q-p}|v_q|^{q-2} \Big|<C
\end{align*}
where $C$ is a constant independence of $q$. Then we deduce from  $|v_q |_{\infty} \leq 1$ and the regularity theory of elliptic partial differential equations that $v_q\in C_{loc}^{2, \beta}(\mathbb{R}^3)$
 for some $\beta\in (0, 1)$. Moreover, there is a
constant $C > 0$ independent of $q$ such that
$$
|v_q|_{C_{loc}^{2, \beta}(\mathbb{R}^3)}\leq C.
$$
Hence, there is $v\in C_{loc}^{2, \beta}(\mathbb{R}^3)$  such that, up to a subsequence,
$$
v_q\rightarrow v \quad\quad \text{in}\ C_{loc}^{2}(\mathbb{R}^3)
$$
as $q\rightarrow \overline{q}$. Then it follows from \eqref{k235} that $v$ satisfies
$$
-\Delta v=|v|^{p-2}v\quad\quad \text{in} \ \mathbb{R}^3.
$$
Hence, the  Liouville theorem \cite{1982-PRSEA-E-L} implies that $v(x) = 0$ for any $x \in \mathbb{R}^3 $, which
contradicts the fact $v(0)= \lim_{q\rightarrow \overline{q}}v_q(0)=1$. Hence, \eqref{k221} follows. Now, we prove \eqref{k222}. Assume to the contrary that there is a subsequence, denoted still by
$\{u_q \}$, such that
$$
\liminf_{q\rightarrow\overline{q}}\|u_q\|_{\infty}=0.
$$
Thus, for any $t>2$, one has
$$
\int_{\mathbb{R}^3}|u_q|^tdx\leq |u_q|_{\infty}^{t-2}a^2\rightarrow 0 \quad\quad \text{as}\ q\rightarrow \overline{q}.
$$
and then
$$
\int_{\mathbb{R}^3}\int_{\mathbb{R}^3}\frac{|u_q(x)|^2|u_q(y)|^2}{|x-y|}dxdy\leq C|u_q|_{\frac{12}{5}}^4\rightarrow 0 \quad\quad \text{as}\ q\rightarrow \overline{q}.
$$
By \eqref{k231}, we have $|\nabla u_q|_2^2\rightarrow 0$ as $q\rightarrow \overline{q}$. Hence
$$
\lim_{q\rightarrow \overline{q}}c_{q,r}=\lim_{q\rightarrow \overline{q}} J_{q}(u_q)=0,
$$
which contradicts Lemma \ref{KJ-Lem1}. Thus, \eqref{k222} holds. We complete the proof.
\end{proof}
\begin{lemma}\label{KJ-Lem5}
Assume that $\mu>0$ and $\frac{10}{3}=\overline{q}<q<p<6$. Let $(\lambda_q, u_q )$  be the radial mountain pass type normalized solution of Eq.~\eqref{k1} at the level $c_{q,r}$. Then,
\begin{align}\label{k236}
|u_q(x)|\leq A e^{-c|x|},
\end{align}
 where constants $A$ and $c$ are independent of $q$. Moreover,
up to a subsequence, as $q\rightarrow \overline{q}$,
$$
u_q\rightarrow u \quad\quad \text{in} \ L^2(\mathbb{R}^3).
$$
\end{lemma}
\begin{proof}
The proof is similar to \cite[Lemma 3.6]{2023-JDE-Qi}, so we omit it here.
\end{proof}

 \noindent{\bf{Proof of Theorem \ref{K-TH14}.}} Let $(\lambda_q, u_q )$ be the radial mountain pass type normalized solution of
 Eq.~\eqref{k1} at the level $c_{q,r}$. In view of Lemmas \ref{KJ-Lem2}, \ref{KJ-Lem3}  and \ref{KJ-Lem5}, there exist $u\in  S_a$, $\lambda> 0$  and a
subsequence, denoted still by $\{(\lambda_q, u_q )\}$, such that, as $q\rightarrow \overline{q}$,
$$
u_q\rightharpoonup u \quad\quad \text{in}\  \mathcal{H}, \quad\quad
u_q \rightarrow u \quad\quad \text{in}\ L^2(\mathbb{R}^3) \cap C_{loc} (\mathbb{R}^3)
$$
and
\begin{align}\label{k239}
\lambda_q\rightarrow \lambda.
\end{align}
Thus, for any $t > 2$, one has
\begin{align}\label{k238}
\int_{\mathbb{R}^3}|u_q-u|^tdx\leq |u_q-u|_{\infty}^{s-2}\int_{\mathbb{R}^3}|u_q-u|^2dx\rightarrow 0\quad\quad \text{as}\  q\rightarrow \overline{q}.
\end{align}
It follows from \eqref{k236} that for any $\varphi\in \mathcal{H} $ and $\varepsilon> 0$, there is an $R > 0$ such that
$$
\Re\int_{\mathbb{R}^3\backslash B_R}|u_q|^{q-2}u_q \overline{\varphi } dx\leq\frac{\varepsilon}{4}, \quad \quad \Re\int_{\mathbb{R}^3\backslash B_R}|u|^{\overline{q}-2}u \overline{\varphi} dx\leq\frac{\varepsilon}{4}, \quad \forall \ q\rightarrow \overline{q}.
$$
The  Vitali convergence theorem  and Lemma \ref{KJ-Lem4} imply that
$$
 \int_{B_R}|(|u_q|^{q-2}u_q - |u|^{\overline{q}-2}u)\overline{\varphi}|dx\leq \frac{\varepsilon}{2}, \quad \forall \ q\rightarrow \overline{q}.
$$
Hence,
\begin{align}\label{k237}
\lim_{q\rightarrow \overline{q}} \Re\int_{\mathbb{R}^3} |u_q|^{q-2}u_q\overline{\varphi} dx=\Re\int_{\mathbb{R}^3} |u|^{\overline{q}-2}u\overline{\varphi} dx, \quad \forall \ \varphi \in \mathcal{H}.
\end{align}
As a consequence, by  Lemma \ref{gle4}, \eqref{k239}, \eqref{k238} and \eqref{k237}, $u$ satisfies
$$
-\Delta u+\lambda u+(|x|^{-1}\ast |u|^2)u=\mu|u|^{\overline{q}-2}u+|u|^{p-2}u\quad\quad \text{in} \ \mathbb{R}^3,
$$
where $|u|_2^2=a^2$. Now, we claim that
\begin{align}\label{k241}
\lim_{q\rightarrow \overline{q}} \int_{\mathbb{R}^3}|u_q|^qdx=\int_{\mathbb{R}^3}|u|^{\overline{q}}dx.
\end{align}
Indeed, on the one hand, by the Fatou lemma, we have
$$
\int_{\mathbb{R}^3}|u|^{\overline{q}}dx \leq \liminf_{q\rightarrow \overline{q}} \int_{\mathbb{R}^3}|u_q|^qdx.
$$
On the other hand, the Young inequality implies
$$
\int_{\mathbb{R}^3}|u_q|^qdx\leq \frac{q-\overline{q}}{p-\overline{q}}\int_{\mathbb{R}^3}|u_q|^pdx
+\frac{p-q}{p-\overline{q}}\int_{\mathbb{R}^3}|u_q|^{\overline{q}}dx.
$$
From \eqref{k238}, one has
$$
\limsup_{q\rightarrow \overline{q}} \int_{\mathbb{R}^3}|u_q|^qdx
\leq \int_{\mathbb{R}^3}|u|^{\overline{q}}dx.
$$
Therefore, the claim holds.
By \eqref{k238}, \eqref{k241} and Lemma \ref{KJ-Lem1}, one obtains
\begin{align*}
c_{\overline{q},r}&\geq \lim_{q\rightarrow \overline{q}}c_{q,r}\\
&=\lim_{q\rightarrow \overline{q}}\mathcal{J}_q(u_q)\\
&=\lim_{q\rightarrow \overline{q}}\left(\frac{1}{2}\int_{\mathbb{R}^3}|\nabla u_q|^2dx+\frac{1}{4}\int_{\mathbb{R}^3}\int_{\mathbb{R}^3} \frac{|u_q(x)|^{2}|u_q(y)|^2} {|x-y|}dxdy
-\frac{1}{p}\int_{\mathbb{R}^3}|u_q|^{p}dx
-\frac{\mu}{q}\int_{\mathbb{R}^3}|u_q|^{q}dx \right)\\
& \geq\frac{1}{2}\int_{\mathbb{R}^3}|\nabla u|^2dx+\frac{1}{4}\int_{\mathbb{R}^3}\int_{\mathbb{R}^3} \frac{|u(x)|^{2}|u(y)|^2} {|x-y|}dxdy
-\frac{1}{p}\int_{\mathbb{R}^3}|u|^{p}dx
-\frac{\mu}{\overline{q}}\int_{\mathbb{R}^3}|u|^{\overline{q}}dx\\
 &=\mathcal{J}_q(u)\\
&\geq c_{\overline{q}, r},
\end{align*}
which implies that $u_q\rightarrow u$ in $\mathcal{H}$ as $q\rightarrow \overline{q}$ and $\mathcal{J}_q(u)= c_{\overline{q}, r}$. Hence, we complete the proof.

\subsection{The orbital stability  of ground state set   $\mathcal{M}_a$}
In this subsection, we focus on the properties of ground states in the case $\mu>0$, $q\in(2, \frac{12}{5}]$ and $p\in(\frac{10}{3}, 6)$. In the first step, we provide the structure of the ground state set $\mathcal{M}_a$. And then,  the orbital stability of  $\mathcal{M}_a$ is analyzed.
The proof of Theorem \ref{K-TH3} is inspired by the classical Cazenave-Lions' stability argument \cite{1982-CMP-CL}, further developed by \cite{2004-ANS-Ha}. \\

 \noindent{\bf{Proof of Theorem \ref{K-TH3}.}}
 We   first describe the characteristic of $\mathcal{M}_a$  as
\begin{align}\label{k94}
\mathcal{M}_a=\big\{ e^{i \theta}|u|:  \ \theta\in \mathbb{R}, \ |u|\in D_{\rho_0},\ \mathcal{J}(|u|)=m(a)\ \text{and}\   |u| >0 \ \text{in}\ \mathbb{R}^3\big\},
\end{align}
where $\mathcal{M}_a=\{u\in  \mathcal{H}: u\in D_{\rho_0}, \ \mathcal{J}(u)=m(a)\}$.
For any $z\in \mathcal{H}$,   let $z(x)=(v(x),w(x))= v(x)+iw(x)$, where $  v, w\in \mathcal{H}$ are real-valued functions   and
$$
\|z\|^2=|\nabla z|_2^2+| z|_2^2, \quad\quad | z|_2^2=|v|_2^2+| w|_2^2\quad \quad \text{and}\quad \quad  | \nabla z|_2^2=|\nabla v|_2^2+|\nabla w|_2^2.
$$
On the one hand, taking $z=(v,w)\in \mathcal{M}_a=\{u\in  \mathcal{H}: u\in D_{\rho_0}, \ \mathcal{J}(u)=m(a)\}$,   we see from \cite[Theorem 3.1]{2004-ANS-Ha} that
$|z|\in D_{\rho_0},\ \mathcal{J}(|z|)=m(a)$ and $ | \nabla |z||_2^2=| \nabla z|_2^2$.
Then, by the fact that $ | \nabla |z||_2^2-| \nabla z|_2^2=0$ one obtains
\begin{align}\label{h92}
\int_{\mathbb{R}^3}\sum_{i=1}^3 \frac{(v \partial_i w-w\partial_i v)^2}{v^2+w^2}dx =0.
\end{align}
Hence, from \cite[Theorem 4.1]{2004-ANS-Ha}, we know that
 \begin{itemize}
   \item [$(i)$] either $v\equiv 0$ or $v(x) \neq   0$ for all $x \in \mathbb{R}^3$;
   \item [$(ii)$] either $w\equiv 0$ or $w(x) \neq   0$ for all $x \in \mathbb{R}^3$.
 \end{itemize}
  let   $u:= |z|$, then we have $J(u)=m(a)$ and $u\in D_{\rho_0}$.   If $w\equiv 0$, then   we deduce that $u:= |z|=|v|>0$ on $\mathbb{R}^3$  and $z=   e^{i \theta} |u| $ where  $\theta =0$ if $v>0$ and   $\theta =\pi$ if $v<0$ on $\mathbb{R}^3$.  Otherwise, it follows from $(ii)$  that $w(x)\neq 0$ for all $x\in\mathbb{R}^3$. Since
\begin{align*}
  \frac{(v \partial_i w-w\partial_i v)^2}{v^2+w^2}=\Big[\partial_i\big (\frac{v}{w}\big)  \Big ]^2 \frac{w^2}{v^2+w^2}\quad\quad \text{where}\ i=1,2, 3,
\end{align*}
for all $x\in \mathbb{R}^3$,    we get from \eqref{h92} that $\nabla (\frac{v}{w})=0$ on $\mathbb{R}^3$.
Therefore, there exists $ C\in \mathbb{R}$ such that $v=C w$ on  $\mathbb{R}^3$. Then we have
\begin{align}\label{h93}
z=(v,w)=v+iw=(C+i)w \quad\quad\text{and}\quad\quad |z|= |C+i||w|.
\end{align}
Let $\theta_1\in \mathbb{R}$ be such that $C+i= |C+i| e^{i \theta_1}$ and let $w=|w| e^{i \theta_2} $ with
\begin{align*} \theta_2=
\begin{cases}
0, \quad\quad &\text{if}\ w>0;\\
\pi, \quad\quad &\text{if}\ w<0.
\end{cases}
\end{align*}
 Then we can see from \eqref{h93} that $ z=(C+i)w=|C+i| |w|e^{i (\theta_1+ \theta_2)} = |z|e^{i (\theta_1+ \theta_2)}$. Setting $\theta=\theta_1+\theta_2$ and $u:=|z|$, then $|u|>0 $ and $z= e^{i \theta }|u|$.
Thus,
$$
\mathcal{M}_a\subseteq \big\{ e^{i \theta}|u|:  \ \theta\in \mathbb{R}, \ |u|\in D_{\rho_0},\ \mathcal{J}(|u|)=m(a)\ \text{and}\   |u| >0 \ \text{in}\ \mathbb{R}^3\big\}.
$$
On the other hand, let $z= e^{i \theta}|u|$ with  $\theta\in \mathbb{R}$, $|u|\in D_{\rho_0}$,  $\mathcal{J}(|u|)=m(a)$    and $ |u|>0$ in $\mathbb{R}^3$, then we have $|z|_2^2=a^2$, $|\nabla z|_2^2=|\nabla |u||_2^2 \leq \rho_0^2$  and $\mathcal{J}(z)=\mathcal{J}(u)= m(a)$,
 which implies that
 $$
  \big\{ e^{i \theta}|u|:  \ \theta\in \mathbb{R}, \ |u|\in D_{\rho_0},\ \mathcal{J}(|u|)=m(a)\ \text{and}\   |u| >0 \ \text{in}\ \mathbb{R}^3\big\}\subseteq  \mathcal{M}_a.
 $$
Hence, \eqref{k94} follows.

 Next,  we prove that $  \mathcal{M}_a$  is orbitally stable. Under the assumptions of Theorem \ref{K-TH3}, it is known that in this case, the Cauchy problem associated with \eqref{kk1} is locally  well posed in $\mathcal{H}$, see \cite{2003-C}.
That is, let $ \varphi(t,x)$  be  the unique solution of  the  initial-value problem \eqref{kk1}
with initial datum $u_0$ on $(-T_{min}, T_{max})$,  it holds that
\begin{align}\label{kkk80}
|u_0|_2=| \varphi|_2\quad\quad \text{and} \quad \quad \mathcal{J}(u_0)=\mathcal{J}(\varphi),
\end{align}
and   either $T_{max} =+ \infty$  or if $T_{max} <+ \infty$,  $| \nabla  \varphi|_2 \rightarrow +\infty$ as $t \rightarrow T^-_{max}$.
 We assume by contradiction  that,  let $v \in \mathcal{M}_a$, there exists $\delta_n\in \mathbb{R}^+$ a decreasing sequence converging to $0$ and  $\{\varphi_n\}\in \mathcal{H}$ satisfying
$$
\|\varphi_n-v \|\leq \delta_n,
$$
but
$$
\sup_{t\in[0, T_{max})}dist( u_{\varphi_n}(t),\mathcal{M}_a)>\varepsilon_0>0,
$$
where $u_{\varphi_n}(t)$ denotes the solution to \eqref{kk1} with initial datum $ \varphi_n$.
We observe that $|\varphi_n|_2\rightarrow |v|_2=a$ as $n\rightarrow\infty$ and   $\mathcal{J}(v)$,  $\mathcal{J}(\varphi_n)\rightarrow m(a)$ as $n\rightarrow\infty$ by continuity.     From the  conservation
laws \eqref{kkk80}, for $n \in \mathbb{N}$ large enough, $u_{\varphi_n }$ will remains inside of $D_{\rho_0}$ for all $t\in [0, T_{max}]$. Indeed, if for some time
$t > 0$, $|\nabla u_{\varphi_n}(t)|_2=\rho_0$,
  then, in view of Lemma \ref{K-Lem2.5} we have that $\mathcal{J}( u_{\varphi_n})\geq0$, which is  contradict  with $ m(a)<0$. Hence, $T_{max}=+\infty$.  This shows that  solutions starting
in  $D_{\rho_0}$ are globally defined in time. That is,
  $u_{\varphi_n}(t)$ is globally defined.
   Now let $t_n > 0$ be the first time such that $dist( u_{\varphi_n}(t_n),\mathcal{M}_a)=\varepsilon_0$ and set $u_n := u_{\varphi_n}(t_n)$. By
conservation laws \eqref{kkk80}, $\{u_n\}  \subset A_{\rho_0}$ satisfies  $|u_n|_2\rightarrow a$ as $n\rightarrow\infty$
 and $\mathcal{J}(u_n)= \mathcal{J}(\varphi_n)\rightarrow m(a)$ and thus, in view of Lemma \ref{K-Lem2.12}, it converges, up to translation, to an element of  $\mathcal{M}_a$. Since  $\mathcal{M}_a$ is invariant under translation this contradicts
the equality $dist(u_n, \mathcal{M}_a)= \varepsilon_0 > 0$. Thus, we complete the proof. $\hfill\Box$

\subsection{Strong instability of the standing wave $e^{i \lambda  t} u $}
In this part, we prove that the standing wave $e^{i \lambda  t} u $  of
 Eq.~\eqref{k1} with $\lambda>0$ and $u\in \mathcal{H}_r$,    obtained in  Theorem \ref{K-TH2} (or Theorem \ref{K-TH5} or Theorem \ref{K-TH15}), is strongly unstable. In what follows, we  present the   result of finite time blow-up.
The   proof  of the finite time blow-up relies  on the classical  convexity method of Glassey \cite{1977-Glassey}, which was further   refined in  Berestycki and  Cazenave \cite{1981CRASSM-BC}.

\begin{lemma}\label{K-Lem5.1}
 Under the assumptions of Theorem \ref{K-TH2} (or Theorem \ref{K-TH5} or Theorem \ref{K-TH15}), let $u_0 \in S_{a,r}$ be such that $\mathcal{J}(u_0)<\inf_{u\in \mathcal{P}_{-,r}}$, and if $|x|u_0\in L^2(\mathbb{R}^3)$
  and $t_{u_0} < 0$, where $t_{u_0}$
is defined in Lemma \ref{K-Lem2.3} (or Lemma \ref{K-Lem4.2} or Lemma \ref{K-J1}), then the solution $\varphi$ of problem \eqref{kk1} with initial datum $u_0$ blows-up in finite time.
\end{lemma}

\begin{proof}
  Under the assumptions of Theorem \ref{K-TH2} (or Theorem \ref{K-TH5} or Theorem \ref{K-TH15}),   let $u_0 \in S_a$ be such that $\mathcal{J}(u_0)<\inf_{u\in \mathcal{P}_{-,r}}$.
Since  the  initial-value problem \eqref{kk1} with initial datum $u_0$  is local well-posed on $(-T_{min}, T_{max})$ with $T_{min}, T_{max}>0 $,
that is, let $ \varphi(t,x)$  be  the solution of  the  initial-value problem \eqref{kk1}
with initial datum $u_0$ on $(-T_{min}, T_{max})$,  it holds that
\begin{align}\label{kk80}
|u_0|_2=| \varphi|_2\quad\quad \text{and} \quad \quad \mathcal{J}(u_0)=\mathcal{J}(\varphi),
\end{align}
and    if $T_{max} <+ \infty$,  $| \nabla  \varphi|_2 \rightarrow +\infty$ as $t \rightarrow T^-_{max}$.
By  $|x|u_0 \in L^2(\mathbb{R}^3 )$  and \cite[Proposition 6.5.1]{2003-C}, we get
\begin{align} \label{kk81}
H(t):=\int_{\mathbb{R}^3}|x|^2|\varphi(t,x)|^2dx<+\infty  \ \ \ \ \ \text{for all}\ t\in(-T_{min}, T_{max} ).
  \end{align}
Moreover, the function $H\in C^2(-T_{min}, T_{max} )$ and the following  Virial identity holds: $H'(t)= 4\Im  \int_{\mathbb{R}^3}\overline{\varphi} (x\cdot \nabla\varphi )dx $ and
\begin{align}\label{k82}
\ H''(t)&= 8 P(\varphi)\nonumber\\
&=8\int_{\mathbb{R}^3}|\nabla \varphi|^2dx+2 \int_{\mathbb{R}^3} \int_{\mathbb{R}^3}  \frac{|\varphi(t,x)|^2|\varphi(t,y)|^2 }{|x-y|} dxdy -8 \gamma_p\int_{\mathbb{R}^3}|\varphi|^{p}dx-8\mu \gamma_q \int_{\mathbb{R}^3}|\varphi|^{q}dx.
  \end{align}
From Lemma \ref{K-Lem2.3} (or   Lemma \ref{K-Lem4.2} or Lemma \ref{K-J1}), for any $u\in S_{a,r}$, the function $\widetilde{\psi}_u(s):=\psi_u(\log s)$ with $s>0$ has a   unique global maximum point $\widehat{t}_u=e^{t_u}$ and $\widetilde{\psi}_u(s)$ is strictly decreasing and concave on $(\widehat{t}_u, +\infty)$. According to the assumption $t_{u_0}<0$,   one obtains $ \widehat{t}_{u_0}\in(0,1)$. We claim that
\begin{align}\label{k85}
\text{if}\ u \in S_{a,r} \ \text{and}\ \widehat{t}_u\in(0,1), \ \text{then}\ P(u)\leq \mathcal{J}(u)-\inf_{\mathcal{P}_{-, r}}\mathcal{J}.
\end{align}
In fact, since $ \widehat{t}_u\in(0,1)$ and $\widetilde{\psi}_u(s)$ is strictly decreasing and concave on $(\widehat{t}_u, +\infty)$, we infer that $t_u<0$, $P(u)<0$ and
\begin{align*}
\mathcal{J}(u)=\psi_u(0) = \widetilde{\psi}_u(1)\geq  \widetilde{\psi}_u(\widehat{t}_u )-\widetilde{\psi}_u'(1)(\widehat{t}_u-1 )
=\mathcal{J}(t_u \star u) -|P(u)|(1-\widehat{t}_u )
\geq \inf_{\mathcal{P}_{-,r}}\mathcal{J} +P(u),
\end{align*}
which completes the claim.
Now, let us consider the solution $\varphi$  for the  initial-value problem \eqref{kk1} with initial datum $u_0$. Since by assumption $t_{u_0} < 0$, and the map $u \mapsto t_u$ is continuous, we deduce that  $t_{ \varphi(t)} < 0$   for every $|t|<\overline{t}$ with $\overline{t}>0$ small enough. Then $\widehat{t}_{ \varphi(t)}\in (0,1)$ for $|t|<\overline{t}$. By \eqref{kk80}, \eqref{k85}  and recalling the assumption     $\mathcal{J}(u_0)< \inf_{\mathcal{P}_{-,r}}\mathcal{J} $, we deduce   that
\begin{align*}
P(\varphi(t))\leq \mathcal{J}(\varphi(t) )-\inf_{\mathcal{P}_{-,r}}\mathcal{J}=\mathcal{J}(u_0 )-\inf_{\mathcal{P}_{-,r}}\mathcal{J}:=-\eta<0  \quad \ \ \text{for all}\ |t|< \overline{t}.
\end{align*}
 Next, we show that
\begin{align}\label{k86}
  P(\varphi(t) )\leq-\eta \quad\quad \text{ for any }\ t\in (-T_{min}, T_{max}).
\end{align}
 Indeed, assume that there is $t_0\in (-T_{min}, T_{max})$ satisfying $ P(\varphi(t_0) )=0$. It holds from \eqref{kk80} that $\varphi(t_0)\in S_{a,r}$  and
$$
 \mathcal{J}(\varphi(t_0))\geq c_{a}=\inf_{\mathcal{P}_{-,r}}\mathcal{J}> \mathcal{J}(u_0)=  \mathcal{J}(\varphi(t_0)),
$$
which is a contradiction.  This shows that    $ P(\varphi(t_0) )\neq0$ for any  $t_0\in (-T_{min}, T_{max})$. Then, since $P(\varphi(t))<0$ for every   $|t|< \overline{t}$, we obtain $ P(\varphi(t) )<0$ for any  $t\in (-T_{min}, T_{max})$ (if at some $t\in  (-T_{min}, T_{max})$, $P(\varphi(t))>0$. By continuity, we have $P(\varphi(\widetilde{t}))=0$ for some $\widetilde{t}\in  (-T_{min}, T_{max})$, a contradiction). Since $ P(\varphi(t) )<0$ for any  $t\in (-T_{min}, T_{max})$, it follows from   Lemma \ref{K-Lem2.3}  (or   Lemma \ref{K-Lem4.2} or Lemma \ref{K-J1}) that $t_{\varphi  }<0$. That is, $\widehat{ t}_{\varphi  }\in (0,1)$. Thus, \eqref{k85} holds and   the above arguments yield
\begin{align*}
  P(\varphi(t) )\leq-\eta \quad\quad \text{ for any }\ t\in (-T_{min}, T_{max}).
\end{align*}
Thus, by  \eqref{k82} we get that $ H$ is a concave function. Then,    from \eqref{kk81}, \eqref{k82} and \eqref{k86}, we deduce
\begin{align*}
0\leq H(t)\leq H(0)+H'(0) t+\frac{1}{2}H''(0) t^2\leq H(0)+H'(0) t-4 \eta t^2 \quad\quad \text{ for any }\ t\in (-T_{min}, T_{max}),
\end{align*}
which implies that $T_{max}<+\infty$.  Thus, the    solution $\varphi$  of the initial-value problem  \eqref{kk1}
with initial datum $u_0$   blows-up in finite time.
\end{proof}

 \noindent{\bf{Proof of Theorem \ref{K-TH4}}.} Under the assumption of Theorem \ref{K-TH2} (or Theorem \ref{K-TH5} or Theorem \ref{K-TH15}), $e^{i\widehat{\lambda} t} \widehat{u}$ is a standing wave of problem \eqref{kk1}, obtained in  Theorem \ref{K-TH2}  (or Theorem \ref{K-TH5} or Theorem \ref{K-TH15}), where $ \widehat{\lambda}>0$ and $ \widehat{u}\in \mathcal{H}_{r}$.
For any $\varrho> 0$, let $u_{\varrho} := \varrho\star \widehat{u}$,  and let  $\varphi_{\varrho}$ be the solution to \eqref{kk1} with initial datum $u_{\varrho} $. One has $u_{\varrho} \rightarrow  \widehat{u }$ as $\varrho\rightarrow 0^+$. Hence it is sufficient to prove that $\varphi_{\varrho}$ blows-up in finite time.
Clearly $t_{u_{\varrho}} = -\varrho < 0$, where $t_{u_{\varrho}}$ is defined in Lemma \ref{K-Lem2.3} (or Lemma \ref{K-Lem4.2} or Lemma \ref{K-J1}).
By definition
$$
\mathcal{J}(u_{\varrho})=\mathcal{J}(\varrho\star \widehat{u})< \mathcal{J}( \widehat{u})=\inf_{ \mathcal{P}_{-,r}}\mathcal{J}( u).
$$
Moreover, since $ \widehat{\lambda}>0$
  and $ \widehat{u}\in \mathcal{H}_{r}$, we have that $ \widehat{u}$ decays exponentially at infinity (see \cite{1983-BL}),
and hence $|x|u_{\varrho} \in L^2(\mathbb{R}^3)$. Therefore, by Lemma \ref{K-Lem5.1} the solution  $\varphi_{\varrho}$ blows-up in finite time.    Hence, $e^{i\widehat{\lambda} t}\widehat{u}$ is strongly unstable from  Definition \ref{de1.2}. Thus, we complete the proof. $\hfill\Box$


\begin{thebibliography}{25}
\setlength{\itemsep}{-2.00pt}

\bibitem{2008MJM-A}
  \newblock   Ambrosetti, A.:
  \newblock  On Schr\"{o}dinger-Poisson systems.
  \newblock  Milan J. Math. \textbf{76},  257--274 (2008)

\bibitem{2008CCM-AR}
  \newblock Ambrosetti,  A.,  Ruiz,  D.:
  \newblock  Multiple bound states for the Schr\"{o}dinger-Poisson problem.
  \newblock  Commun. Contemp. Math. \textbf{10},  391--404 (2008)



\bibitem{2022CValves}
  \newblock  Alves, C.O.,  Ji, C., Miyagaki,   O.H.:
 \newblock  Normalized solutions for a Schr\"{o}dinger equation with critical growth in $\mathbb{R}^N$.
 \newblock Calc. Var. Partial Differ. Equ.   \textbf{61}, 18 (2022)

\bibitem{2012DIE}
\newblock    Akahori, T.,   Ibrahim, S.,     Kikuchi H.,     Nawa, H.:
 \newblock Existence of a ground state and blow-up problem for a
nonlinear Schr\"{o}dinger equation with critical growth.
\newblock  Differential Integral Equations   \textbf{25},  383--402  (2012)

\bibitem{2010AIHP-ADP}
  \newblock  Azzollini, A.,  d'Avenia, P.,  Pomponio, A.:
   \newblock  On the Schr\"{o}dinger-Maxwell equations under the effect of a general nonlinear
term.
\newblock  Ann. Inst. H. Poincar\'{e} Anal. Non Lin\'{e}aire  \textbf{27}, 779--791 (2010)

\bibitem{2008JMAA-AP}
  \newblock  Azzollini, A., Pomponio,  A.:
   \newblock  Ground state solutions for the nonlinear Schr\"{o}dinger-Maxwell equations.
   \newblock  J. Math. Anal. Appl. \textbf{345},  90--108 (2008)




\bibitem{2002RMP-BF}
 \newblock  Benci, V.,   Fortunato, D.:
  \newblock Solitary waves of the nonlinear Klein-Gordon equation coupled with Maxwell equations.
   \newblock Rev. Math. Phys.  \textbf{14}, 409--420   (2002)


\bibitem{1981CMP-BBL}
  \newblock Benguria, R.,  Brezis, H.,   Lieb, E.H.:
    \newblock The Thomas-Fermi-von Weizs\"{a}cker theory of atoms and molecules.
     \newblock  Comm. Math. Phys. \textbf{79},   167--180 (1981)


\bibitem{1981CRASSM-BC}
\newblock   Berestycki, H.,  Cazenave, T.:
\newblock  Instabilit\'{e} des \'{e}tats stationnaires dans les \'{e}quations de Schr\"{o}dinger et de Klein-Gordon non lin\'{e}aires,
\newblock  C. R. Acad. Sci., S\'{e}r. 1 Math. \textbf{293},   489--492  (1981)



\bibitem{1983-BL}
\newblock  Berestycki, H.,   Lions, P.L.:
 \newblock Nonlinear scalar field equations I, existence of a ground state.
 \newblock Arch. Ration. Mech. Anal.   \textbf{82},  313--345 (1983)


\bibitem{1983-BL-II}
\newblock  Berestycki, H.,   Lions, P.L.:
\newblock Nonlinear scalar field equations. II. Existence of infinitely many solutions.
\newblock Arch. Ration. Mech. Anal. \textbf{82},  347--375 (1983)



%

\bibitem{2011ZAMP-BS}
\newblock  Bellazzini, J.,  Siciliano,  G.:
\newblock Stable standing waves for a class of nonlinear Schr\"{o}dinger-Poisson
equations.
\newblock Z. Angew. Math. Phys. \textbf{62},  267--280 (2011).

\bibitem{2011JFA-BS}
\newblock  Bellazzini, J.,  Siciliano,  G.:
\newblock  Scaling properties of functionals and existence of constrained minimizers.
\newblock J. Funct. Anal. \textbf{261}, 2486--2507 (2011).


\bibitem{2013PLMS}
\newblock  Bellazzini, J.,  Jeanjean, L.,   Luo, T.:
\newblock Existence and instability of standing waves with prescribed norm for a class of
Schr\"{o}dinger-Poisson equations.
\newblock Proc. Lond. Math. Soc. \textbf{107},  303--339 (2013)







\bibitem{2003-C}
\newblock  Cazenave, T.: Semilinear Schr\"{o}dinger Equations, Courant Lecture Notes in Mathematics, vol. 10, New York University, Courant Institute of Mathematical Sciences/American Mathematical Society, New York/Providence, RI, (2003)


\bibitem{1982-CMP-CL}
\newblock  Cazenave, T.,  Lions, P.L.:
\newblock Orbital stability of standing waves for some nonlinear Schr\"{o}dinger equations.
\newblock Commun. Math. Phys. \textbf{85},   549--561  (1982)




\bibitem{1992CPDE-CL}
\newblock  Catto, I.,  Lions,  P.L.:
\newblock Binding of atoms and stability of molecules in Hartree and Thomas-Fermi
type theories. I. A necessary and sufficient condition for the stability of general molecular systems.
\newblock Comm. Partial Differential Equations \textbf{17},  1051--1110 (1992)

\bibitem{2016N-CM}
\newblock  Cerami, G.,  Molle, R.:
 \newblock Positive bound state solutions for some Schr\"{o}dinger-Poisson systems.
 \newblock Nonlinearity \textbf{29},  3103 (2016)


\bibitem{2009-NA-CT}
\newblock Chen,  S.J.,  Tang, C.L.:
\newblock High energy solutions for the superlinear Schr\"{o}dinger-Maxwell
equations.
 \newblock Nonlinear Analysis \textbf{71},  4927--4934 (2009)





\bibitem{2020JMAA-CT}
\newblock  Chen, S.T.,  Tang, X.H., Yuan, S.:
 \newblock Normalized solutions for Schr\"{o}dinger-Poisson equations with general nonlinearities.
\newblock J. Math. Anal. Appl. \textbf{481},   123447 (2020)



\bibitem{2022.9}
\newblock  Chen, S.T.,  Tang, X.H.:
\newblock New approaches for Schr\"{o}dinger equations with prescribed mass: The Sobolev subcritical case and
The Sobolev critical case with mixed dispersion.
\newblock (2022) arXiv: 2210.14503v1.


\bibitem{2016JDECheng}
\newblock   Cheng, X.,  Miao, C.,    Zhao, L.:
\newblock Global well-posedness and scattering for nonlinear Schr\"{o}dinger equations with combined nonlinearities inthe radial case.
\newblock  J. Differential Equations \textbf{261},   2881--2934 (2016)

\bibitem{1982-PRSEA-E-L}
\newblock  Esteban, M., Lions,  P.L.:
\newblock  Existence and nonexistence results for semilinear elliptic problems in unbounded domains.
\newblock  Proc. R. Soc. Edinb., Sect. A      \textbf{93}, 1--14  (1982/83)




\bibitem{2018JEE}
\newblock   Feng, B.H.:
\newblock On the blow-up solutions for the nonlinear Schr\"{o}dinger equation with combined power-type nonlinearities.
\newblock J. Evol. Equ.   \textbf{ 18},   203--220  (2018)

 \bibitem{1977-Glassey}
 \newblock Glassey,  R.T.:
 \newblock On the blowing up of solution to the Cauchy problem for nonlinear Schr\"{o}dinger operators.
 \newblock  J. Math. Phys. \textbf{8},  1794--1797 (1977)


%
%


\bibitem{2004-ANS-Ha}
  \newblock  Hajaiej, H.,   Stuart, C.A.:
  \newblock  On the variational approach to the stability of
standing waves for the nonlinear Schr\"{o}dinger equation.
  \newblock  Adv. Nonlinear Stud.   \textbf{ 4},  469--501 (2004)


%
%
%

  \bibitem{2022-JMPA-jean}
  \newblock  Jeanjean, L.,   Jendrej,  J.,   Le,  T.T.,  Visciglia,   N.:
  \newblock Orbital stability of ground states for a Sobolev critical
Schr\"{o}dinger equation.
  \newblock J. Math. Pures Appl.  \textbf{164}, 158--179  (2022)


  \bibitem{2021-JDE-JL}
  \newblock  Jeanjean, L.,   Le T.T.:
  \newblock Multiple normalized solutions for a Sobolev critical Schr\"{o}dinger-Poisson-Slater equation.
  \newblock J. Differential Equations \textbf{303}, 277--325  (2021)

\bibitem{2013ZAMP-JL}
\newblock  Jeanjean, L.,  Luo, T.:
\newblock Sharp nonexistence results of prescribed $L^2$-norm solutions for some class of Schr\"{o}dinger-Poisson and quasi-linear equations,
\newblock  Z. Angew. Math. Phys. \textbf{64}, 937--954 (2013)




  \bibitem{2022-MA-jean}
  \newblock  Jeanjean, L.,   Le, T.T.:
   \newblock Multiple normalized solutions for a Sobolev critical Schr\"{o}dinger equation.
    \newblock Math. Ann.    \textbf{384},  101--134  (2022)

\bibitem{2023-JGA-KLT}
\newblock  Kang. J.C., Liu, X.Q., Tang, C.L.:
\newblock  Ground state sign-changing solutions for critical Schr\"{o}dinger-Poisson system with steep potential well.
\newblock  J. Geom. Anal.   \textbf{33}, 59 (2023)



\bibitem{2017ARMA}
\newblock   Killip, R.,   Oh, T.,  Pocovnicu, O.,   Visan,  M.:
 \newblock Solitons and scattering for the cubic-quintic nonlinear
Schr\"{o}dinger equation on $\mathbb{R}^3$.
\newblock Arch. Rational Mech. Anal.  \textbf{225},   469--548 (2017)

%




\bibitem{2016RMI}
\newblock  Le Coz, S.,     Martel, Y.,    Rapha\"{e}l, P.:
 \newblock Minimal mass blow up solutions for a double power nonlinear
Schr\"{o}dinger equation.
\newblock Rev. Mat. Iberoam.\textbf{32},    795--833 (2016)

\bibitem{2014-JFA-le}
\newblock  Lehrer,  R., Maia,  L.A.:
\newblock  Positive solutions of asymptotically linear equations via Pohozaev manifold.
\newblock  J. Funct. Anal. \textbf{266},  213--246  (2014)



\bibitem{2021CVLixinfu}
\newblock   Li, X.F.:
\newblock Existence of normalized ground states for the Sobolev critical
Schr\"{o}dinger equation with combined nonlinearities.
\newblock Calc. Var. Partial Differ. Equ. \textbf{60},  169 (2021)



\bibitem{1981RMP-L}
\newblock  Lieb, E.H.:
 \newblock Thomas-Fermi and related theories and molecules.
 \newblock Rev. Modern Phys. \textbf{53},  603--641 (1981)

\bibitem{1984CMP-L}
 \newblock   Lions, P.L.:
 \newblock  Solutions of Hartree-Fock equations for Coulomb systems.
   \newblock  Comm. Math. Phys. \textbf{109}, 33--97  (1984)

\bibitem{2016-AMPA-LWZ}
  \newblock Liu, Z.L., Wang, Z.Q., Zhang, J.J.:
   \newblock Infinitely many sign-changing solutions for the nonlinear
Schr\"{o}dinger-Poisson system.
 \newblock  Ann. Mat. Pura Appl. \textbf{195}, 775--794 (2016)





\bibitem{2014JMAA-L}
\newblock  Luo, T.J.:
\newblock Multiplicity of normalized solutions for a class of nonlinear Schr\"{o}dinger-Poisson-Slater equations.
\newblock J. Math. Anal. Appl. \textbf{416}, 195--204  (2014)



\bibitem{2013CMP}
\newblock  Miao, C.,   Xu,  G.,     Zhao, L.:
\newblock The dynamics of the 3D radial NLS with the combined terms.
\newblock Commun. Math. Phys.  \textbf{318},   767--808 (2013)


\bibitem{2017CV}
\newblock    Miao, C.,  Zhao, T.,   Zheng, J.:
 \newblock On the 4D nonlinear Schr\"{o}dinger equation with combined terms under the
energy threshold.
\newblock Calc. Var. Partial Differ. Equ.  \textbf{56},    179 (2017)


\bibitem{2023-JDE-Qi}
\newblock   Qi, S.J.,   Zou, W.M.:
\newblock Mass threshold of the limit behavior of normalized
solutions to Schr\"{o}dinger equations with combined
nonlinearities.
\newblock J. Differential Equations  \textbf{375},  172--205 (2023)





%

\bibitem{2006JFA}
\newblock   Ruiz,  D.:
\newblock     The Schr\"{o}dinger-Poisson equation under the effect of a nonlinear local term.
\newblock    J. Funct. Anal. \textbf{237},  655--674 (2006)


\bibitem{2004JSP}
\newblock  Sanchez,  O.,  Soler, J.:
\newblock  Long-time dynamics of the Schr\"{o}dinger-Poisson-Slater system.
 \newblock  J. Statist. Phys. \textbf{114}, 179--204  (2004)





\bibitem{2023IJM}
 \newblock  Siciliano, G.,   Silva, K.:
   \newblock   On the structure of the Nehari set associated to a Schr\"{o}dinger-Poisson system with prescribed mass: old and new results. Israel J. Math. (2023), DOI: 10.1007/s11856-023-2477-9.





%
%

\bibitem{2020Soave}
 \newblock  Soave, N.:
  \newblock Normalized ground state for the NLS equations with combined nonlinearities.
  \newblock   J. Differential Equations  \textbf{269},  6941--6987 (2020)

%
\bibitem{2021-JFA-Soave}
\newblock Soave, N.:
  \newblock Normalized ground states for the NLS equation with combined nonlinearities: The Sobolev critical case.
  \newblock  J. Funct. Anal.  \textbf{279}, 108610  (2020)

\bibitem{2016-JDE-SM}
\newblock Sun. J.J., Ma, S.W.:
\newblock  Ground state solutions for some Schr\"{o}dinger¨CPoisson systems with periodic potentials.
\newblock  J. Differential Equations  \textbf{260},  2119--2149  (2016)





\bibitem{2007CPDETao}
\newblock  Tao, T.,   Visan, M., Zhang,  X.:
\newblock The nonlinear Schr\"{o}dinger equation with combined power type nonlinearities.
\newblock  Commun. Partial Differ. Equ.  \textbf{32}, 1281--1343 (2007)





\bibitem{2023AMP-WQ}
\newblock  Wang, Q.,   Qian, A.X.:
\newblock Normalized solutions to the Schr\"{o}dinger-Poisson-Slater equation with general nonlinearity: mass supercritical case.
\newblock Anal. Math. Phys. \textbf{13},  35 (2023)


%
%
\bibitem{2022JFAWei}
\newblock  Wei, J.C.,   Wu, Y.Z.:
\newblock Normalized solutions for Schr\"{o}dinger equations with
critical Sobolev exponent and mixed nonlinearities.
\newblock J. Funct. Anal.  \textbf{283},  109574 (2022)


\bibitem{1983W}
\newblock  Weinstein, M.I.:
 \newblock Nonlinear Schr\"{o}dinger equations and sharp interpolation estimates.
 \newblock Comm. Math. Phys.  \textbf{87},   567--576  (1983)

\bibitem{1983Wi}
 \newblock   Willem, M.:
  \newblock  Minimax Theorems, vol. 24, Birkh\"{o}user, Boston, Mass., (1996)


\bibitem{2020MMAS}
\newblock  Xie, W.H.,   Chen, H.B.,   Shi, H.X.:
\newblock Existence and multiplicity of normalized solutions for a class of Schr\"{o}dinger-Poisson equations with general nonlinearities.
\newblock Math. Methods Appl. Sci. \textbf{43},  3602--3616 (2020)


\bibitem{2023.3}
\newblock  Yao, S.,  Hajaiej, H.,  Sun, J.T., Wu,  T.F.:
\newblock Standing waves for the NLS equation with competing
nonlocal and local nonlinearities: the double
$L^2$-supercritical case.
\newblock  (2023) arXiv:2102.10268v2.




%

\bibitem{2017CMA-Y}
\newblock  Ye, H.Y.:
\newblock The existence and the concentration behavior of normalized solutions for the $L^2$-critical Schr\"{o}dinger-Poisson system.
\newblock Comput. Math. Appl. \textbf{74},    266--280 (2017)


\bibitem{2018ZAMP-YL}
\newblock  Ye, H.Y.,    Luo, T.J.:
\newblock On the mass concentration of $L^2$-constrained minimizers for a class of Schr\"{o}dinger-Poisson equations.
\newblock Z. Angew. Math. Phys.   \textbf{69},  66 (2018)

\bibitem{2017JMAA-ZZ}
\newblock Zeng,  X.Y.,  Zhang, L.:
\newblock Normalized solutions for Schr\"{o}dinger-Poisson-Slater equations with unbounded potentials.
\newblock J. Math. Anal. Appl. \textbf{452}, 47--61  (2017)



\bibitem{2008-JMAA-ZHAO}
\newblock  Zhao, L.G., Zhao, F.K.:
\newblock  On the existence of solutions for the Schr\"{o}dinger-Poisson equations.
\newblock   J. Math. Anal. Appl.  \textbf{346},    155--169 (2008).


\bibitem{2018-NARWA-ZT}
\newblock  Zhong, X.J., Tang, C.L.:
\newblock Ground state sign-changing solutions for a Schr\"{o}dinger-Poisson system with a critical nonlinearity in $\mathbb{R}^3$.
\newblock Nonlinear Anal. Real World Appl. \textbf{39}, 166--184 (2018)





\end{thebibliography}

\end{document}